\documentclass[11pt,reqno]{amsart}
\usepackage{times}
\usepackage{amsmath,amsfonts, amstext,amssymb,amsbsy,amsopn,amsthm}
\usepackage{color}
\usepackage{dsfont}
\usepackage{esint}
\usepackage{graphicx}   
\usepackage{hyperref}
\usepackage[all]{xy}
\usepackage{bm}

\newcommand{\dR}{\mathds{R}}

\DeclareMathOperator{\Conj}{ConjRad}
\DeclareMathOperator{\diam}{diam}
\DeclareMathOperator{\II}{II}
\DeclareMathOperator{\id}{Id}
\DeclareMathOperator{\Image}{Image}
\DeclareMathOperator{\InjRad}{InjRad}
\DeclareMathOperator{\Isom}{Isom}
\DeclareMathOperator{\length}{length}
\DeclareMathOperator{\Lip}{Lip}

\DeclareMathOperator{\N}{N}

\DeclareMathOperator{\pr}{pr}
\DeclareMathOperator{\R}{R}
\DeclareMathOperator{\rank}{rank}
\DeclareMathOperator{\Ric}{Ric}
\DeclareMathOperator{\Rm}{Rm}

\DeclareMathOperator{\Step}{Step}
\DeclareMathOperator{\Supp}{Supp}
\DeclareMathOperator{\Tor}{Tor}

\DeclareMathOperator{\Vol}{Vol}

\newtheorem{theorem}{Theorem}[section]

\newtheorem{lemma}[theorem]{Lemma}
\newtheorem{proposition}[theorem]{Proposition}
\newtheorem{corollary}[theorem]{Corollary}
\theoremstyle{definition}
\newtheorem{definition}[theorem]{Definition}
\theoremstyle{remark}
\newtheorem{remark}{Remark}[section]

\newtheorem{example}{Example}[section]

\theoremstyle{remark}

\numberwithin{equation}{section}

\begin{document}

\title{Topology and $\epsilon$-regularity Theorems on Collapsed Manifolds with Ricci Curvature Bounds}
\date{\today}
\author{Aaron Naber}
\address{Department of Mathematics, Northwestern, Evanston, Il, USA}
\email{anaber@math.northwestern.edu}
\author{Ruobing Zhang}
\address{Department of Mathematics, Princeton University, Princeton, NJ 08544, USA}
\email{ruobingz@math.princeton.edu}

\maketitle

\begin{abstract}
In this paper we discuss and prove $\epsilon$-regularity theorems for Einstein manifolds $(M^n,g)$, and more generally manifolds with just bounded Ricci curvature, in the collapsed setting.  

A key tool in the regularity theory of noncollapsed Einstein manifolds is the following:  If $x\in M^n$ is such that $\Vol(B_1(x))>v>0$ and that $B_2(x)$ is sufficiently Gromov-Hausdorff close to a cone space $B_2(0^{n-\ell},y^*)\subset \dR^{n-\ell}\times C(Y^{\ell-1})$ for $\ell\leq 3$, then in fact $|\Rm|\leq 1$ on $B_1(x)$.  No such results are known in the collapsed setting, and in fact it is easy to see without more such results are false.  It turns out that the failure of such an estimate is related to topology.  Our main theorem is that for the above setting in the collapsed context, either the curvature is bounded, or there are topological constraints on $B_1(x)$.

More precisely, using established techniques one can see there exists $\epsilon(n)$ such that if $(M^n,g)$ is an Einstein manifold and $B_2(x)$ is $\epsilon$-Gromov-Hausdorff close to ball in $B_2(0^{k-\ell},z^*)\subset\dR^{k-\ell}\times Z^\ell$, then the fibered fundamental group $\Gamma_\epsilon(x)\equiv \Image[\pi_1(B_\epsilon(x))\to\pi_1(B_2(x))]$ is almost nilpotent with $\rank(\Gamma_{\epsilon}(x))\leq n-k$.  The main result of the this paper states that if $\rank(\Gamma_{\epsilon}(x))= n-k$ is maximal, then $|\Rm|\leq C$ on $B_1(x)$.  In the case when the ball is close to Euclidean, this is both a necessary and sufficient condition.  There are generalizations of this result to bounded Ricci curvature and even just lower Ricci curvature.

\end{abstract}

\tableofcontents

\section{Introduction}

A classical theme in Einstein manifolds, and indeed any nonlinear equation, is that of an $\epsilon$-regularity theorem.  The spirit of any such theorem is that one assumes some weak control on the solution, and then proves from this strong control on the solution.  The original example of such a theorem in the case of an Einstein manifold $(M^n,g)$ is to assume on some ball $B_2(x)$ that
\begin{align}
\fint_{B_2(x)}|\Rm|^{n/2} <\epsilon\, .\label{L^{n/2}-average}
\end{align}
In this case it has then been proven that $|\Rm|\leq 1$ on $B_1(x)$,  see \cite{Anderson-L^2}.  This had proven to be especially powerful in the four dimensional case.  In fact, in the four dimensional case (\ref{L^{n/2}-average}) can be relaxed to $\int_{B_{2}(p)}|\Rm|^2<\epsilon$, which gives the $\epsilon$-regularity for collapsed Einstein $4$-manifolds under an assumed $L^2$ curvature bound, see \cite{CT}.  Unfortunately, in higher dimensions, even in the noncollapsed setting, there is no natural context in which one can apply the above $\epsilon$-regularity without simply assuming a bound on $\int_M |\Rm|^{n/2}$.  To fix this one has to look for a more powerful $\epsilon$-regularity theorem.  Namely, it is the combined work of \cite{Anderson}, \cite{ChC1} that if one assumes $|\Ric|\leq\epsilon$ and
\begin{align}\label{e:strong_noncollapsed_eps_reg}
d_{GH}(B_2(x),B_2(0^n))<\epsilon\, ,
\end{align}
where $d_{GH}$ is the Gromov-Hausdorff distance and $B_2(0^n)\subseteq \dR^n$, then it still holds that $|\Rm|\leq 1$ on $B_1(x)$. 
More generally, by the work of \cite{ChNa2} one needs only assume that $B_2(x)$ is Gromov-Hausdorff close to $\dR^{n-\ell}\times C(Y^{\ell-1})$ for a metric space $Y^{\ell-1}$ with $\dim C(Y^{\ell-1})=\ell\leq 3$, see Theorem \ref{t:eps_reg_noncollapsed} for a slight refinement.

These turn out to be a much more powerful $\epsilon$-regularity theorem, the reason being that the condition $d_{GH}(B_{2r}(x),B_{2r}(y^*))<\epsilon r$, with $y^*\in\dR^{n-\ell}\times C(Y^{\ell-1})$ a cone point, is one that can be shown to hold at most points and most scales $r<1$ under only a noncollapsing assumption, see \cite{ChC1}, \cite{ChNa1} for the classical and quantitative stratification results.  In fact, using this as a starting point, a whole regularity theory, including {\it apriori} $L^p$ estimates for the curvature, can be proven for general noncollapsed Einstein manifolds, see \cite{ChNa1}, \cite{ChNa2}.  One would like a similar structure theory for the collapsed setting.

The above $\epsilon$-regularity theorems, and their induced regularity theories, depend heavily on the assumption of noncollapsing.  To be more precise, given a pointed Riemannian manifold $(M^n,g,p)$ with $\Ric\geq -(n-1)$ we say that $M$ is noncollapsed, or $v$-noncollapsed, if $\Vol(B_1(p))\geq v>0$.  One can see that $\ln\Vol(B_1(p))$ plays the role of a weak energy, and so a lower bound on the volume acts as an upper bound on the energy, see \cite{ChNa1} for more on this.

The goal of this paper is to extend the $\epsilon$-regularity theorems of (\ref{e:strong_noncollapsed_eps_reg}) and \cite{ChNa2} to the collapsed setting.  That is, we wish to replace the assumption $d_{GH}(B_2(x),B_2(0^n))<\epsilon$ with
\begin{align}\label{e:strong_collapsed_eps_reg}
d_{GH}(B_2(x),B_2(0^k))<\epsilon\, ,
\end{align}
where $k<n$ and $0^k\in\dR^k$.  In fact, we will wish to deal with more general situations where $0^k\in \dR^{k-\ell}\times Z^{\ell}$, where $Z^{\ell}$ is a length space of dimension $\ell\leq 3$.  However, we will begin our discussion on the case $0^k\in \dR^k$.  The motivation is that even in the collapsed setting it is still the case that the assumption of 
\begin{equation}d_{GH}(B_{2r}(x),B_{2r}(0^k))<\epsilon r\end{equation} is sufficiently weak that it may be controlled at most points and scales, see \cite{ChC1}, \cite{ColdingNaber}.

Unfortunately, condition (\ref{e:strong_collapsed_eps_reg}) by itself is simply not enough when $k<n$.  We will discuss these examples in more detail in Section \ref{s:motivating examples}, but let us briefly describe an important example of Gross-Wilson in \cite{GrW}.  In this case by studying a holomorphic fibration map $f:K3\to S^2$ they build a sequence of Ricci flat spaces $(K3,g_j)$ which collapse to the topological $2$-sphere $X\approx S^2$.  $X$ is smooth away from a finite number of points $\{p_\alpha\}$, however the tangent cone is $\dR^2$ at every point of $X$.  In particular, for $r$ small and $j$ large every ball $B_r(x)\subseteq (K3,g_j)$ is close to the corresponding ball $B_r(0^2)\subset\dR^2$ in the sense of Gromov-Hausdorff topology.  However, the curvature blows up near $p_\alpha$, and hence the no $\epsilon$-regularity can possibly hold.

To see how to fix this, and to motivate our theorems, let us expand a little on the behavior of the above example.  If we begin again with a ball $B_{r}(x)\subseteq (K3,g_j)$, then we have two cases to analyze.  Namely, either the ball $B_r(x)$ is near a singular point or not.  On the one hand, if the ball is not near any singular point then we have the desired curvature bound $|\Rm|\leq r^{-2}$.  A closer look shows us that in this case the topology of $f^{-1}(B_r(x))\approx B_r(0^2)\times T^2$, where $T^2$ is the two dimensional torus.  In particular, we would like to emphasize that the fundamental group of $f^{-1}(B_r(x))$ is a free abelian group of rank $2$, which is of maximal rank.  On the other hand, let us assume for simplicity that $x=p_\alpha$ is one of the singular points.  In this case as we discussed we have no such {\it apriori} curvature bounds on $B_r(p_\alpha)$, and a closer look also tells us the topology of $f^{-1}(B_r(p_\alpha))$ is also more ill-behaved.  In particular $f^{-1}(p_\alpha)$ is topologically of Kodaira type $I_1$, which is a two dimensional torus with one circle fiber identified to a point.  In particular, we would like to emphasize that the fundamental group of $f^{-1}(B_r(p_\alpha))$ is of rank one, which is not maximal.  See Section \ref{s:motivating examples} for more on this.  See also \cite{GVZ} for more examples.

The main results of this paper are to see that this the above example is typical in its behavior.

\subsection{$\epsilon$-Regularity and the Fundamental Group}
\label{ss:1-1}

In this subsection we will state the main theorems of this paper.  Throughout, we are always assuming $(M^n,g,p)$ is a Riemannian manifold such that $B_2(p)$ has a compact closure in $B_4(p)$.  This acts as a form of local completeness.  Similar to the standard $\epsilon$-regularity theorem for noncollapsed spaces, we will be interested in the case when $B_2(p)$ has bounded Ricci curvature $|\Ric|\leq n-1$ and is Gromov-Hausdorff close to a singular space of the form $\dR^{k-\ell}\times Z^{\ell}$,
\begin{equation}\label{e:cone_space}
d_{GH}(B_2(p),B_2(0^{k-\ell},z^{\ell}))<\delta,\ (0^{k-\ell},z^{\ell})\in\dR^{k-\ell}\times Z^{\ell},\end{equation}
where $\dim(\dR^{k-\ell}\times Z^{\ell})=(k-\ell)+\ell\equiv k$.  Note by \cite{ColdingNaber} that if $\dR^{k-\ell}\times Z^{\ell}$ is the Gromov-Hausdorff limit of a sequence of manifolds with a lower Ricci bound then the length space $Z^{\ell}$ does in fact have a well defined dimension (See theorem \ref{limiting-dimension} for the more precise description).  We call such a $Z^\ell$ a Ricci limit space.  Inspired by \cite{FY} our $\epsilon$-regularity theorems come from studying the `fibered' fundamental group
\begin{equation}\label{e:fibered_fg}
\Gamma_\delta(p)\equiv \Image[\pi_1(B_{\delta}(p))\rightarrow\pi_1(B_2(p))]\, .
\end{equation}
Note that if $M$ were an honest fiber bundle over $\dR^{k-\ell}\times Z^{\ell}$, then $\Gamma_\delta(p)$ would be the fundamental group of the fiber.  Following the techniques of \cite{FY} and \cite{KW} it can be checked for $\delta\leq \delta_0(n,B_1(z^{\ell}))$ that under the assumption of (\ref{e:cone_space})  $\Gamma_{\delta}(p)$ contains a nilpotent subgroup of index $\leq w_0$
and rank $\leq n-k$, see Section \ref{sss:almost_nilpotency} and  Appendix \ref{s:proof-of-margulis} for more on this.  Note that given any finitely generated almost nilpotent subgroup $\Gamma$, the nilpotency rank of $\Gamma$ is well-defined. In the above context, it thus holds that $\rank(\Gamma_{\delta}(p))\leq n-k$.  Our main theorem states that an $\epsilon$-regularity result holds if $\rank(\Gamma_\delta(p))$ is maximal.  More precisely:

\begin{theorem}\label{t:eps_reg_collapsed}
Let $(M^n,g,p)$ satisfy the Ricci bound $|\Ric|\leq n-1$, then for each Ricci-limit space $(Z^{\ell},z^\ell)$ with $\dim Z^{\ell}= \ell\leq 3$ in the sense of \cite{ColdingNaber}, there exists a positive constant 
$\delta=\delta(n,B_1(z^{\ell}))>0$, $w_0=w_0(n,B_1(z^\ell))$, $c_0=c_0(n,B_1(z^\ell))$ such that if
\begin{equation}
d_{GH}\big(B_2(p),B_2(0^{k-\ell},z^{\ell})\big)<\delta,\ (0^{k-\ell},z^{\ell})\in\dR^{k-\ell}\times Z^{\ell},\label{Gromov-Hausdorff-Control}
\end{equation}
then $\Gamma_\delta(p)\equiv\Image[\pi_1(B_{\delta}(p))\rightarrow\pi_1(B_2(p))]$ is $(w_0,n-k)$-nilpotent with $\rank(\Gamma_{\delta}(p))\leq n-k$, and if equality holds then for each $q\in B_{1}(p)$ we have the conjugate radius bound
\begin{equation}\label{e:conj_bound_ein}
\Conj(q)\geq c_0(n,B_1(z^{\ell}))>0.
\end{equation}
In particular, if $M^n$ is Einstein, then we have that
\begin{equation}\sup\limits_{B_1(p)}|\Rm|\leq C(n,B_1(z^{\ell})).\end{equation}  
Finally, we have the following converse.  Assume $|\Ric|\leq n-1$ and $\Conj(q)\geq c_0>0$ holds for each $q\in B_3(p)$, then if 
\begin{equation}d_{GH}\big(B_3(p),B_3(0^{k})\big)<\delta(n,c_0),\ 0^k\in \dR^k, \label{smooth-limit}
\end{equation}
 then $\Gamma_{\delta}(p)\equiv\Image[\pi_1(B_{\delta}(p))\rightarrow\pi_1(B_2(p))]$ is almost nilpotent with $\rank(\Gamma_{\delta}(p))=n-k$. 
\end{theorem}
\begin{remark} To obtain the uniform curvature bound,
instead of Einstein we may simply assume a bound on $|\nabla\Ric|$.
\end{remark}
\begin{remark}
We will in fact prove a lower bound on the weak harmonic radius, not just the conjugate radius.  However, under bounded Ricci these two are known to be equivalent.
\end{remark}


\begin{remark}
The dimension assumption $\dim Z^{\ell}=\ell\leq 3$ in Theorem \ref{t:eps_reg_collapsed} is sharp. That is, if $\dim Z^{\ell}=4$, one cannot expect the $\epsilon$-regularity in this context.
See Example \ref{ss:example3} on Eguchi-Hanson metric for more details.
\end{remark}

\begin{remark}
In the converse direction of Theorem \ref{t:eps_reg_collapsed}, we have assumed
\begin{equation} d_{GH}(B_3(p), B_3(0^k))<\delta,
\end{equation}
where $B_3(0^k)\subset\dR^k$. Notice that this assumption 
 $B_3(0^k)\subset\dR^k$ is necessary. If this condition is removed, there are counterexamples in which the curvature is uniformly bounded while the rank is not maximal. See Example \ref{half-space-counter-example}.
 
\end{remark}


Let us note that unlike the $\epsilon$-regularity in the non-collapsed case (see Theorem \ref{t:eps_reg_noncollapsed}), the positive constants $\delta(n,B_1(z^{\ell}))$, $c_0(n,B_1(z^{\ell}))$, $C(n,B_1(z^{\ell}))$ have a dependence on $Z^{\ell}$ in the collapsed context.  In fact, this dependence can be made quite explicit.  Indeed, by \cite{ColdingNaber}, $Z^{\ell}$ has a unique dimension in that for $a.e.\  z^{\ell}\in Z^{\ell}$ the tangent cone is unique and isometric to $\dR^\ell$ (see theorem \ref{limiting-dimension} for the more precise statement).   Thus, given the base point $z^{\ell}\in Z^{\ell}$ let us consider the {\it noncollapsing radius} defined by
\begin{equation}\label{e:noncollapsing_radius}
r_{c}(z^\ell)\equiv \sup\Big\{0<r\leq 1\Big|\exists\, z\in B_1(z^\ell), d_{GH}\big(B_{r}(z),B_r(0^\ell)\big)<10^{-6} r\Big\}\, 
\end{equation}
 where $0^\ell\in \dR^\ell$.  Then, in Theorem \ref{t:eps_reg_collapsed} we have that $\delta=\delta(n,r_{c}(z^\ell)), c_0=c_0(n,r_{c}(z^\ell)), C=C(n,r_{c}(z^\ell))$.  Thus, a lower bound on the noncollapsing radius is a natural replacement for the lower volume bound assumption needed in the noncollapsing case.

\subsection{Outline of the Proof}\label{ss:intro_outline_proof}

The proof of the main result requires several basic steps.  Let us give the outline for Einstein manifolds, the general bounded Ricci case is verbatim, but one replaces curvature bounds by lower bounds on the harmonic or conjugate radius.  To describe the proof let us even begin by discussing the noncollapsed setting.  In \cite{ChNa2} it was proved that if $(M^n,g,p)$ satisfies $\Vol(B_1(p))>v>0$ then there exists $\epsilon(n,v)>0$ such that if $|\Ric|\leq \epsilon$ and
\begin{align}
d_{GH}\Big(B_2(p),B_2(0,y^*)\Big)<\epsilon\, ,
\end{align}
where $(0,y^*)\in \dR^{n-\ell}\times C(Y^{\ell-1})$ and $\ell\leq 3$, then
\begin{align}
\sup_{B_1(p)}|\Rm|\leq 1\, .
\end{align}
The first observation is that the assumption of a cone structure on the three dimensional factor, while natural in the noncollapsed setting, is very unnatural in the collapsed setting.  As a first step it is therefore helpful then to see that even in the noncollapsed setting it is unnecessary.  Indeed, combining the $\epsilon$-regularity results of \cite{ChNa2} with the quantitative differentiation of \cite{ChNa1}, we see in Theorem \ref{t:eps_reg_noncollapsed} that the following may be proved.  If $(M^n,g,p)$ satisfies $\Vol(B_1(p))>v>0$ then there exists $\epsilon(n,v)>0$ such that if $|\Ric|\leq \epsilon$ and
\begin{align}
d_{GH}\Big(B_2(p),B_2(0,z^*)\Big)<\epsilon\, ,
\end{align}
where $(0,z^*)\in \dR^{n-\ell}\times Z^\ell$ and $\ell\leq 3$, then
\begin{align}
\sup_{B_1(p)}|\Rm|\leq C(n,v)\, .
\end{align}
Notice that, in exchange for losing the cone structure we have only changed the curvature bound from $1$ to some uniform constant $C(n,v)$.

Now our primary goal in proving Theorem \ref{t:eps_reg_collapsed} is to reduce the statement to the one of Theorem \ref{t:eps_reg_noncollapsed}. To accomplish this we need to show two points.  First, if $(M^n,g,p)$ satisfies $|\Ric|\leq \epsilon$ with 
\begin{align}\label{e:outline:assumptions}
d_{GH}\Big(B_2(p),B_2(0,z^*)\Big)<\epsilon\, ,\notag\\
\rank(\Gamma_\delta(p))=n-k\, ,
\end{align}
where $(0,z^*)\in \dR^{k-\ell}\times Z^\ell$, then we need to show that
\begin{align}\label{e:outline:noncollapsed}
\Vol(B_1(\tilde p))>v>0\, ,
\end{align}
has some uniform lower bound, where $\tilde p$ is a lift of $p$ to the universal cover of $B_2(p)$.  Additionally, we then need to show that
\begin{align}\label{e:outline:splitting}
B_{r}(\tilde p)\approx B_r(0,z^*)\subseteq \dR^{n-\ell}\times \tilde Z^\ell\, ,
\end{align}
where $\tilde Z$ is potentially a different length space and $r>0$ is of some uniform size.

Let us begin by mentioning that the proof of \eqref{e:outline:noncollapsed} and \eqref{e:outline:splitting} is easier if $\ell=0$ and thus $Z^\ell\equiv\{pt\}$ is a point and $B_r(\tilde p)\approx \dR^n$ is a ball in $\dR^n$.  Indeed, in this case one can use the rank condition along with the ideas of \cite{FY} and the algebraic constructions of Section \ref{s:almost-nilpotent} to build $(n-k)$-independent lines on the universal cover of $B_2(p)$, which in turn force the new $\dR$ factors on the cover by using the splitting theorem.  The moral of this argument is similar to the abelian rank arguments presented in \cite{ChGr}.  The precise statement is in Proposition \ref{quantitative-splitting-at-origin}.

In the general case of \eqref{e:outline:noncollapsed} and \eqref{e:outline:splitting} when $\ell\neq 0$, one cannot directly use the rank condition to construct the new lines on the cover.  The reason for this is the potential lack of compactness of the factor $Z^\ell$.  See Section \ref{ss:prelim:GH_basics} to see why the compactness plays a role, and see Example \ref{ss:example 5} to see that the line splitting argument legitimately fails when $Z^\ell$ is not compact.  That is, in Example \ref{ss:example 5}, $Z^{\ell}$ is smooth and noncompact with positive Ricci curvature, and $\Gamma_{\delta}(p)$ has maximal nilpotency rank, but the non-collapsed covering space does not admit any lines.

  Indeed, we must argue in an entirely different manner, and instead rely on the cone splitting ideas of \cite{ChNa1} to build our new $\dR$-factors on the cover.  However, in the context of lower bounds on Ricci curvature, the construction of cone points fundamentally requires noncollapsing.  Thus, we will in fact need to prove \eqref{e:outline:noncollapsed} first, and then once this is in hand we will prove \eqref{e:outline:splitting} through a cone splitting argument.  This is all contained in Theorem \ref{non-collapsed-splitting-maximal-rank}, which is the main technical challenge of the paper.

In more detail, to prove \eqref{e:outline:noncollapsed} in the general case we begin by proving a refined version of the $\ell=0$ case.  Indeed, in Section \ref{ss:symmetry-and-splitting} we show that if the assumptions \eqref{e:outline:assumptions} hold with $\ell=0$, then for any normal covering of $B_2(p)$ whose deck transformation group satisfies the appropriate rank condition, the conclusion that $B_r(\hat p)\approx \dR^n$ holds, where $\hat p$ is a lift of $p$ to the given normal cover.  The proof of this is similar in spirit to that of the $\ell=0$ case on the universal cover, that is we want to use the rank assumption to build lines on the cover. Similar but weaker statements were proved in \cite{FY} and \cite{KW} in the context of lower sectional and Ricci respectively; in our case the algebraic requirements are much more involved, see Remark \ref{remark-on-line-splitting} for more detailed explanation on this.  In the construction of the new lines by exploiting the rank condition, a key step of this argument is to build geometrically {\it well behaved} collection of generators for the deck transformation group.  More precisely, in Section \ref{s:almost-nilpotent} we show how to build a geometrically compatible polycyclic extension of the lower central series for a general deck transformation group in this context.  With this in hand we can argue as in \cite{FY} to build independent lines on the normal cover, and thus force new splittings.

We use this result to prove the noncollapsing \eqref{e:outline:noncollapsed} in the following manner.  Returning to the general factor space $Z^\ell$, we use \cite{ColdingNaber} to see that there exists a ball $B_{2r}(z)\subseteq Z^\ell$ with $B_{2r}(z)\approx\dR^\ell$.  In particular, there exists a ball $B_{2r}(q)\subseteq B_2(p)$ in the original ball such that $B_{2r}(q)\approx \dR^{k}$.  If we lift $B_{2r}(q)$ to the universal cover of $B_2(p)$, then we can view the connected component of this lifting as a normal covering of $B_{2r}(q)$ itself.  We prove in Lemma \ref{every-point-rank} a nonlocalness property for the fibered fundamental group.  Roughly, it shows that if $\rank\Gamma_\delta(p)=n-k$, then $\rank\Big(\Image[\pi_1(B_{\delta' r}(q))\to \pi_1(B_2(p))]\Big)=n-k$, where $\delta<<\delta' r$.  This is sufficient to apply the results of Section \ref{ss:symmetry-and-splitting} discussed in the last paragraph, and in particular to see that there is some definite size ball $B_{r'}(q')\subseteq B_2(\tilde p)$ which is close to a ball in $\dR^n$.  By volume convergence one can immediately conclude the noncollapsing \eqref{e:outline:noncollapsed} for $B_2(\tilde p)$ itself.

Once \eqref{e:outline:noncollapsed} has been proved, then we can proceed to prove \eqref{e:outline:splitting} through a cone splitting argument.  More precisely, using the quantitative differentiation of \cite{ChNa1}, see lemma \ref{l:bad_scales_finite} for a precise statement, we see that after dropping to some definite radius that $B_{r}(\tilde p)\approx \dR^{k-\ell}\times C(Y)$ is very close to being a cone space with $\tilde p$ the cone point.  On the other hand, with the algebraic results of Section \ref{s:almost-nilpotent}, we can use the rank condition on the fibered fundamental group to construct $n-k$ independent points $p_1,\ldots, p_{n-k}\in C(Y)$ inside the cone factor, which are themselves also tips of cone points.  The cone splitting of \cite{ChNa1}, see lemma \ref{cone-splitting-principle}, tells us that if we have $n-k$ independent points for which a metric space is a cone space with respect to all of these points, then our metric space must split a $\dR^{n-k}$-factor.  That is, we can conclude that $C(Y)\approx \dR^{n-\ell}\times \tilde Z^\ell$, which in particular proves \eqref{e:outline:splitting}.  When $\ell\leq 3$, by applying Theorem \ref{t:eps_reg_noncollapsed} we can thus prove Theorem \ref{t:eps_reg_collapsed}.

\subsection{Organization of the Paper}\label{ss:organization}
This paper is organized as follows:

 In Section \ref{s:prelim} we will introduce some preliminary notions and results from the literature.  At several points we will give slight extensions of those results from the literature, but these extensions are minor and mostly well understood by experts.

Section \ref{s:motivating examples} is dedicated to presenting some motivating examples which intuitively show the pictures of the $\epsilon$-regularity in the collapsed setting.  Additionally, we will give examples in order to show that each assumption in Theorem \ref{t:eps_reg_collapsed} is sharp.

In Section \ref{s:almost-nilpotent} we focus primarily on studying some algebraic properties of almost nilpotent groups.  Our main result of this Section will be to produce a polycyclic extension of the lower central series of a nilpotent group which will behave well from a geometric point of view.  This will be important in the proof the quantitative splitting results of Section \ref{s:quantitative-splitting}.
 
In Section \ref{s:quantitative-splitting}, we will develop our central technical tools for the $\epsilon$-regularity in the collapsed setting. Our main result of this Section, and indeed the main technical point of the paper, will be Theorem \ref{non-collapsed-splitting-maximal-rank}, whose content was the focus of the previous subsection.

In Section \ref{s:eps_reg_lower_and_bounded_ricci} we combine the results of all the previous Sections in order to prove the main Theorem of the paper.  Indeed, we will also prove a version of the main Theorem in the context of only a lower Ricci bound, though the result is obviously much stronger under a two sided Ricci curvature bound.

In Appendix \ref{s:proof-of-margulis} we prove a generalization of the main result of \cite{KW}.  In fact, the proof is primarily just a combination of the results of \cite{KW} with some technical constructions from Section \ref{s:almost-nilpotent} and \ref{s:quantitative-splitting}.  However, since this refinement is important for applications we include it.

In Appendix \ref{s:proof-of-fiber-bundle} we give a proof of a fiber bundle structure in the context of bounded Ricci curvature and a lower conjugate radius bound.  This structure has been stated without proof in the literature, and in any case is well known to experts, but we include a proof for completeness since we require this result for the converse direction of Theorem \ref{t:eps_reg_collapsed}.

\section{Preliminaries}\label{s:prelim}

In this section, we will review some preliminary results needed for the paper.  At times we will give mild generalizations of known results in the literature.  The organization of this section is as follows.  In Section \ref{ss:prelim:GH_basics} we begin by recalling the notion of Gromov-Hausdorff convergence, as well as some basic results on isometries and lines in homogeneous spaces.  In Section \ref{ss:prelim:eps_reg} we recall the noncollapsed $\epsilon$-reglarity results of \cite{ChNa2}.  In Section \ref{ss:prelim:Riccilimit} we discuss some basic properties of Ricci limit spaces.  In particular, we recall from \cite{ColdingNaber} that there is a well defined dimension, and what this means, as we will be using this point.  Finally in Section \ref{ss:prelim:fundamental_group} we recall the basics of nilpotent groups and their relation to the fundamental group of spaces with lower Ricci bounds.  We recall the results of \cite{KW}, and give a mild extension of their results which will be useful in our context.  This extension is proved in Apprendix \ref{s:proof-of-margulis}.

\subsection{Gromov-Hausdorff Convergence and Group Actions}\label{ss:prelim:GH_basics}
\subsubsection{Basic concepts in Gromov-Hausdorff convergence theory} To begin with,
we list basic definitions in Gromov-Hausdorff theory.

\begin{definition}[Pointed $\epsilon$-GHA] Let $(X,p)$ and $(Y,q)$ be pointed metric spaces. A pointed map $f:B_{\epsilon^{-1}}(p)\rightarrow (Y,q)$ with $f(p)=q$
is called a pointed $\epsilon$-GHA (Gromov-Hausdorff approximation) if it satisfies the following:

(i) ($\epsilon$-onto) $B_{\epsilon^{-1}}(q)\subset B_{\epsilon}(f(B_{\epsilon^{-1}}(p)))$,

(ii) ($\epsilon$-isometric) for any $x_1,x_2\in B_{\epsilon^{-1}}(p)$, we have that
\begin{equation}|d_X(x_1,x_2)-d_Y(f(x_1),f(x_2))|<\epsilon.\end{equation}
\end{definition}
By the pointed $\epsilon$-GHA, we can define the pointed $GH$-distance between two metric spaces.
\begin{definition}[pointed GH-distance] Let $(X,p)$ and $(Y,q)$ be pointed metric spaces, then pointed GH-distance is defined as follows:
\begin{equation}d_{GH}^p((X,p),(Y,q))\equiv\inf\Big\{\epsilon>0\Big|\exists \epsilon\text{-GHAs}\
f: (X,p)\rightarrow(Y,q)\ \text{and}\ g:(Y,q)\rightarrow (X,p).\Big\}\end{equation}
\end{definition}
\begin{remark}
(1) Let $\mathcal{M}et^p\equiv\{\text{all of the isometric classes of proper and complete pointed metric spaces}\}$, then this collection endowed with the pointed GH-distance $(\mathcal{M}et^p,d_{GH}^p)$ is a complete metric space.

(2) For a convergent sequence of pointed complete metric spaces $\{(X_i,p_i)\}$, assume that they are proper and $\diam(X_i)\leq D$, then the convergence is independent of the choice of reference points.

\end{remark}

\subsubsection{Homogeneity and the existence of lines}  In this subsection, we introduce a general lemma by which we can construct a line on a non-compact length space.  

\begin{definition}Let $(X,d)$ be a metric space, the isometric embedding $\gamma:(-\infty,+\infty)\rightarrow X$ is called a line, and the isometric embedding $\gamma:[0,+\infty)\rightarrow X$ is called a ray.
\end{definition}

A standard result (see for instance \cite{yamaguchi}) is that, for any non-compact complete proper length space $X$, if it is homogeneous, then it admits a line.  The following is a slight generalization of this fact, and is fairly standard.  We give a proof for the convenience of the readers.

\begin{lemma}\label{homoline}
Let $(X,d)$ be a non-compact complete locally compact length space. If $X$ is $C$-homogeneous for some $C<\infty$, that is, for every $x,y\in X$, there is an isometry $f\in \Isom(X)$ with
$d_X(y,f(x))\leq C,$
then $X$ admits a line.
\end{lemma}

\begin{remark}If the constant $C$ in the above lemma is $0$, the space $X$ is homogeneous and thus the above lemma is the standard one.\end{remark}

\begin{proof}
First, we show that for any $p\in X$ there is a ray at $p$.
Take a sequence of points $\{q_i\}\subset X$ such that $d(p,q_i)\rightarrow\infty$.
Since $X$ is a locally compact and complete length space,
for each $q_i$, there is a minimal geodesic joining $q_i$ and $p$, denoted by $\gamma_i$.
Let $B_1(p)$ be a closed metric ball of radius $1$ in $X$, so it is compact in our context.
Then, by Arzel{\`a}-Ascoli theorem, there exists a subsequence $\{\gamma_{i_1}\}$ which converge to a minimal geodesic $\gamma_1:[0,1]\rightarrow X$.
Similarly, in $B_2(p)$, there exists a subsequence of $\{\gamma_{i_1}\}$, denoted by $\{\gamma_{i_2}\}$ which converge to a minimal geodesic $\gamma_2:[0,2]\rightarrow X$. Note that $\gamma_2$ coincides with $\gamma_1$ in $B_1(p)$. Continuing we can construct $\gamma_3$, $\gamma_4$ and so forth.
Then the ray is defined by $\gamma=\bigcup\limits_{i=1}^{\infty}\gamma_i:[0,\infty)\rightarrow X$.

The next is to construct a line from $\gamma$. For each $m\in\mathbb{N}$, 
let $g_m\in\Isom(X)$ be an isometry satisfying $d_X(g_m\cdot\gamma(m),p)\leq C$.
We denote by $\gamma_m=g_m\circ\gamma:[0,+\infty)\rightarrow X$ the image of $\gamma$ under the isometry $g_m\in\Isom(X)$, then $\gamma_m$ is a ray.
Let  $\sigma_m:[-m,+\infty)\rightarrow X$ be a ray defined by $\sigma(t):=\gamma_m(t+m)$. Since $d_X(\sigma_m(0),p)=d_X(g_m\cdot\gamma(m),p)\leq C$,  after passing to a subsequence the rays $\sigma_m:[-m,\infty)\rightarrow X$ converge to a line.
\end{proof}

\subsection{Ricci Curvature and Noncollapsed $\epsilon$-regularity Theorems}\label{ss:prelim:eps_reg}

The following is a standard useful manner of measuring regularity on a manifold:

\begin{definition}[$C^{1}$-harmonic coordinates]
Let $u=(u_1,\ldots,u_n):B_r(p)\rightarrow\dR^n$ with $u(p)=0$ and $u$ a diffeomorphism onto its image.  We call $u$ a $C^{1}$-harmonic coordinates system with $\|u\|_r\leq 1$ if the following properties hold:
\begin{enumerate}
\item For each $1\leq k\leq n$, $u_k$ is harmonic.
\item If $g_{ij}=g(\nabla u_i,\nabla u_j)$ is the metric in coordinates, then
$|g_{ij}-\delta_{ij}|_{C^0(B_r(p))}+r| g_{ij}-\delta_{ij}|_{C^1(B_r(p))}<10^{-6}$, where the scale-invariant norms are taken in euclidean coordinates.
\end{enumerate}
\end{definition}

Using the above we immediately have the notion of the harmonic radius of $M$:
\begin{definition}
For $x\in M$ we define the harmonic radius $r_h(x)$ by
\begin{align}
r_h(x)\equiv\sup\{r>0| \exists \text{ harmonic coordinates } u:B_r(x)\to \dR^n \text{ with } \|u\|_r\leq 1\}.
\end{align}
\end{definition}
\begin{remark}
Note that the harmonic radius is scale invariant.  That is, if $r_h(x)\equiv r$, and we rescale the metric of $M$ so that $B_r(x)\to B_1(x)$, then $r_h(x)\equiv 1$ in the new space.
\end{remark}
\begin{remark}
Notice the Lipschitz bound $|\nabla r_h|\leq 1$.
\end{remark}

It was discussed in the introduction that a key $\epsilon$-regularity theorem tells us that for spaces with bounded Ricci curvature $|\Ric|\leq\epsilon(n)$ that if $d_{GH}(B_2(p),B_2(z^*))<\epsilon(n)$, where $z^*\in \dR^{n-3}\times C(Y)$, then $r_h(p)\geq 1$.  We will want to study in this section what happens when the cone assumption is dropped.  In this case we have the following, which is a slight extension of the $\epsilon$-regularity theorem (theorem 6.1) in \cite{ChNa2}.

\begin{theorem}\label{t:eps_reg_noncollapsed} Given $n$, $v>0$, there are positive constants 
$\epsilon=\epsilon(n,v)>0$ and $r_0=r_0(n,v)>0$ such that if  $(M^n,g,p)$ satisfies $|\Ric|\leq n-1$, $\Vol(B_1(p))>v>0$ and
\begin{equation}
d_{GH}(B_2(p),B_2(0^{n-3},y))<\epsilon, \ (0^{n-3},y)\in \dR^{n-3}\times Y,\label{(n-3)-GH-control}
\end{equation}
for some metric space $(Y,y)$,
 then for each $q\in B_1(p)$, we have that
 \begin{equation}r_h(q)\geq r_0>0.\end{equation}
\end{theorem}


\begin{remark}
We can also write the above result in a scale-invariant version.  That is, if $|\Ric|\leq n-1$, $\Vol(B_r(p))> v\cdot r^n$ and \begin{equation}d_{GH}(B_{2r}(p),B_{2r}(0^{n-3},y))<\epsilon r,\ (0^{n-3},y)\in\dR^{n-3}\times Y
 \end{equation}
 then correspondingly we have the harmonic radius bound $r_h(p)\geq r\cdot r_0$.
\end{remark}
Before the proof, we need clarify some notions.
\begin{definition}
Let $(X,d)$ be a length space, then we define the following:

$(i)$ We say $X$ is $k$-symmetric at $x$ if there is some compact metric space $Y$ such that $X$ is isometric to $C(Y)\times\dR^k$, where $C(Y)$ is  a cone space over $Y$ and $x$ is the cone tip under this isometry.

$(ii)$ Given $x\in X$, $0<r\leq 1$ and $\epsilon>0$, we say $X$ is $(k,\epsilon,r)$-symmetric at $x$ if there is some compact  metric space $Z$ which is $k$-symmetric at $z\in Z$ such that
\begin{equation}
d_{GH}(B_r(x),B_r(z))<r\epsilon.
\end{equation}

$(iii)$ Denote by $r_{\alpha}= 2^{-\alpha}$. Given $x\in X$, we call $r_{\alpha}$
 is a {\textit{good scale}} if $X$ is $(0,\epsilon,r_{\alpha})$-symmetric at $x$, and a {\textit{bad scale}} otherwise.
\end{definition}
The following lemmas allow us to reduce the result of Theorem \ref{t:eps_reg_noncollapsed} to theorem 6.1 of \cite{ChNa2}:

\begin{lemma}\label{l:bad_scales_finite}
\cite{ChNa1}\label{bad-scale-estimate} Let $(M^n,g,p)$ be a Riemannian manifold with $\Ric\geq-(n-1)$ and $\Vol(B_1(p))>v>0$. Then for each $\epsilon>0$, there exists $N(n,\epsilon,v)>0$ such that
such that for every $q\in B_1(p)$, we have that there are at most $N$ bad scales at $q$. 
\end{lemma}
We also the need the following quantitative cone splitting lemma which simply follows from standard contradicting arguments.
\begin{lemma}\label{cone-splitting-lemma}
Let $(M^n,g,p)$ be a Riemannian manifold with $\Ric\geq-(n-1)$. For each $\epsilon>0$, there exists $\delta_0(n,\epsilon)>0$ such that if for some metric space $(Y,y)$,
\begin{equation}
d_{GH}(B_2(p), B_2(0^{k},y))<\delta_0,\ \ (0^k,y)\in\dR^k\times Y,
\end{equation}
and
\begin{equation}
d_{GH}(B_2(p),B_2(y_1))<\delta_0,\ y_1\in C(Y_1)
\end{equation}
for some cone space $C(Y_1)$ over some compact metric space $Y_1$ with the cone tip $y_1$, then
we have that
\begin{equation}
d_{GH}(B_2(p),B_2(0^{k},y_0))<\epsilon,\ (0^k,y_0)\in\dR^{k}\times C(Y_0)
\end{equation}
where $C(Y_0)$ is some cone space over a compact metric space $Y_0$ with the cone tip $y_0$.
\end{lemma}

Applying the above lemmas, we proceed to prove Theorem \ref{t:eps_reg_noncollapsed}.
\begin{proof}[The proof of Theorem \ref{t:eps_reg_noncollapsed}] 
Fix $n$, $v$, we will explicitly determine 
\begin{equation}\epsilon(n,v)>0, r_0=r_0(n,v)>0
\end{equation}
 such that for some metric space $Y$,
if (\ref{(n-3)-GH-control}) holds for $\epsilon>0$, then 
\begin{equation} r_h(p)\geq r_0>0.
\end{equation}

Let $\epsilon'=\epsilon'(n,v)>0$ be the positive constant in the non-collapsed $\epsilon$-regularity of theorem 6.1 in \cite{ChNa2} and denote by $\delta_0=\delta_0(\epsilon',n)>0$ the constant in Lemma \ref{cone-splitting-lemma}. Take any $q\in B_1(p)$, Lemma \ref{bad-scale-estimate} implies that, with respect to the constant 
$\delta_0>0$ defined above,  we can drop a definite number of factors \begin{equation}N=N(n,\delta_0(\epsilon',n),v)>0\end{equation} such that there exists some $r>0$ with
\begin{equation} 2^{-2N(n,\epsilon',v)}< r <2^{-N(n,\epsilon',v)}<<\epsilon',\label{harmonic-radius-determined}\end{equation} then for some cone space $C(Y^*)$, it holds that 
 \begin{equation}d_{GH}(B_{2r}(q),B_{2r}(y^*))<r\delta_0,\ q\in B_1(p), \label{good-scale}\end{equation} where $y^*\in C(Y^*)$
 is the cone tip. 
Let (\ref{(n-3)-GH-control}) hold for 
\begin{equation}0<\epsilon\equiv r\cdot\delta_0(n,\epsilon')\, .\label{GH-control-determined}
\end{equation}
Then if we restrict to $B_{2r}(q)\subset B_2(p)$, it holds that
\begin{equation}
d_{GH}(B_{2r}(q),B_{2r}(0^{n-3},y))<\epsilon=r\cdot\delta_0.\end{equation}
Combining with (\ref{good-scale}), by Lemma \ref{cone-splitting-lemma}, there exists
a cone space $\dR^{n-3}\times C(Y_0)$ over some compact metric space $Y_0$ with the cone tip $(0^{n-3},y_0)$ such that
 \begin{equation}
 d_{GH}(B_{2r}(q),B_{2r}(0^{n-3},y_0))<r\epsilon'.
 \end{equation}
Applying theorem 6.1 of \cite{ChNa2}, we obtain the harmonic radius bound for each $q\in B_1(p)$,
\begin{equation}
r_h(q)\geq r>0.\end{equation}
Since the constants $r>0$, $\epsilon>0$ are determined by  (\ref{harmonic-radius-determined}) and (\ref{GH-control-determined}) which depend only on $n$, $v$, we have finished the proof.

 \end{proof}
 
We end this subsection by introducing the following \textit{Cone-splitting Principle} (see \cite{ChNa1} for more details) which is a powerful tool to study the $\epsilon$-regularity especially in the non-collapsed context (see Subsection \ref{ss:symmetry-and-splitting}). The following lemma will be heavily used in the proof of the quantitative splitting result.

\begin{lemma}[Cone-splitting Principle] \label{cone-splitting-principle} 
Let $(C(Y),y^*)$ be a metric cone with vertex $y^*$ over some compact metric space $Y$. Assume that there exists a metric cone $(C(\bar{Y}),\bar{y}^*)$ and there is an isometry $F:\dR^s\times C(Y')\rightarrow C(Y)$ such that $y^*\not\in F(\dR^s\times\{\bar{y}^*\})$, then for some compact metric space $W$, $C(Y)$ is isometric to $\dR^{s+1}\times C(W)$.
\end{lemma}

\subsection{Geometry of Ricci-limit Spaces}\label{ss:prelim:Riccilimit}

In this subsection, we briefly discuss the geometric properties of a Ricci-limit space which will be used in the proof of our $\epsilon$-regularity Theorems.  Specifically, we call a metric-measure space $(X,d_X,\nu_X,x)$ a Ricci-limit space if there exists a sequence $(M_i^n,g_i,\nu_i,p_i)$ of Riemannian manifolds with $\Ric_{g_i}\geq-(n-1)\lambda$ and $\nu_i=\Vol(B_1(p_i))^{-1}dv_{g_i}$ such that
\begin{equation}
(M_i^n,g_i,\nu_i, p_i)\xrightarrow{GH}(\dR^k\times X,d_{\dR^k\times X},\nu_{\dR^k}\times \nu_X,x)\, .
\end{equation}
Note this differs a little from other papers in the literature, which might call a Ricci limit space a metric-measure space which arises as a direct limit of Riemannian manifolds.  It takes very little work to see that all the results that hold for these spaces also hold for Ricci-limit spaces in our sense.  In this subsection, we focus on the concept of the dimension of a Ricci-limit space and its isometry group.

\subsubsection{The dimension of a Ricci-limit space}

There is a natural definition of the dimension of a Ricci-limit space. This definition is given by the following theorem, which follows from \cite{ColdingNaber} directly.

\begin{theorem}
[\cite{ColdingNaber}] \label{limiting-dimension}
Let $(X,d,\nu,x)$ be a Ricci limit space.  Then there is a unique integer $k\geq 0$ such that $\nu(X\setminus\mathcal{R}_k)=0$,
 where 
 \begin{equation}
 \mathcal{R}_k(X)\equiv\Big\{y\in X\Big|\text{each tangent cone at}\  y\ \text{is isometric to}\ \dR^k\Big\},
 \end{equation}
The above unique integer $k$ is called the dimension of  $X$, 
 and the points in $\mathcal{R}_k(X)$ are called $k$-regular points.
\end{theorem}

\subsubsection{The isometry group of a Ricci-limit space}

To study the isometry group of a Ricci-limit space, we recall some basic notions in compact-open topology.
Given a metric space $Y$, it is standard to define the topology of $\Isom(Y)$ which is called the compact-open topology. Denote
\begin{equation}
\mathcal{V}(K,U)\equiv\{g\in\Isom(Y)|g(K)\subset U\},
\end{equation}
where $K\subset Y$ is compact and $U\subset Y$ is open and let
\begin{equation}
\mathcal{K}(Y)\equiv\{\mathcal{V}(K,U)|K\ \text{is compact},\ U\ \text{is open}\}.
\end{equation}

\begin{definition}
The compact-open topology of $\Isom(Y)$ is the topology which is generated by $\mathcal{K}(Y)$ such that $\mathcal{K}(Y)$ is the subbase. 
\end{definition}
Let us briefly review some basic facts of the compact open topology of $\Isom(Y)$.
In fact, the compact open topology of $\Isom(Y)$ is equal to the compact-convergence topology.
Moreover, $\Isom(Y)$ and $Y$ share the same separation properties with respect to the 
compact-open topology.
There is another standard fact that if $Y$ is a proper metric space, then $\Isom(Y)$
is locally compact with respect to the compact-open topology.

Let $G\equiv\Isom(Y)$ be the isometry group of $Y$ endowed with the compact-open topology. Hence, $G$ is locally compact. Gleason-Yamabe's theorem gives a criterion for a locally compact topological group to be a Lie group. Roughly speaking, a locally compact topological group $G$ is a Lie group if $G$ has no small subgroup. The above idea is applied to study the isometry group of a Ricci-limit space.  Applying the H\"older continuity result of tangent cones developed in \cite{ColdingNaber}, the following was proved:
\begin{theorem}
[\cite{ColdingNaber}]\label{lie-group} Let  $(Y,d,\nu,y)$  be a Ricci-limit space, then $\Isom(Y)$ is a Lie group.
\end{theorem}

Let $G$ be as above. 
Now we give a metric characterization of the identity component of $G$.
Denote by $G_0$ the identity component of the Lie group $G=\Isom(Y)$, and then 
we will prove the following Lemmas.

\begin{lemma}\label{small-neighborhood} In the above notations, let $\mathcal{B}_e\subset\Isom(Y)$ be any open set containing the identity element $e\in\Isom(Y)$, there exists $\epsilon_0(Y,\mathcal{B}_e)>0$ such that
\begin{equation}
I(\epsilon_0)\equiv\Big\{g\in\Isom(Y)\Big| d(g\cdot z,z)<\epsilon_0,\ \forall z\in\overline{B_{\epsilon_0^{-1}}(y)}\Big\}
\end{equation}
is an open set containing the identity and
\begin{equation}
I(\epsilon_0)\subset \mathcal{B}_e.
\end{equation}
 
\end{lemma}

\begin{proof} 
First, by definition, for each fixed $\epsilon>0$, it is obvious that
\begin{equation}
I(\epsilon)=\mathcal{V}(\overline{B_{\epsilon^{-1}}(y)}, T_{\epsilon}(\overline{B_{\epsilon^{-1}}(y)})),
\end{equation}
where $ T_{\epsilon}(\overline{B_{\epsilon^{-1}}(y)}))$ is the $\epsilon$-neighborhood of the compact set 
$\overline{B_{\epsilon^{-1}}(y)}$. Hence,
$I(\epsilon)$
 is a subbase element and thus $I(\epsilon)$ is open.

Since $\mathcal{B}_e$ is open,
by the definition of compact-open topology, there exist a open set $e\subset\mathcal{B}_e'\subset\mathcal{B}_e$ and finitely many $\mathcal{V}(K_j,U_j)$ for $1\leq j\leq N$ such that
\begin{equation}
e\in \mathcal{B}_e'=\bigcap\limits_{j=1}^N\mathcal{V}(K_j,U_j).
\end{equation}
We will prove that for some sufficiently small $\epsilon_0(Y,\mathcal{B}_e)>0$, $I(\epsilon_0)\subset\mathcal{B}_e'$. 
It suffices to show that for all $1\leq j\leq N$ and
for every $g\in I(\epsilon_0)$, it holds that
\begin{equation}
g(K_j)\subset U_j.
\end{equation}
Since $\bigcup\limits_{j=1}^NK_j$ is compact, 
there exists $R_0>0$ such that
\begin{equation}
\bigcup\limits_{j=1}^NK_j\subset B_{R_0}(y).
\end{equation}
Notice that, $e(K_j)=K_j$
for every $1\leq j\leq N$, which implies that $K_j\subset U_j$. Since $K_j$ is compact, there exists an finite open cover of $K_j$ such that
\begin{equation}
K_j\subset \bigcup\limits_{\nu=1}^{m_j}O_{\nu}\subset U_j.
\end{equation}
Let $\delta_j>0$ be the Lebesgue number of this open cover, and then for each $q_j\in K_j$ it holds that for some $1\leq\nu_j\leq m_j$,
\begin{equation}
B_{\delta_j/3}(q_j)\subset O_{\nu_j}\subset U_j.
\end{equation}
Let 
\begin{equation}
\epsilon_0\equiv\frac{1}{3}\min\Big\{\delta_1,\ldots, \delta_N, (10\cdot R_0)^{-1}\Big\},
\end{equation}
then for each $z_j\in K_j\subset B_{R_0}(y)\subset B_{\epsilon_0^{-1}}(y)$ and for each $g\in I(\epsilon_0)$,
\begin{equation}
g\cdot z_j\in B_{\epsilon_0}(z_j)\subset B_{\delta_j/3}(z_j)\subset U_j.
\end{equation}
Therefore, $g\in \mathcal{V}(K_j,U_j)$. So we have finished the proof.

\end{proof}

\begin{lemma}\label{generating-set-of-identity-component}Let $(Y,d,\nu,y)$ be a Ricci-limit space and let $G\equiv\Isom(Y)$. Denote by $G_0$ the identity component of the Lie group $G$,  then there exists $\epsilon_0(Y)>0$ such that 
\begin{equation}
G_0=\langle I(\epsilon_0)\rangle.\end{equation}
\end{lemma}

\begin{proof}

Since $(Y,y)$ is a Ricci limit space, by theorem \ref{lie-group},
$G=\Isom(Y)$ is a Lie group and hence the identity component $G_0$ is a normal subgroup of $G$.
By Lemma \ref{small-neighborhood}, given the identity component $G_0$, there exists $\epsilon_0(Y)>0$ such that \begin{equation}I(\epsilon_0)\subset G_0.\end{equation}
If $I(\epsilon_0)=e$.
Notice that $G_0$ is Hausdorff and thus $I(\epsilon_0)$ is closed.
Since $I(\epsilon_0)$ is open and $G_0$ is connected,
\begin{equation}
G_0=\{e\}=I_{\epsilon_0}.
\end{equation}
We are done. Now we focus on the case $I(\epsilon_0)\neq\{e\}$, that is, $I(\epsilon_0)$ is a non-trivial neighborhood containing $e$. Since $G_0$ is a connected Lie group, $G_0$ can be generated by any arbitrary non-trivial neighborhood. Therefore, 
\begin{equation}
G_0=\langle I(\epsilon_0)\rangle.
\end{equation}

\end{proof}

\subsection{Ricci Curvature and Fundamental Group}\label{ss:prelim:fundamental_group}

This section is to show the connections between curvature and fundamental group. Specifically, we will discuss some Generalized Margulis Lemmas in the context of sectional and in general Ricci curvature bounded from below. Roughly speaking, the moral of General Margulis Lemma is that 
if curvature (sectional or Ricci) is bounded from below, the subgroup of $\pi_1(B_2(p))$ generated by short loops is well controlled, i.e. an almost nilpotent group. 

\subsubsection{Nilpotent groups and polycyclic groups}

In this subsection, we recall some basic notions related to nilpotent groups and polycyclic groups. First, we give the definitions of a nilpotent group and a polycyclic group.

\begin{definition} A group $\Gamma$ is called {\textit{nilpotent}} if the lower central series of $\Gamma$ is of finite length, that is, there exists a finite descending normal series 
\begin{equation}
\Gamma\equiv\Gamma_0\rhd\Gamma_1\rhd\ldots\rhd\Gamma_k=\{e\},
\end{equation}
where $\Gamma_{j}\equiv[\Gamma,\Gamma_{j-1}]$ for all $1\leq j \leq k$. The length of the lower central series is called the \textit{nilpotency step or class}  of $\Gamma$, namely, $\Step(\Gamma)=k$.
\end{definition}

\begin{definition}
A group $\Gamma$ is called \textit{polycyclic} if there is a finite subnormal series 
\begin{equation}\Gamma\equiv\Gamma_m\rhd\Gamma_{m-1}\rhd\ldots\rhd\Gamma_1\rhd\Gamma_0=\{e\}\end{equation}
such that $\Gamma_{j-1}$ is normal in $\Gamma_j$ and $\Gamma_{j}/\Gamma_{j-1}$ is cyclic for each $1\leq j\leq m$. Such a subnormal series is called a polycyclic series. 
\end{definition}
\begin{remark}
It is known that all finitely generated nilpotent groups are polycyclic (see \cite{funda}).
\end{remark}

For a polycyclic group, we introduce the following invariants 
 which will be used to define the rank of a finitely generated nilpotent group.
\begin{lemma}
Let $\Gamma$ be a polycyclic group with a subnormal series\begin{equation}\Gamma\equiv\Gamma_m\rhd\Gamma_{m-1}\rhd\ldots\rhd\Gamma_1\rhd\Gamma_0=\{e\},\end{equation}
such that $\Gamma_{j-1}$ is normal in $\Gamma_{j}$ and $\Gamma_{j}/\Gamma_{j-1}$ is nontrivial and cyclic for each $1\leq j\leq m$. Define the integer
\begin{equation}
\N_{\mathbb{Z}}(\Gamma)\equiv\#\Big
\{1\leq j\leq m\Big|\Gamma_{j}/\Gamma_{j-1}\cong\mathbb{Z}\Big\}.
\end{equation}
 Then  $\N_{\mathbb{Z}}(\Gamma)$ is independent of the choice of the polycyclic series, and is called the polycyclic rank of $\Gamma$. In particular, if $\Gamma$ is torsion-free, then the above normal series can be chosen such that $\Gamma_{j}/\Gamma_{j-1}\cong\mathbb{Z}$ for each $0\leq j\leq m-1$.
\end{lemma}

\begin{definition}
Let $\Gamma$ be a finitely generated nilpotent group, then the nilpotency rank is defined by:
\begin{equation}
\rank(\Gamma)\equiv\N_{\mathbb{Z}}(\Gamma).
\end{equation}

\end{definition}

Let us define the nilpotency length of a finitely generated nilpotent group.
\begin{definition}
Let $\Gamma$ be a finitely generated nilpotent group, then its {\textit{nilpotency length}}, denoted by $\length(\Gamma)$, is defined by the length of the shortest polycyclic series 
\begin{equation}
\{e\}=A_0\lhd A_1\lhd A_2\lhd\ldots\lhd A_k=N
\end{equation}
which satisfies for all $1\leq i,j\leq k$,
\begin{equation}
[A_i,A_j]\leq A_{\min\{i,j\}-1}.
\end{equation}
\end{definition}

The following lemma for a finitely generated nilpotent group will be applied in our later Theorems. The proof of the lemma is quite standard, for instance, see \cite{funda}. 

\begin{lemma}\label{basic-nil-lemma} Let $\Gamma$ be a finitely generated nilpotent group, then the following hold:

$(i)$ 
$\Gamma$ contains a torsion-free nilpotent subgroup of finite index.

$(ii)$ Given a nilpotent subgroup $\Gamma'\leq \Gamma$, then $[\Gamma:\Gamma']<\infty$ iff $\rank(\Gamma)=\rank(\Gamma')$.

$(iii)$ $\rank(\Gamma)\leq\length(\Gamma)$, $\Step(\Gamma)\leq\length(\Gamma)$.\end{lemma}

In fact, the above definition of nilpotency rank can be extended to a finitely generated almost nilpotent group. 
\begin{definition}
A group $\Gamma$ is called \textit{almost nilpotent} if there exists a nilpotent subgroup $N$ with $[\Gamma:N]<\infty$, and it is called $(C,m)$-\textit{nilpotent} if there exists a nilpotent subgroup $N$
with $[\Gamma:N]\leq C$ and $\length(N)\leq m$.
\end{definition}
The following lemma gives the definition of the rank of a finitely generated almost nilpotent group.

\begin{lemma}\label{basic-almost-nilpotent}
Let $\Gamma$ be a finitely generated almost nilpotent group, then

$(i)$ each nilpotent subgroup $N$ with $[\Gamma:N]<\infty$ has the same nilpotency rank. The common rank is called the nilpotency rank of $\Gamma$, denoted by $\rank(\Gamma)$,

$(ii)$ if there exists $N\leq\Gamma$ such that $\rank(N)=m$, then $\rank(\Gamma)\geq m$.
\end{lemma}

\subsubsection{Almost nilpotency theorems}\label{sss:almost_nilpotency}

Let us now recall the connection between Ricci curvature and nilpotency of the fundamental group which is called the Generalized Margulis Lemma.  The first such results go back to K. Fukaya and T. Yamaguchi (see \cite{FY} and \cite{yamaguchi-alexandrov}), where under the assumption of a lower sectional curvature bound (or generally in the context of Alexandrov space) they proved the following:

\begin{theorem}[\cite{FY}, \cite{yamaguchi-alexandrov}]
There exists some positive constant $\epsilon(n)>0$ such that the following holds:  If $(X^n,p)$ is an $n$-dimensional Alexandrov space with curvature $\geq-1$, then for any $p\in X$ we have that
\begin{equation}
\Gamma_{\epsilon}(p)\equiv\Image[\pi_1(B_{\epsilon}(p))\rightarrow\pi_1(B_1(p))]\end{equation}
is almost nilpotent.
\end{theorem}

The primary drawback of the above result is that the index of the nilpotent group was not a priori bounded.  This was improved more recently in \cite{KPT} to deal with this, where it was shown the index is indeed uniformly bounded.  Recently using similar techniques, though technically much more demanding, these results have been extended in \cite{KW} to the Ricci curvature context.  

\begin{theorem}[\cite{KW}]\label{Generalized-Margulis-Lemma}
There are positive constants $\epsilon_1(n)>0$
and $w(n)<\infty$ such that the following holds:  If $(M^n,g,p)$ is a complete $n$-dimensional complete Riemannian manifold with $\Ric\geq-(n-1)$, then for any $p\in M^n$ we have that
\begin{equation}
\Gamma_{\epsilon_1}(p)\equiv\Image[\pi_1(B_{\epsilon_1}(p))\rightarrow\pi_1(B_1(p))]\end{equation}
is $(w(n),n)$-nilpotent.
\end{theorem}

Though not explicitly stated, in fact their techniques lead to some refinements, which will be important in this work.  In particular, with the help of the algebraic structure of Section \ref{s:almost-nilpotent}, the nonlocalness Lemma of Section \ref{s:quantitative-splitting}, the dimensional result of theorem \ref{limiting-dimension}, and the induction lemma of \cite{KW} we prove the following in Appendix \ref{s:proof-of-margulis}:

\begin{theorem}\label{t:KW_almost_nilpotent} Let $(Z^k,z^k)$ be a pointed Ricci-limit metric space with $\dim Z^k = k$ in the sense of theorem \ref{limiting-dimension}.  Then there exists 
$\epsilon_0=\epsilon_0(n,B_1(z^k))>0$, $w_0=w_0(n,B_1(z^k))<\infty$ such that if a Riemannian manifold $(M^n, g,p)$ with 
$\Ric\geq-(n-1)$ satisfies that $B_2(p)$ has a compact closure in $B_4(p)$ and
 \begin{equation}d_{GH}(B_2(p),B_2(z^k))<\epsilon_0,
\end{equation}
 then the   group
\begin{equation}
\Gamma_{\epsilon_0}(p)\equiv \Image[\pi_1(B_{\epsilon_0}(p))\rightarrow\pi_1(B_2(p))]\label{image-group}\end{equation}
is $(w_0,n-k)$-nilpotent. In particular, $\rank(\Gamma_{\epsilon_0}(p))\leq n-k$.\end{theorem}
\begin{remark}
The constants $\epsilon_0$, $w_0$ depend only on the noncollapsing radius $r_c(z^k)$, see \eqref{e:noncollapsing_radius}.
\end{remark}

\section{Motivating Examples}\label{s:motivating examples}

The eventual goal of this paper is to $\epsilon$-regularity Theorems for collapsed manifolds with Ricci curvature bounds.  The straightforward generalization of the noncollapsed $\epsilon$-regularity of Theorem \ref{t:eps_reg_noncollapsed} fails, and in this Section we study the important examples which tell us why, and help motivate the main Theorem's of the paper. We will also present some examples which show that the assumptions in our $\epsilon$-regularity Theorems are optimal.

\subsection{Example \ref{ss:example1}: Anderson's Example of Codimension 1 Collapse}\label{ss:example1}

In \cite{AnEx} via a surgery construction of the Riemannian Schwarzschild metric M. Anderson built a complete simply connected Ricci flat manifold $(M^n,g)$ which is asymptotically $\dR^{n-1}\times S^1$ outside of a disk.  In particular we have that the asymptotic cone of $M^n$ is isometric to $\dR^{n-1}$, which is to say for $x\in M^n$ fixed that as $r\to\infty$ we have
\begin{align}
\Big(M^n,r^{-2}g,x\Big) \xrightarrow{GH}(\dR^{n-1},0^{n-1}).
\end{align}
In particular, for any $\epsilon>0$ and all $r$ sufficiently large we have that
\begin{align}\label{e:example1:1}
d_{GH}\Big( B_2(x_r), B_2(0^{n-1})\Big)<\epsilon\, ,
\end{align}
 where $x_r\in M^n$ denotes the point $x$ in the space $(M^n,r^{-2}g)$.  On the other hand, it is easy to check in the example that $r_h(x_r)\approx r^{-1}\to 0$ as $r\to\infty$, and hence the direct version of Theorem \ref{t:eps_reg_noncollapsed} cannot hold in the collapsed setting.  On the other hand, as is stated in Theorem \ref{t:eps_reg_collapsed} we should expect the fiberd fundamental group $\Gamma_\delta(x)$ to be almost nilpotent of rank at most $1$.  Using that $M$ is simply connected and the behavior of its asymptotics, it is not hard to see that in this example we have that rank$(\Gamma_\delta(x_r))=0$ for all $\delta>0$ and $r$ sufficiently large.  Hence, though it does not satisfy the $\epsilon$-regularity theorem, we do have the expected drop in the fundamental group.

\subsection{Example \ref{ss:example2}: Singular Fibration of K3 Surface}\label{ss:example2}

In \cite{GrW} the authors constructed a family of Ricci flat K\"ahler metrics $(K3,g_j)$ on a K3 surface with elliptic fibrations over $\mathbb{C}P^1\cong S^2$,
\begin{equation}
f: K3\longrightarrow S^2.
\end{equation}
Away from $24$ singular points, the fibers of $f$ are tori, even almost isometrically with respect to the geometries on $(K3,g_j)$.  Each of the $24$ singular fibers is of Kodaira type 
$I_1$ which is the pinched torus, that is, there is a finite subset $\{q_1,\ldots q_{24}\}\subset S^2$ such that
\begin{equation}f^{-1}(q_{\ell})=I_1, \ \ell=1,\ldots, 24\, .
\end{equation}
Geometrically, the sequence of collapsing K3 geometries converge
\begin{align}
(K3,g_j)\xrightarrow{GH} (S^2,d_{\infty}),
\end{align}
where the metric $d_\infty$ is a smooth metric on the topological two sphere $S^2$ away from the $24$ singular points.  It may be checked, using the explicit coordinate expression for the metric $d_\infty$, that at every point of $(S^2,d_{\infty})$ the tangent cone is $\dR^2$.  Hence, for every $\epsilon>0$ we have for all $r>0$ sufficiently small and any $x\in S^2$ that
\begin{align}
d_{GH}\Big(B_{r}(x),B_r(0^2)\Big)<\epsilon r\, .
\end{align}
In particular, we are again in a position to test the hypothesis of the noncollapsing $\epsilon$-regularity Theorem \ref{t:eps_reg_noncollapsed}.  Again, we immediately see that if $x$ is one of the $24$ singular points, then the curvature of $(K3,g_j)$ blows up along $f^{-1}(x)$.  However, note that since $f^{-1}(x)\approx I_1$ is a torus with a pinched neck, we have that $\pi_1(f^{-1}(x))\cong\mathbb{Z}$.  With a little work it can be then seen that one should expect the same of the fibered fundamental group $\Gamma_\delta(x_j)$.  More precisely, for $x_j\in (K3,g_j)\to x$ and all $\delta>0$, we see that for large enough $j$ that $\Gamma_\delta(x_j)\cong \mathbb{Z}$.  Since we have collapsed two dimensions here, this is a drop from the maximal possible fundamental group, and again confirms the picture of Theorem \ref{t:eps_reg_collapsed}.

\subsection{Example \ref{ss:example3}: The Eguchi-Hanson Metric}\label{ss:example3}

We will give a sequence of collapsing Ricci flat K\"ahler manifolds. In this example, the fibered fundamental group attains the maximal rank, however the $\epsilon$-regularity fails due to a lack of Gromov-Hausdorff control.

Consider first a sequence of tori with flat metrics which collapse to  a point:
\begin{equation}
(T^{n-4},g_i)\xrightarrow{GH}pt\, ,
\end{equation}
where $n\geq 4$.  On the other hand let $(M^4, h_i)=(M^4, i^{-2}g_{EH})$ be a sequence of rescaled Eguchi-Hanson spaces with $\Ric_{h_i}\equiv 0$, then it is a standard fact that
\begin{equation}
(M^4, h_i)\xrightarrow{GH}(\dR^4/\mathbb{Z}_2,d_0),
\end{equation}
by letting $i\rightarrow\infty$.  Topologically we have that $M^4$ is the cotangent bundle over $S^2$, and in particular simply connected.

Now let us consider the Ricci flat spaces $(M_i^{n},\omega_i)\equiv(T^k\times M^4, g_i\oplus h_i)$.  Then we have that  
\begin{equation}
(M_i^{n},\omega_i,p_i)\xrightarrow{GH}(\dR^4/\mathbb{Z}_2,d_0,0^*)\, .
\end{equation}
In this sequence it is not hard to check that $\rank (\Gamma_{\delta_0}(p_i))=n-4$, which is maximal. However,
the curvature blows up near to the vertex.
This example shows both the necessity of Gromov-Hasudorff control and that the dimensional assumption in Theorem \ref{t:eps_reg_collapsed} is sharp.

\subsection{Example \ref{ss:small-conj}: Non-Collapse and Conjugate Radius}
\label{ss:small-conj}

We construct a family of non-collapsed metrics $(M^2,g_{\delta},p_{\delta})$ with nonnegative Gaussian curvature  such that when $\delta\rightarrow0$ the pointed Gromov-Hausdorff limit space is $(\dR^2,g_{\dR^2},0^2)$. However, there exists $q_{\delta}\in B_1(p_{\delta})$ such that  $\Conj_{g_{\delta}}(q_{\delta})\rightarrow0$
and in particular, $\InjRad_{g_{\delta}}(q_{\delta})\rightarrow0$.
This example shows what we can expect for the $\epsilon$-regularity in the context of sectional curvature or Ricci curvature only bounded from below. Even in the non-collapsed setting with smooth limit space, the conjugate radius and hence injectivity radius may go to zero at some point. Hence, in this paper, in the context of lower Ricci curvature we go for a uniform control on the {\textit{weak conjugate radius}} (see Definition \ref{weak-conj} for the precise definition). In Subsection \ref{ss:eps-reg-lower-Ricci}, we will prove the $\epsilon$-regularity in the context of Ricci curvature bounded from below.

Now we give the concrete construction of metric.
Consider the following $C^{1,1}$-warped product metrics on $(\dR^2,0^2)$. Fix some  $0<\delta<10^{-2}$, for each $\epsilon>0$, let
\begin{equation}
g_{\delta,\epsilon}=dr^2+f_{\epsilon}^2(r)d\theta^2
\end{equation}
with
\begin{equation}
f_{\delta,\epsilon}(r)=\begin{cases}
\epsilon\sin\frac{r}{\epsilon}, \ 0\leq r\leq \delta\cdot\epsilon,\nonumber\\ 
r\cos\delta+\epsilon(\sin\delta-\delta\cos\delta) ,\ r>\delta\cdot\epsilon.
\end{cases}
\end{equation}
Observe that $f_{\delta,\epsilon}(r)$ is a concave $C^{1,1}$-function with $\Lip(f_{\delta,\epsilon}'(r))\leq1/\epsilon$ for $0\leq r< \infty$, and thus $g_{\delta,\epsilon}$ is a $C^{1,1}$-Riemannian metric. 
Also it holds that for fixed $0<\delta<10^{-2}$,
\begin{equation}
(\dR^2,g_{\delta,\epsilon}, 0^{2})\xrightarrow{pGH}(C({S}^1), g_{\delta,0},0^2),
\end{equation}
where $g_{\delta,0}=dr^2+ r^2\cdot \cos^2\delta\cdot d\theta^2$ is a cone metric.

Now we prove that the conjugate radius tends to zero. To this end, we focus on the behavior of the Jacobi field with respect to the $C^{1,1}$-Riemannian metric $g_{\delta,\epsilon}$.
Fix  $\delta\cdot\epsilon<t_0<1$which will be determined later in the proof, let $\gamma:(-\infty,\infty)\rightarrow \dR^2$
be a unit speed geodesic through $0^2$ with \begin{equation}
d(\gamma(0),0^2)=t_0.
\end{equation}
Let $e_1(t)\perp\gamma'(t)$ be a parallel vector field along the geodesic $\gamma$.  
We will calculate the Jacobi field $J(t)=F(t)e_1(t)$ with $J(0)=0$, $J'(0)=e_1(\gamma(0))$. 
Let \begin{equation}t_1\equiv\min\{t>0|\gamma(t)\cap B_{\delta\cdot\epsilon}(0^2)\}=t_0-\delta\cdot\epsilon,\end{equation} 
and
\begin{equation}t_2\equiv\max\{t>0|\gamma(t)\cap B_{\delta\cdot\epsilon}(0^2)\}=t_0+\delta\cdot\epsilon,\end{equation}
 then
the Jacobi field $J(t)$ has the following expression,
\begin{equation}
F(t)=\begin{cases}
t,\ 0<t\leq t_1,\\
C_1\epsilon\sin\frac{t-t_1}{\epsilon}+C_2\cos\frac{t-t_1}{\epsilon},\ {t}_1\leq t\leq {t}_2,\\
F'({t}_2)(t-{t}_2)+F({t}_2), \ {t}_2\leq t<\infty.
\end{cases}\label{jacobi-expression}
\end{equation}
Consider the geodesic segment $\gamma|_{t_1\leq t\leq t_2}\subset B_{\delta\epsilon}(0^2)$. By the initial data 
\begin{equation}
F(t_1)=t_1,\ F'(t_1)=1,
\end{equation}
it holds that
\begin{equation}
C_1=1,\ C_2=t_1.
\end{equation}
That is, for $t_1\leq t\leq t_2$, 
\begin{equation}
F(t)=\epsilon\sin\frac{t-t_1}{\epsilon}+t_1\cos\frac{t-t_1}{\epsilon}>0.
\end{equation}
In particular,
\begin{equation}
F(t_2)=\epsilon\sin2\delta+t_1\cos2\delta>0,\ F'(t_2)=\cos2\delta-\frac{t_1}{\epsilon}\sin2\delta.
\end{equation}

Now we are in a position to estimate the conjugate radius for $g_{\epsilon}$. Let \begin{equation}t_1\equiv\sqrt{\epsilon},\end{equation} then $t_2=2\delta\cdot\epsilon+\sqrt{\epsilon}$ and for sufficiently small $\epsilon>0$, $F'(t_2)<0$.
Recall that $K(\gamma(t))=0$ for $t\geq{t}_2$, so it holds that
\begin{equation}
F(t)=F'({t}_2)(t-{t}_2)+F({t}_2).
\end{equation}
Let $t_{conj}>{t}_2$
such that
\begin{equation}
F(t_{conj})=0,
\end{equation}
and hence 
\begin{eqnarray}
t_{conj}&=&t_2+\frac{F(t_2)}{-F'(t_2)}\nonumber\\
&=&t_2+\frac{\epsilon\sin2\delta+\sqrt{\epsilon}\cos2\delta}{-\cos2\delta+\frac{1}{\sqrt{\epsilon}}\sin2\delta}\nonumber\\
&< &2\sqrt{\epsilon}.
\end{eqnarray}
The last inequality holds for sufficiently small $\epsilon>0$.
Therefore, $\gamma(t_{conj})$ is a conjugate point of $\gamma(0)$ and thus
\begin{equation}
\Conj_{g_{\delta,\epsilon}}(\gamma(0))< 2\sqrt{\epsilon},
\end{equation}
with $d(\gamma(0),0^2)=t_0=\delta\epsilon+\sqrt{\epsilon}$.

For fixed $\delta>0$, $\epsilon>0$,  
by standard smoothing calculations, for any $i>0$, there is a smooth Riemannian metric ${g}_{\delta,\epsilon}^{(i)}=dr^2+(f_{\delta,\epsilon}^{(i)}(r))^2d\theta^2$
such that the following holds:

$(i)$ $f_{\delta,\epsilon}^{(i)}\rightarrow f_{\delta,\epsilon}$ and thus ${g}_{\delta,\epsilon}^{(i)}\rightarrow g_{\delta,\epsilon}$ pointwise and the radial geodesic keeps the same in the sequence,

$(ii)$ $(f_{\delta,\epsilon}^{(i)})''\leq 0$ and thus
$ K_{{g}_{\delta,\epsilon}^{(i)}}=-\frac{f_{\delta,\epsilon}^{(i)})''}{(f_{\delta,\epsilon}^{(i)})}\in[0,\epsilon^{-2}+O(1/i)]$, 

$(iii)$ for almost every $x\in(\dR^2,0^2)$, $\lim\limits_{i\rightarrow\infty} K_{g_{\delta,\epsilon}^{(i)}}(x)=K_{{g}_{\delta,\epsilon}}(x)$. 
\\ Consider the 
 Jacobi equations 
with respect to the metrics $\bar{g}_{\delta,\epsilon,i}$,
\begin{equation}
F_i''=-K_iF_i.
\end{equation}
Classical result shows that 
\begin{equation}
\|F_i-F_0\|_{C^{0}}\rightarrow0,\end{equation}
and notice that $\|F_i\|_{C^{1,1}}\leq C(\delta,\epsilon)$ it is easy to conclude that
\begin{equation}
\|F_i-F_0\|_{C^{1}}\rightarrow0.
\end{equation}
Let $q_{\delta,\epsilon}\equiv\gamma_{\delta,\epsilon}(0)$ in the above construction,
\begin{equation}
\Conj_{{g}_{\delta,\epsilon}^{(i)}}(q_{\delta,\epsilon})<4\sqrt{\epsilon}.
\end{equation}
Choose $\epsilon=\delta^2$ and when $\delta\rightarrow0$, it holds that
\begin{equation}
(\dR^2,g_{\delta}\equiv{g}_{\delta,\delta^2},0^2)\xrightarrow{pGH}(\dR^2,g_{\dR^2},0^2),
\end{equation}
and for $q_{\delta}\equiv q_{\delta,\delta^2}$,  \begin{equation}
\Conj_{{g}_{\delta}}(q_{\delta,\delta^2})<4\delta\rightarrow0.
\end{equation}

\subsection{Example \ref{ss:example 5}: Isometric Nilpotent Group Action without Lines}\label{ss:example 5} 

Here we present an example  based on the calculations in \cite{Wei}. 
This example is used to show the necessity of the cone splitting ideas in the proof of our $\epsilon$-regularity theorem,  which is a new manner 
to prove the splitting given appropriate symmetry compared to
the classical line splitting arguments.  \\

More specifically, in \cite{Wei} a complete noncompact space $(M^n,g)$ was built with the following properties:
\begin{enumerate}
\item $\Ric_g > 0$,
\item There exists a free isometric action on $M^n$ by a simply connected nilpotent lie group $N^{n-k}$.
\item $M$ does not contain any lines.  Equivalently, due to $(1)$, $M^n$ does not split any $\dR$ factors.
\item $M^n/N^{n-k}$ is diffeomorphic, but not isometric, to $\dR^{k}$.
\end{enumerate}

Let us analyze these properties in the context of this paper.  Recall from the discussion of Section \ref{ss:intro_outline_proof} that our basic strategy in the proof of the $\epsilon$-regularity theorem is to construct new $\dR$-factors on the limit space of normal covers. 
For convenience, let us briefly recall the context by presenting the following toy model. Let $(M_i^n,g_i,p_i)$ satisfy $\Ric_{g_i}\geq-i^{-2}$ and the following diagram,
\begin{equation}
\xymatrix{(\widetilde{M}_i^n,\Gamma_i,\tilde{p}_i)\ar [d]^{\pi_{i}} \ar [rr]^{eqGH}&& (\widetilde{Y},\Gamma_{\infty},\tilde{y})\ar [d]^{\pi_{\infty}} \\
(M_i^n,p_i)\ar [rr]^{pGH}&&(Y,y),}
\end{equation}
where $\widetilde{M}_i^n$
is the universal cover of $M_i^n$ and assume that $\Gamma_i\equiv\pi_1(M_i^n)$ and $\rank(\Gamma_i)=m$.  If $Y$ is compact, then by the techniques of \cite{FY} $\widetilde{Y}$ has $m$-independent lines and thus $\widetilde{Y}\cong\dR^{m'}\times\widetilde{Z}$ with $m'\geq m$ and $\widetilde{Z}$ compact, see lemma \ref{homoline}.
If $Y$ is non-compact, the above line splitting fails, even in the context of nonnegative Ricci curvature.  In fact, let us now use the example of \cite{Wei} presented above to produce a counterexample. \\ 

Indeed, let $\Gamma_i\leq N^{n-k}$ be a sequence of increasingly dense co-compact lattices, and let $M_i^n\equiv M^n/\Gamma_i$, where $M^n$ is the example presented at the beginning of this subsection.  For this example we then have that
\begin{align}
&\widetilde M_i^n = M^n\xrightarrow{GH} \widetilde Y\equiv M^n\, ,\notag\\
&\Gamma_i\to \Gamma_\infty\equiv N^{n-k},\ \rank(\Gamma_i)=n-k\, ,\notag\\
&M^n_i\xrightarrow{GH} Y\, ,
\end{align}
where $Y$ is diffeomorphic to $\dR^{k}$  and thus noncompact.  But now of course $\widetilde Y= M^n$ does not contain any lines, and so does not split any $\dR$-factors.  In particular, this shows that in the general case, when $Y$ is noncompact one requires new arguments to produce lines, which is why we use a cone splitting argument for the proof of Theorem \ref{t:eps_reg_collapsed}.

\section{The Structure of Almost Nilpotent Groups}\label{s:almost-nilpotent}

The proof of Theorem \ref{t:eps_reg_collapsed} will require not only the fundamental group control of \cite{KW}, but various refinements which require the ability to pick geometrically compatible bases along with studying not only the fibered fundamental group but the structure of the deck transformations of other normal covers.  This section is devoted to developing all the algebraic tools which will be necessary to make these studies.  Roughly speaking, the main result of this section is to build polycyclic extensions of the lower central series of nilpotent subgroups of $\Gamma_{\delta}(p)$ which are compatible with the geometry.  Such extensions will play an important role in inductive arguments throughout the paper.

\subsection{Effective Generating Set of a Finite-index Subgroup}
\label{ss:effective-generating-set}

In this subsection, we focus on Schreier's Lemma (lemma 4.2.1 in \cite{schreier-generators}) and its applications. Schreier's Lemma explicitly gives a generating set of a subgroup, called {\textit{Schreier's generating set}}. We will prove some effective estimates on the above generating set. That is, if a subgroup $H\leq G$ has a priori controlled index $C<\infty$ in $G$, then each generator in Schreier's generating set has controlled length in terms of the generators of $G$. 

Although Schreier's Lemma works for any group and its subgroup, every group in this subsection is assumed to be finitely generated, which is enough for our geometric applications. 

\subsubsection{Alphabetical ordering in a finitely generated group}
In this paragraph, we define the alphabetical ordering which is a total ordering of a finitely generated group. This ordering is the foundation to build the effective generating set of a subgroup of controlled index.
We start with some basic notions.

\begin{definition}
Let $\Gamma$ be a finitely generated group with an ordered symmetric generating set $S=(s_1,\ldots,s_d)$ with $S^{-1}= S$. We define the following notions:

$(i)$ Let $g\in\Gamma$. If $g=s_{i_1}\cdot s_{i_2}\cdot\ldots\cdot s_{i_m}$ with $s_{i_j}\in S$ for $1\leq j\leq m$, then $s_{i_1}\cdot s_{i_2}\cdot\ldots\cdot s_{i_m}$ is called a {\textit{presentation}} of $g$ in terms of the elements in $S$. A specific presentation is denoted by the corresponding bold letter,
\begin{equation}
\bm{g}=s_{i_1}\cdot s_{i_2}\cdot\ldots\cdot s_{i_m}.
\end{equation}

$(ii)$
 Given a group element $g\in\Gamma$, a presentation $\bm{g}=s_{i_1}\cdot s_{i_2}\cdot\ldots\cdot s_{i_m}$ is called a {\textit{reduced presentation}} of length $m$ if $s_{i_j}\cdot s_{i_{j+1}}\neq e$ for all $1\leq j\leq m-1$ and denote $\ell_S(\bm{g})=m$. The set of all reduced 
presentations of $g\in\Gamma$ is denoted by $\mathcal{P}_g$, while the set of all reduced presentations in $\Gamma$ is denoted by $\mathcal{P}(\Gamma)$. Thus $\ell_S$ is a function from $\mathcal{P}(\Gamma)$ to $\mathbb{N}$.

\end{definition}

\begin{remark}
By definition, it happens that two different presentations $\bm{g}_1$ and $\bm{g}_2$ corresponds to the same group element in $\Gamma$.
\end{remark}

We will give a total ordering among all reduced presentations of a group. 
\begin{definition}
Let $\Gamma$ be a finitely generated group with an ordered symmetric generating set 
$S=(s_1,\ldots,s_d)$ with $S^{-1}=S$.  We define an ordering $\prec_p$ for any two different reduced presentations in $\mathcal{P}(\Gamma)$:

$(i)$ If $\bm{g}\in\mathcal{P}(\Gamma)$ and $\bm{g}\neq e$, then $e\prec_p \bm{g}$.

$(ii)$ $s_1\prec_p s_2\prec_p\ldots\prec_p s_d$.

$(iii)$ Let $\bm{g}_1=\prod\limits_{\nu=1}^ms_{i_{\nu}}$ and $\bm{g}_2=\prod\limits_{\nu=1}^ns_{j_{\nu}}$ be two reduced presentations. We say $\bm{g}_1\prec_p \bm{g}_2$ if either 

$(a)$ $\ell_S(\bm{g}_1)<\ell_S(\bm{g}_2)$, i.e. $m<n$, or

$(b)$ $\ell_S(\bm{g}_1)=\ell_S(\bm{g}_2)=m$, then there exists $1\leq \nu_0\leq m$ such that $i_{\nu_0}<j_{\nu_0}$ and $i_{\nu}\leq j_{\nu}$ for all $1\leq \nu\leq \nu_0$.
\end{definition}

\begin{remark}
The above ordering $\prec_p$ in the set of all reduced words is equivalent to the ordering in the set of all presentations.
\end{remark}

\begin{lemma}\label{basic-properties-prec_p} The following properties hold:

$(i)$ Let $\Gamma$ and $S$ as above, for any two reduced presentations $\bm{g}_1,\bm{g}_2\in\mathcal{P}(\Gamma)$, then exactly one of $\bm{g}_1=\bm{g}_2$, $\bm{g}_1\prec_p \bm{g}_2$ and $\bm{g}_2\prec_p \bm{g}_1$ is true.

$(ii)$ If $\bm{g}_1\prec_p \bm{g}_2$, $\bm{g}_2\prec_p \bm{g}_3$, then $\bm{g}_1\prec_p \bm{g}_3$.

$(iii)$ If $\bm{g}_1\prec_p \bm{g}_2$ and there is no cancellation between $\bm{g}_3$ and $\bm{g}_2$, then $\bm{g}_3\cdot \bm{g}_1\prec_p  \bm{g}_3\cdot \bm{g}_2$ and $\bm{g}_1\cdot \bm{g}_3\prec_p  \bm{g}_2\cdot \bm{g}_3$.
\end{lemma}
\begin{proof}
The proof immediately follows from the definition. So we omit the proof here.\end{proof}

With the above ordering, we can define the canonical presentation of a group element, and then we are able to provide a total ordering among the group elements. 
\begin{definition}
Let $\Gamma$ be a finitely generated group with an ordered symmetric generating set (not containing the identity)
\begin{equation}
S^{-1}=S=(s_1, s_2,\ldots, s_d).
\end{equation}
We call a presentation $\bm{g}=s_{i_1}\cdot s_{i_2}\cdot\ldots\cdot s_{i_m}\in\mathcal{P}_g$ is a {\textit{canonical presentation}}  of $g$  if this presentation satisfies the following:

$(i)$ For any presentation $\bm{g}=s_{j_1}\cdot s_{j_2}\cdot\ldots\cdot s_{j_{m'}}\in\mathcal{P}_g$ with $s_{j_{\nu}}\in S$, we have $m\leq m'$.

$(ii)$ Denote by $\mathcal{M}_g\subset\mathcal{P}_g$ the set of all reduced presentations of $g$ such that the length of each presentation in $\mathcal{M}_g$ is equal to $m$, then the presentation $\bm{g}=s_{i_1}\cdot s_{i_2}\cdot\ldots\cdot s_{i_m}$ is the first member in $\mathcal{M}_g$ w.r.t. the ordering $\prec_p$.
\end{definition}

\begin{remark}
The above canonical presentation is well-defined. In fact, for the above fixed $m$, $\#(\mathcal{M}_g)\leq d^m$
and thus we can pick up the first element in $\mathcal{M}_g$ w.r.t. the ordering $\prec_p$.
\end{remark}

\begin{definition}Let $g\in\Gamma$, then $\length_S(g)$ is the defined by the length of its canonical presentation.
\end{definition}

\begin{example}
Suppose we have $4$ different presentations of $g$: $s_1^3s_4^2s_2^1$, $s_2s_1s_4s_3$, $s_2s_1s_3s_5$, $s_2^2s_4s_6$. Then by definition,  $\bm{g}=s_2s_1s_3s_5$ is the canonical presentation, $\length_S(g)=4$, and $s_2s_1s_3s_5\prec_p s_2s_1s_4s_3\prec_p s_2^2s_4s_6\prec_p s_1^3s_4^2s_2^1$. 
\end{example}

\begin{lemma}
Let $\Gamma$ as above, then each group element $g\in\Gamma$ has a unique canonical presentation.
\end{lemma}
\begin{proof}
Given $g\in\Gamma$. Let $\ell_s:\mathcal{P}(\Gamma)\rightarrow\mathbb{N}$ be the function which maps each reduced presentation to its length. 
First, the image set $\ell_S(\mathcal{P}_g)$ is a subset of $\mathbb{N}$, and thus the set $\ell_S(\mathcal{P}_g)$ has a finite minimum in $\mathbb{N}$. Therefore, $g$ has a presentation satisfying $(i)$.

Let $m\equiv\length_S(g)$, then  $\#(\mathcal{M}_g)\leq d^m$.
Immediately, there exists a unique member in $\mathcal{M}_g$ which satisfies $(ii)$.
\end{proof}

\begin{lemma}\label{smaller-canonical}
Let $s_{i_1}\cdot s_{i_2}\cdot\ldots\cdot s_{i_m}$ be the canonical presentation of some element $g\in\Gamma$, then for each $1\leq k\leq m-1$, both $\prod\limits_{\nu=1}^k s_{i_{\nu}}$ and $\prod\limits_{\nu=k+1}^m s_{i_{\nu}}$ are canonical presentations of the corresponding group elements.\end{lemma}

\begin{proof}
First, we prove that $\bm{g}_k=\prod\limits_{\nu=1}^k s_{i_{\nu}}$ is the canonical presentation of $g_k$ for each $1\leq k\leq m$  by contradiction. Suppose that there is some $1\leq k_0\leq m-1$ such that $\bm{g}_{k_0}\equiv\prod\limits_{\nu=1}^{k_0}s_{i_{\nu}}$ is not a canonical presentation of the element $g_{k_0}$. Let $\bm{g}_{k_0}=\prod\limits_{\nu=1}^{k_1}s_{j_{\nu}}$
be the canonical presentation of $g_{k_0}$ such that $\prod\limits_{\nu=1}^{k_1}s_{j_{\nu}}\prec_p \prod\limits_{\nu=1}^{k_0}s_{i_{\nu}}$ with $k_1\leq k_0$. Since there is no cancellation between the two presentations $\prod\limits_{\nu=1}^{k_0}s_{i_{\nu}}$ and $\prod\limits_{\nu=k_0+1}^{m}s_{i_{\nu}}$, property $(iii)$ in Lemma \ref{basic-properties-prec_p} implies the following inequality
\begin{equation}
\prod\limits_{\nu=1}^{k_1}s_{j_{\nu}}\cdot \prod\limits_{\nu=k_0+1}^{m}s_{i_{\nu}}\prec_p \prod\limits_{\nu=1}^{k_0}s_{i_{\nu}}\cdot\prod\limits_{\nu=k_0+1}^{m}s_{i_{\nu}}= \prod\limits_{\nu=1}^{m}s_{i_{\nu}}.
\end{equation}
Consequently,
$\prod\limits_{\nu=1}^{m}s_{i_{\nu}}$ is not the canonical presentation of $g$, which gives the contradiction.

We have proved that $\prod\limits_{\nu=1}^k s_{i_{\nu}}$ is the canonical presentation of $g_k$ for each $1\leq k\leq m$. Immediately, 
$\prod\limits_{\nu=k+1}^m s_{i_{\nu}}$
is the canonical presentation of $h_k$.
Otherwise, by the same arguments,
the canonical presentation of $g$ would be strictly before $\prod\limits_{\nu=1}^m s_{i_{\nu}}$, which is impossible.
\end{proof}

Now we define a total ordering in a finitely generated group.
\begin{definition}
Let $\Gamma$ and $S$ as above, there is an ordering relation $\prec$ called the \textit{alphabetical ordering}   such that for every two {\textit{different}} elements, the following holds:

$(i)$ $e\prec g$ for any $g\neq e$.

$(ii)$ $s_1\prec s_2\prec\ldots\prec s_d$.

$(iv)$ Let $\bm{g}=s_{i_1}\cdot s_{i_2}\cdot\ldots\cdot s_{i_m}$ and $\bm{h}=s_{j_1}\cdot s_{j_2}\cdot\ldots\cdot s_{j_n}$ be the unique canonical presentations of $g,h\in\Gamma$, then $g\prec h$ iff $s_{i_1}\cdot s_{i_2}\cdot\ldots\cdot s_{i_m}\prec_p s_{j_1}\cdot s_{j_2}\cdot\ldots\cdot s_{j_n}$.
\end{definition}

\subsubsection{The canonical transversal of a subgroup}

In this subsection, we introduce the most important technical tool in this section. The canonical transversal gives a unique ``minimal" representative of each coset, which plays a crucial role 
in controlling the behavior of a  subgroup. We begin with the following basic definitions and conventions.

\begin{definition}
Let $\Gamma$ be a group and let $N\leq\Gamma$. Given $g\in\Gamma$, \begin{equation}g\cdot N\equiv\{g\cdot h| h\in N\},\ N\cdot g\equiv\{h\cdot g| h\in N\},\end{equation} are called a \textit{left coset} and a \textit{right coset} respectively.
The number of left cosets, which is equal to the number of the right cosets, is called the index of $N$ in $\Gamma$, and denoted by $[\Gamma:N]$.
\end{definition}
\begin{remark}
If $N\lhd\Gamma$, then both
\begin{equation}
\Gamma/N\equiv\{g\cdot N|\forall g\in\Gamma\}
\end{equation}
and 
\begin{equation}
N\setminus \Gamma\equiv\{N\cdot g|\forall g\in\Gamma\}
\end{equation}
have a group structure and they are isomorphic. Moreover, for all $g\in\Gamma$, we have that $g\cdot N=N\cdot g$.
\end{remark}

\begin{definition}
Let $\Gamma$ be a finitely generated group and $N$ be a subgroup of $\Gamma$ with $[\Gamma:N]=C<\infty$.  Let $\{g_j\cdot N\}_{j=1}^C$ be the left cosets such that $\bigcup\limits_{j=1}^Cg_j\cdot N=\Gamma$ and $(g_i\cdot N)\cap (g_j\cdot N)=\emptyset$ for $i\neq j$. A {\textit{transversal}} or {\textit{section}} of $N$ in $\Gamma$ is a selection of representatives, i.e. it is a surjective map  
\begin{equation}F:\Gamma\rightarrow\{g_1,\ldots,g_C\},
\end{equation}
such that $g\in g_j\cdot N$ if and only if
$F(g)\equiv g_j.$
\end{definition}

With the definition of the transversal of a subgroup, we will explain Schreier's explicit construction of a generating set of a subgroup. Roughly speaking, each transversal of a subgroup gives an algorithm which finds a generating set of the subgroup.
\begin{lemma}
[O. Schreier]\label{schreier-lemma} Let $S\equiv\langle s_1,\ldots, s_d\rangle$ be a generating set of the finitely generated group $\Gamma$. Assume that $N\leq\Gamma $ with $[\Gamma:N]=C<\infty$. Let $F$ be a transversal of $N$ in $\Gamma$ with $\Image(F)=\{g_1,\ldots,g_C\}$,  then 
\begin{equation}
\bar{S}\equiv\{(F(g_j\cdot s_i))^{-1}\cdot s_i\cdot g_j|g_j\in\Image(F), s_i\in S\}
\end{equation}
is a generating set of $N$.
\end{lemma}

\begin{remark}
In the above notations, obviously,
\begin{equation}
\#(\bar{S})\leq C^2\cdot d.
\end{equation}

\end{remark}

Now we give the existence and uniqueness of the canonical transversal with some effective control.
\begin{lemma}[Canonical transversal]\label{canonical-transversal}
Let $\Gamma$, $S$ and $N$ as above. Let $\{N_j\}_{j=1}^C$ be the left cosets of $N$ in $\Gamma$, then there is a unique transversal (called the canonical transversal)\begin{equation}F_c:\Gamma\rightarrow\Image(F)=\{g_1,\ldots,g_C\}\end{equation} such that

$(i)$ $g_j$ is the first element in $N_j$ w.r.t. the alphabetical ordering $\prec$,

$(ii)$ $\length_S(g_j)\leq C$
for each $1\leq j\leq C$.

\end{lemma}

\begin{proof}

First, take the first element in $N_j$ for each $1\leq j\leq C$.
This choice has uniquely determined the transversal $F_c$. Now it suffices to check property $(ii)$ for the transversal $F_c$.
Argue by contradiction and suppose $(ii)$ fails. That is, there exists $g_{j_0}\in\Image(F_c)$ with the canonical presentation in terms of the elements in $S$,
\begin{equation}
{\bm{g}}_{j_0}=s_{l_1}\cdot s_{l_2}\cdot\ldots\cdot s_{l_{C+m}},\ m\geq 1.\end{equation}
For each 
$ 1\leq k\leq C+m$, denote\begin{equation}
w_k\equiv\prod\limits_{\nu=C+m-k+1}^{C+m} s_{l_{\nu}}, \label{smaller-tail}
\end{equation}
and \begin{equation}h_k\equiv\prod\limits_{\nu=1}^{C+m-k} s_{l_{\nu}}.\label{smaller-head}\end{equation}
Clearly, $g_{j_0}=h_k\cdot w_k$.
By Lemma \ref{smaller-canonical}, for each  $1\leq k\leq C+m$, equation (\ref{smaller-tail}) and equation (\ref{smaller-head}) gives the canonical presentation of $w_k$ and $h_k$ respectively. In particular, $\length_S(w_k)=k$ and $\length_S(h_k)=C+m-k$.
We will prove the following claim:

\vspace{0.5cm}

{\bf{Claim.}} $w_k\in\Image(F_c)$ for each $1\leq k\leq C+m$. 

\vspace{0.5cm}

We argue by contradiction. Suppose $w_{k_0}\not\in\Image(F_c)$ for some $1\leq k_0\leq C+m$, then $F_c(w_{k_0})\prec w_{k_0}$. The contradiction would arise if we proved that 
\begin{equation}h_{k_0}\cdot F_c(w_{k_0})\prec g_{j_0}.\label{elements-comparison}\end{equation}
 In fact, notice that, $ F_c(w_{k_0})\cdot N=w_{k_0}\cdot N$ which leads to \begin{equation}h_{k_0}\cdot F_c(w_{k_0})\cdot N=h_{k_0}\cdot w_{k_0}\cdot N= g_{j_0}\cdot N.\end{equation} Consequently, $g_{j_0}$ is not the first element in the left coset $g_{j_0}\cdot N$. Contradiction.

Now we are in a position to prove inequality (\ref{elements-comparison}). Let 
\begin{equation}
\bm{h}_{k_0}=\prod\limits_{\nu={1}}^{C+m-k_0}s_{l_{\nu}},\
{\bm{w}}_{k_0}\equiv\prod\limits_{\nu=C+m-k_0+1}^{C+m} s_{l_{\nu}},\
\bm{F_c(w_{k_0})}=\prod\limits_{\nu=1}^{k'}s_{j_{\nu}}.
\end{equation}
be the canonical presentations. By assumption, $k'\leq k_0$.
Property $(iii)$ in Lemma \ref{basic-properties-prec_p} shows that
\begin{equation}
\bm{h}_{k_0}\cdot\bm{F_c(w_{k_0})}=\Big(\prod\limits_{\nu=1}^{C+m-k_0}s_{l_{\nu}}\Big)\cdot\Big(\prod\limits_{\nu={1}}^{k'}s_{j_{\nu}}\Big)\prec_p \Big(\prod\limits_{\nu=1}^{C+m-k_0}s_{l_{\nu}}\Big)\cdot\Big(\prod\limits_{\nu=C+m-k_0+1}^{C+m} s_{l_{\nu}}\Big)=\bm{g}_{j_0}.
\end{equation}
Since the canonical presentation $h_{k_0}\cdot F_c(w_{k_0})$ is before or equal to the presentation $(\prod\limits_{\nu=1}^{C+m-k_0}s_{l_{\nu}})\cdot(\prod\limits_{\nu={1}}^{k'}s_{j_{\nu}})$, the canonical presentation of $h_{k_0}\cdot F_c(w_{k_0})$ is before $\prod\limits_{\nu=1}^{C+m} s_{l_{\nu}}$.
By assumption, $\prod\limits_{\nu=1}^{C+m} s_{l_{\nu}}$ is the canonical presentation of $h_{k_0}\cdot w_{k_0}$, then it holds that,
\begin{equation}
h_{k_0}\cdot F_c(w_{k_0})\prec h_{k_0}\cdot w_{k_0}=g_{j_0}.
\end{equation}
We have proved the Claim.

\vspace{0.5cm}

With the Claim, we obtain a sequence of different elements $\{w_1,\ldots, w_{C+m}\}\subset\Image(F_c)$. However, $\#\Image(F_c)=C<C+m$. This contradiction shows that property $(ii)$
has to be true.
\end{proof}

Our main result of this section is the following Proposition, which gives an effective estimate for the Schreier's generating set of a subgroup with a priori controlled index.

\begin{proposition}\label{effective-generating-set}
Let $\Gamma$ be a finitely generated group with a generating set $S=\{s_1,s_2,\ldots,s_d\}$. Let $N\leq \Gamma$ with $[\Gamma:N]=C<\infty$. Then $N$ has a generating set $\bar{S}$ such that 

$(i)$ $\#(\bar{S})\leq C^2\cdot d$, 

$(ii)$
$\length_S(\bar{s})\leq 2C+1$ for each $\bar{s}\in\bar{S}$.

\end{proposition}

\begin{proof}The proof is the combination of lemma \ref{schreier-lemma} and Lemma \ref{canonical-transversal}. In fact,
let $\bar{S}$ be the generating set from lemma \ref{schreier-lemma} with respect to the canonical transversal $F_c$ of $N$, then the remark after that lemma shows that
\begin{equation}
\#(\bar{S})\leq C^2\cdot d.
\end{equation}
Let $\bar{s}\in\bar{S}$, then there exists $g_{j_0}\in\Image(F_c)$
and $s_i\in S$ such that
\begin{equation}
\bar{s}=(F_c(s_i\cdot g_{j_0}))^{-1}\cdot s_i\cdot g_{j_0}.
\end{equation}
Lemma \ref{canonical-transversal} gives that
\begin{equation}
\length_S(F_c(s_i\cdot g_{j_0}))\leq C, \ \length_S(g_{j_0})\leq C,
\end{equation}
which implies that
\begin{equation}
\length_S(\bar{s})\leq C+1+C=2C+1.
\end{equation}
The proof is complete.
\end{proof}

In the above context, if $N$ is a normal subgroup, then we can obtain some better properties on the canonical transversal.

\begin{lemma}\label{shortening-words} Let $\Gamma$ and $S$ as above. 
Assume that $N\lhd \Gamma$ and let $F_c$ be the canonical transversal with respect to the quotient group $\Gamma/N$ (group of left cosets). Given $w\in\Gamma$
with the canonical presentation in terms of the elements in $S$,
\begin{equation}\bm{w}=s_{l_1}\cdot\ldots\cdot s_{l_k}.\end{equation} For any $1\leq j\leq k$, let $u\equiv s_{l_1}\cdot\ldots\cdot s_{l_j}$ and $v\equiv s_{l_{j+1}}\cdot\ldots\cdot s_{l_k}$. Then $w\in\Image(F_c)$ implies that $u\in\Image(F_c)$ and $v\in\Image(F_c)$.
\end{lemma}

\begin{proof}
Lemma \ref{canonical-transversal} shows that $v\in\Image(F_c)$. So it suffices to show $u\in\Image(F_c)$. Suppose that $u\not\in\Image(F_c)$, then $F_c(u)\prec u$. Then by the same arguments of the proof of the Claim in Lemma \ref{canonical-transversal}, \begin{equation}F_c(u)\cdot v\prec u\cdot v.\label{ordering-contradiction}\end{equation}
The definition of $F_c$ gives that $F_c(u)\cdot N= u\cdot N$.
Since $N\lhd\Gamma$, it holds that $v\cdot N=N\cdot v$ and then
\begin{equation}
F_c(u)\cdot v\cdot N=F_c(u)\cdot N\cdot v=u\cdot N\cdot v=u\cdot v\cdot N.
\end{equation}
The above equation shows that $F_c(u)\cdot v\cdot N$
and $u\cdot v\cdot N$ belong to the same left coset, and thus inequality (\ref{ordering-contradiction})
contradicts to the assumption that $u\cdot v$ is the first element in 
$u\cdot v \cdot N$. Therefore, $v\in\Image(F_c)$.
\end{proof}

\begin{corollary}\label{generator-is-minimal}
Let $\Gamma$, $S$ and $N$ as Lemma \ref{shortening-words}. Let $F_c$ be the minimal transversal with respect to the quotient $\Gamma/N$.  Given $w\in\Gamma$
with the canonical presentation in terms of the elements in $S$,
$\bm{w}=s_{l_1}\cdot\ldots\cdot s_{l_k}$. Assume that $w\in\Image(F_c)$, then for every $i\leq j$,
\begin{equation}
s_{l_i}\cdot s_{l_{i+1}}\cdot\ldots\cdot s_{l_j}\in\Image(F_c).
\end{equation}
In particular, $s_{l_i}\in\Image(F_c)$ for all $1\leq i\leq k$.
\end{corollary}
\begin{proof}
Lemma \ref{shortening-words} shows that 
$\bar{u}\equiv s_{l_1}\cdot s_{l_{i+1}}\cdot\ldots\cdot s_{l_j}\in\Image(F_c)$. Applying Lemma \ref{shortening-words} to $\bar{u}$, we have that $u\equiv s_{l_i}\cdot s_{l_{i+1}}\cdot\ldots\cdot s_{l_j}\in\Image(F_c)$.
\end{proof}

\subsection{Basics in Commutators Calculus}
\label{ss:commutators}

This subsection is to introduce some useful facts in the computations of commutators. 

\begin{definition}
Let $G$ be any arbitrary group with a generating set $B$. The sets of weighted basic commutators are inductively defined as follows:

$(i)$ $\mathcal{C}_0(B)\equiv B$.

$(ii)$ $\mathcal{C}_k(B)\equiv\{[g_k,s_i]|g_k\in \mathcal{C}_{k-1}(B), s_i\in B\}$.

\end{definition}

The following lemma is standard in commutators calculus which  inductively computes the lower commutators subgroup from higher commutators subgroup.

\begin{lemma}[\cite{funda}]\label{graded-commutators-subgroup}
Let $G$ be any arbitrary group with a generating set $B$, and let $G_0\equiv G$. Then for each $s\in\mathbb{N}_+$, $G_s\equiv[G_{s-1},G]$ can be generated by $\mathcal{C}_{s}(B)$ and $G_{s+1}$.
\end{lemma}

\begin{corollary}\label{generating-lower-central-series}
Let $N$ be a nilpotent group with a generating set $B$ and $\Step (N)=c_0$. Then $N_s\equiv[N_{s-1},N]$ has a generating set
$\bigcup\limits_{k=s}^{c_0}\mathcal{C}_{k}(B)$.
\end{corollary}

\begin{proof}
Since $N$ is nilpotent with $\Step(N)=c_0$, it holds that
\begin{equation}
N_{c_0}=\{e\},\ \mathcal{C}_{c_0}(B)=\{e\}.
\end{equation}
By lemma \ref{graded-commutators-subgroup},
$N_{c_0-1}$ is generated by $\mathcal{C}_{c_0-1}(B)$. Repeating lemma \ref{graded-commutators-subgroup}, we obtain that for each $1\leq s\leq c_0$
$N_s$ is generated by $\bigcup\limits_{k=s}^{c_0}\mathcal{C}_k(B)$.
\end{proof}

We finish this subsection by constructing a subgroup of controlled-index in a finitely-generated nilpotent group. 
This construction will be used in the proof of Theorem \ref{t:KW_almost_nilpotent}.  
\begin{lemma}\label{power-controlled-index}
Let $N$ be a finitely generated nilpotent group with a generating set $B_0=\{\sigma_{0,1},\ldots,\sigma_{0,d_0}\}$ and $\Step(N)=c$.
For each $1\leq k\leq c-1$, let $\mathcal{C}_k(B)=\{\sigma_{k,1},\ldots,\sigma_{k,d_k}\}$ be the set of basic commutators of weight $k$.
For each $0\leq k\leq c-1$, let \begin{equation}a_k=(a_{k,1},\ldots,a_{k,d_k})\end{equation}
be a $d_k$-tuple in $\mathbb{N}_+^{d_k}$ and
let 
\begin{equation}
N_{a_k}\equiv\langle\sigma_{k,1}^{a_{k,1}},\ldots,\sigma_{k,d_k}^{a_{k,d_k}}\rangle,
\end{equation}
then each $g\in N$ can be written in the following standard form
\begin{equation}
g=\Big(\prod\limits_{k=0}^{c-1}\prod\limits_{j=1}^{d_k}\sigma_{k,j}^{b_{k,j}}\Big)\cdot g_a,\ g_a\in N_a\equiv\prod\limits_{k=0}^{c-1}N_{a_k},\label{standard-form}
\end{equation}
where $0\leq b_{k,j}< a_{k,j}$ for all $0\leq k\leq c-1$ and $1\leq j\leq d_k$. In particular,
\begin{equation}
[N:N_a]\leq \prod\limits_{k=0}^{c-1}\prod\limits_{j=1}^{d_k}a_{k,j}.
\end{equation}

\end{lemma}

\begin{proof}
We can argue by induction on the step of $N$. Consider $N\rhd N_1$. Obviously, the statement of the Proposition holds for $\Step(N)=1$, namely, $N$ is a finitely generated abelian group. Assume that the statement is true for all finitely generated nilpotent group of step $\leq c-1$. Now we prove that for a finitely generated nilpotent group $N$ with $\Step(N)=c$. 

The notations are the same as those in the statement of the Proposition.
First, claim that each element $g\in N$ can be presented as follows: for some $(r_{0,1},\ldots, r_{0,d})\in\mathbb{Z}^{}$ such that
\begin{equation}
g=\sigma_{0,1}^{r_{0,1}}\cdot\ldots\cdot\sigma_{0,d_0}^{r_{0,d_0}}\cdot g_1, \ g_1\in N_1.\label{ordering}
\end{equation}
In fact, for any $i<j$ and $\bar{g}\in N$, it holds that
\begin{equation}
\sigma_{0,j}\cdot\bar{g}\cdot\sigma_{0,i}=\sigma_{0,i}\cdot\sigma_{0,j}\cdot\bar{g}\cdot\bar{g}_1,
\end{equation}
where $g_1\equiv[(\sigma_{0,j}\cdot\bar{g})^{-1},\sigma_{0,i}^{-1}]\in N_1$.
Hence, for any $g=\sigma_{0,l_1}\cdot\sigma_{0,l_2}\cdot\ldots\cdot\sigma_{0,l_m}\in N$, we can rearrange the generators such that equation (\ref{ordering}) holds. We have proved the claim. Furthermore,
for each $1\leq j\leq d_0$, there exists a unique $0\leq b_{0,j}<a_{0,j}$ and $m_{0,j}\in\mathbb{Z}$ such that 
\begin{equation}
r_{0,j}=m_{0,j}\cdot a_{0,j}+b_{0,j}.
\end{equation}
Simple inductive arguments shows that 

\begin{equation}
g=\sigma_{0,1}^{b_{0,1}}
\cdot\ldots\cdot
\sigma_{0,d_0}^{b_{0,d_0}}\cdot g_1'\cdot g_a',\label{refined-ordering}\end{equation} where 
$ g_1'\in N_1, \ g_a'\in N_a$.

Now we proceed to prove that $g\in N$ has the standard form in equation (\ref{standard-form}). We present $g$ in terms of the ordered generators as in equation (\ref{refined-ordering}). Since $N_1=[N,N]$ satisfies $\Step(N_1)=c-1$, by the induction hypothesis, 
\begin{equation}
g_1'=\Big(\prod\limits_{k=1}^{c-1}\prod\limits_{j=1}^{d_k}\sigma_{k,j}^{b_{k,j}}\Big)\cdot g_a'',\ g_a''\in \prod\limits_{k=1}^{c-1}N_{a_k}.
\end{equation}
Then equation (\ref{standard-form}) follows.
\end{proof}

\subsection{Refinement of the Lower Central Series of a Nilpotent Group}

\label{ss:refinement-lcs}

We prove in this subsection the main technical result in Section \ref{s:almost-nilpotent}. For a finitely generated nilpotent group $N$ with a generating set $S$, we will give an effective and geometrically compatible refinement of the lower central series. This effective refinement guarantees that the generator of each $\mathbb{Z}$-factor group has controlled length with respect to $S$, which plays a fundamental role in the inductive arguments throughout the paper.
 First, recall the following standard isomorphism lemma in group theory.
\begin{lemma}
\label{isomorphism-lemma}\label{pre-image-lemma}Let $A$ and $B$ be two groups. If $f:A\rightarrow B$ is a surjective homomorphism, then the following holds:

$(i)$ For each $B_1\leq B$, the pre-image $f^{-1}(B_1)$ is a group.

$(ii)$ If $B_1\lhd B_2$, then
$f^{-1}(B_1)\lhd f^{-1}(B_2)$. Moreover, $f^{-1}(B_2)/f^{-1}(B_1)
\cong B_2/B_1$.

$(iii)$ let $A_1\leq A_2\leq  A$, then
\begin{equation}
[f(A_2):f(A_1)]\leq [A_2:A_1].
\end{equation}

\end{lemma}

With the algebraic preparations in the above sections, we are ready to show the main result of this section.
Throughout this subsection, we will assume that $N$ is a finitely generated nilpotent group with $\rank(N)=m$ and $\Step(N)=c_0$.
Let $N\equiv\langle B\rangle$ with $B\equiv\{g_1, g_2,\ldots, g_{\bar{d}}\}$ and $B^{-1}=B$. Let
\begin{equation}
N\equiv N_0\rhd N_1\rhd\ldots\rhd N_{c_0-1}\rhd N_{c_0}=\{e\}\label{lower-central-series}
\end{equation}
be the lower central series with $N_{s}\equiv[N_{s-1},N]$ for every $1\leq s\leq c_0$. It is standard that for each $1\leq s\leq c_0$, $N_s$ is normal in $N$ and $N_{s-1}/N_s$
is a finitely generated abelian group. Denote $n_s\equiv\rank(N_{s-1}/N_{s})$. Obviously,
 $\sum\limits_{s=1}^{c_0}n_s=m$.

\begin{theorem}\label{t:refinement-of-lower-central-series} In the above notations,
there exists a refinement of the lower central series (\ref{lower-central-series}) such that for each $1\leq s\leq c_0$ the following properties hold:

$(i)$ There is a normal series
\begin{equation}
N_{s-1}\rhd N_{s,n_s}\rhd N_{s,n_s-1}\rhd N_{s,n_s-2}\rhd\ldots\rhd N_{s,1}\rhd N_{s,0}\rhd N_{s}
\end{equation}
 such that $N_{s,k}/N_{s,k-1}\cong\mathbb{Z}$ for $1\leq k\leq n_s$, while $N_{s-1}/N_{s,n_s}$ and $N_{s,0}/N_s$ are finite abelian groups.

$(ii)$ There exists a subset \begin{equation}\{\sigma_{s,1},\sigma_{s,2},\ldots,\sigma_{s,n_s}\}\subset N_{s-1}\label{good-cyclic-factors}\end{equation} such that for each $1\leq k\leq n_s$, 
\begin{equation}\length_{B}(\sigma_{s,k})\leq 3\cdot 2^{c_0}-2.\end{equation} 
 Moreover, the quotient group $N_{s,k}/N_{s,k-1}$ is generated by $\sigma_{s,k}\cdot N_{s,k-1}$. In particular, $\langle\sigma_{s,k}\cdot N_{s,k-1}\rangle\cong\mathbb{Z}$. 
\end{theorem}
 
\begin{proof}
At the beginning of the proof, we clarify the notations which will be used in the proof. We have assumed that the finitely generated nilpotent group $N$ has a symmetric generating set $B$, then Lemma \ref{generating-lower-central-series} shows that
$N_{s-1}$ has a generating set $B_{s-1}\equiv\bigcup\limits_{k=s-1}^{c_0}\mathcal{C}_k(B)$. By straightforward calculations, for each $b_{s-1}\in B_{s-1}$, \begin{equation}\length_B(b_{s-1})\leq 3\cdot 2^{c_0}-2.\label{length-estimate-higher-generating-set}\end{equation}
Let \begin{equation}
\pr_s:N_{s-1}\rightarrow N_{s-1}/N_s
\end{equation}
be the natural quotient 
homomorphism and denote $A_s\equiv N_{s-1}/N_s$.
Since $A_s$ is a finitely generated abelian group with $\rank(A_s)=n_s$, then $A_s\rhd\Tor(A_s)$ with $A_s/\Tor(A_s)\cong\mathbb{Z}^{n_s}$, where $\Tor(A_s)$ is the torsion subgroup of $A_s$.
Denote $G_{s,0}\equiv\pr_s^{-1}(\Tor(A_s))$, by  Lemma \ref{pre-image-lemma}, $G_{s,0}\lhd N_{s-1}$ and
\begin{equation}
N_{s-1}/G_{s,0}=\pr_s^{-1}(A_s)/\pr_s^{-1}(\Tor(A_s))\cong A_{s}/\Tor(A_s)\cong\mathbb{Z}^{n_s}.
\end{equation}
Let $\{\bar{v}_1,\ldots,\bar{v}_{n_s}\}$
be a  sequence of elements in $N_{s-1}/G_{s,0}$ such that
\begin{equation}
\langle\bar{v}_1,\ldots,\bar{v}_k\rangle\cong\mathbb{Z}^{k},\ \forall 1\leq k\leq n_s.\label{Z^k-generators}
\end{equation}
Let \begin{equation}\pi_s: N_{s-1}\longrightarrow N_{s-1}/G_{s,0}\end{equation}
be the natural quotient homomorphism. In the proof of this Theorem, for convenience, given any $g\in N_{s-1}$, the image $\pi_s(g)$ is denoted by $\bar{g}$.
The Theorem will be proved through the following Claims.

\vspace{0.5cm}

{\bf{Claim 1.}} There exists a sequence of elements $\{\sigma_{s,1},\sigma_{s,2},\ldots,\sigma_{s,n_s}\}\subset B_{s-1}\subset N_{s-1}$ such that for each $1\leq k\leq n_s$, \begin{equation}\length_{B}(\sigma_{s,k})\leq 3\cdot 2^{c_0}-2,\end{equation} and
\begin{equation}
 A_{s,k}'\equiv\langle\bar{\sigma}_{s,1},\ldots, \bar{\sigma}_{s,k}\rangle,\label{good-free-abelian-subgroup}
\end{equation}
is a free abelian subgroup in $N_{s-1}/G_{s,0}$ with $\rank(A_{s,k}')=k$, where $\bar{\sigma}_{s,k}\equiv\pi_s(\sigma_{s,k})$.

\vspace{0.5cm}

Let $F_s$ be the canonical transversal with respect to the quotient $N_{s-1}/G_{s,0}$ with $F_s(e)=e$. With respect to the canonical transversal $F_s$, there is a sequence of elements $\{v_{\ell}\}_{\ell=1}^{n_s}\subset \Image(F_s)\subset N$ such that for each $1\leq\ell\leq n_s$, $v_{\ell}\in\bar{v}_{\ell}$, where $\bar{v}_{\ell}$ are torsion-free generators as in (\ref{Z^k-generators}). For each $1\leq \ell \leq n_s$,
let \begin{equation}\bm{v}_{\ell}=b_{\ell, 1}\cdot b_{\ell, 2}\cdot\ldots\cdot b_{\ell, d_{\ell}}\label{canonical-presentation-v_l}\end{equation}
 be the canonical presentation of $v_{\ell}$ such that $b_{\ell, j}\in B_{s-1}$ with $1\leq j\leq d_{\ell}$. By Corollary \ref{generator-is-minimal}, $b_{\ell, j}\in\Image(F_s)$, and hence $\bar{b}_{\ell, j}$ is of infinite order 
 in $N_{s-1}/G_{s,0}$ for all $1\leq j\leq d_{\ell}$. That is,
\begin{equation}
\langle\bar{b}_{\ell, j}\rangle\cong\mathbb{Z},\ \forall 1\leq j\leq d_{\ell}, \ 1\leq \ell\leq n_s.\label{single-element-torsion-free}
\end{equation}

Claim 1 follows from the following inductive statement. 

\vspace{0.5cm}

($*$) \quad For each 
$1\leq \ell\leq n_s$,  there exists some $b_{\ell,\mu_{\ell}}\in B_{s-1}$  with $1\leq \mu_{\ell}\leq d_{\ell}$
in the canonical presentations of $\{v_1,\ldots, v_{n_s}\}$ such that $\sigma_{s,\ell}$ can be chosen as $b_{\ell,\mu_{\ell}}$.

\vspace{0.5cm}

Let $\ell=1$ and we choose \begin{equation}\sigma_{s,1}\equiv b_{1,1}.\end{equation}
Equation (\ref{single-element-torsion-free}) shows that $\langle\bar{\sigma}_{s,1}\rangle=\langle\bar{b}_{1,1}\rangle\cong\mathbb{Z}$. 

We just set $n_s>1$. Otherwise, it is done. Assume that Statement $(*)$ holds for $1\leq\ell\leq k-1$, we will search some $b_{k,\mu_k}$
in the canonical presentations of $\{v_1,\ldots, v_k\}$ such that $\sigma_{s,k}$ can be chosen as $b_{k,\mu_k}$.
Observe that
 there exists some $b_{k,\mu_{k}}\in B_{s-1}$ with $1\leq\mu_{k}\leq d_{k}$ such that $\bar{b}_{k,\mu_{k}}\not\in\langle\bar{\sigma}_{s,1},\ldots,\bar{\sigma}_{s,k-1}\rangle$. If not, then for all $1\leq \ell\leq k$ and $1\leq j\leq d_{\ell}$, it holds that
\begin{equation}
 \bar{b}_{\ell, j}\in\langle\bar{\sigma}_{s,1},\ldots,\bar{\sigma}_{s,k-1}\rangle.
\end{equation}
which implies that
\begin{equation}
\bar{v}_{\ell}\in\langle\bar{\sigma}_{s,1},\ldots,\bar{\sigma}_{s,k-1}\rangle,\ \forall 1\leq\ell\leq k.
\end{equation}
That is, $\rank\langle\bar{v}_1,\ldots,\bar{v}_k\rangle=k-1$ which contradicts to equation (\ref{Z^k-generators}).
Therefore, if we choose $\sigma_{s,k}=b_{k,\mu_k}$, then by (\ref{single-element-torsion-free})
\begin{equation}
\langle\bar{\sigma}_{s,k}\rangle=\langle\bar{b}_{k,\mu_k}\rangle\cong\mathbb{Z},
\end{equation}
and
\begin{equation}
\langle\bar{\sigma}_{s,1},\ldots,\bar{\sigma}_{s,k-1},\bar{\sigma}_{s,k}\rangle\cong\mathbb{Z}^k.
\end{equation}
Now the induction is complete. 

Finally, for each $1\leq k\leq n_s$, $\length_{B}(\sigma_{s,k})\leq 3\cdot 2^{c_0}-2$ follows from $\sigma_{s,k}\in B_{s-1}$ and equation (\ref{length-estimate-higher-generating-set}).
We have finished the proof of the Claim 1.

\vspace{0.5cm}

We have constructed the subset (\ref{good-cyclic-factors}). So now we are in a position to give the desired normal series. 

\vspace{0.5cm}

{\bf{Claim 2.}} There exists a normal series 
\begin{equation}
N_{s-1}\rhd N_{s,n_s}\rhd N_{s,n_s-1}\rhd\ldots\rhd N_{s,1}\rhd N_{s,0}\rhd N_s=[N_{s-1},N] \label{first-level-normal-series}
\end{equation}
such that for each $1\leq k\leq n_s$, the following holds:

$(i)$
$N_{s,k}/N_{s,k-1}\cong\mathbb{Z}$. Moreover, $N_{s-1}/N_{s,n_s}$ and $N_{s,0}/N_s$ are finite abelian groups.

$(ii)$ The quotient factor $N_{s,k}/N_{s,k-1}$ is generated by $\sigma_{s,k}\cdot N_{s,k-1}$, and hence $\langle\sigma_{s,k}\cdot N_{s,k-1}\rangle\cong\mathbb{Z}$.

\vspace{0.5cm}

For each $1\leq k\leq n_s$, let $A_{s,k}'$ be the free abelian group
defined by equation (\ref{good-free-abelian-subgroup}) such that
\begin{equation}
A_{s,k}'/A_{s,k-1}'\cong\mathbb{Z}.
\end{equation}
 Recall that $\pi_s$ is the natural quotient homomorphism
\begin{equation}
\pi_s: N_{s-1}\rightarrow N_{s-1}/G_{s,0},
\end{equation}
then define \begin{equation}N_{s,k}\equiv\pi_1^{-1}(A_{s,k}')\end{equation} for each $1\leq k\leq n_s$.
Lemma \ref{pre-image-lemma} implies that
$N_{s,k-1}\lhd N_{s,k}$ and $N_{s,k}/N_{s,k-1}\cong\mathbb{Z}$ for each $1\leq k\leq n_s$. Moreover, $N_{s,n_s}\lhd N_{s-1}$, $N_{s}\lhd N_{s,0}$ and
\begin{equation}
N_{s-1}/N_{s,n_s}\cong A_{s}/A_{s,n_s}',\ N_{s,0}/N_s\cong\Tor(A_s).
\end{equation}
Hence both of the above two quotient groups are finite abelian.
We have obtained normal series (\ref{first-level-normal-series})
and proved property $(i)$.

Now we are showing that $N_{s,k}=\langle\sigma_{s,k}, N_{s,k-1}\rangle$ for each $1\leq k\leq n_s$. By definition, $\bar{\sigma}_{ks}\in A_{s,k}'$, and thus $\sigma_{k,s}\in\pi_s^{-1}(A_{s,k}')=N_{s,k}$. So it suffices to show that 
\begin{equation}
N_{s,k}\leq\langle\sigma_{s,k}, N_{s,k-1}\rangle.
\end{equation}
 In fact, take any $g\in N_{s,k}$, there are $a_1,\ldots, a_k\in \mathbb{Z}$ such that
\begin{equation}
\bar{g}=\bar{\sigma}_{s,1}^{a_1}\cdot\bar{\sigma}_{s,2}^{a_2}\cdot\ldots\cdot\bar{\sigma}_{s,k}^{a_{k}}\in A_{s,k}'.
\end{equation}
Since the quotient $\pi_s$ is a homomorphism,
\begin{equation}
\overline{g\cdot\sigma_{s,k}^{-a_{k}}}=\bar{\sigma}_{s,1}^{a_1}\cdot\ldots\cdot\bar{\sigma}_{s,k-1}^{a_{k-1}}\in A_{s,k-1}'.
\end{equation}
By definition, $g\cdot\sigma_{s,k}^{-a_{k}}\in N_{s,k-1}$ and thus
\begin{equation}
g\in \langle\sigma_{s,k}, N_{s,k-1}\rangle.
\end{equation}

The last is to show that the quotient group $N_{s,k}/N_{s,k-1}$ is generated by $\sigma_{s,k}\cdot N_{s,k-1}$.  The relation $N_{s,k-1}\lhd N_{s,k}$ implies that for any $d_1\in\mathbb{Z}$ and $h\in N_{s,k-1}$, there exists $h'\in N_{s,k-1}$ such that
$\sigma_{k,s}^{-d_1}\cdot h\cdot\sigma_{k,s}^{d_1}=h', 
$ i.e. $ h\cdot\sigma_{k,s}^{d_1}=\sigma_{k,s}^{d_1}\cdot h'$. Consequently, any $g\in N_{s,k}$
can be presented in terms of $\sigma_{s,k}^{d}\cdot g'$
for some $d\in\mathbb{Z}$ and for some $g'\in N_{s,k-1}$. Therefore, \begin{equation}g\cdot N_{s,k-1}=(\sigma_{k,s}\cdot N_{s,k-1})^d.
 \end{equation}
 Property $(iii)$ has been proved. 
Now the proof of Claim 2 is complete.

\vspace{0.5cm}

The Theorem immediately follows from Claim 1 and Claim 2.

\vspace{0.5cm}

\end{proof}

The following Corollary will be applied in the proof of some Lemmas in Section \ref{ss:nonlocal_fund_group}.

\begin{corollary}\label{collection-of-graded-generators} In the notations of the Theorem \ref{t:refinement-of-lower-central-series}, the following properties hold: 

$(i)$ \begin{equation}N'\equiv\langle\sigma_{1,n_1},\sigma_{1,n_1-1},\ldots,\sigma_{1,1},\sigma_{2,n_2},\sigma_{2,n_2-1},\ldots,\sigma_{2,1},\ldots, \sigma_{c_0,n_{c_0}},\sigma_{c_0,n_{c_0}-1},\ldots,\sigma_{c_0,1}\rangle
\end{equation} is a nilpotent subgroup in $N$ with $\rank(N')=m$.

$(ii)$ for any $a_{i,j}\in\mathbb{Z}$ with $1\leq i\leq c_0$ and $1\leq j\leq n_i$,
\begin{equation}
N''\equiv\langle\sigma_{1,n_1}^{a_{1,n_1}},\sigma_{1,n_1-1}^{a_{1,n_1-1}},\ldots,\sigma_{1,1}^{a_{1,1}},\sigma_{2,n_2}^{a_{2,n_2}},\sigma_{2,n_2-1}^{a_{2,n_2-1}},\ldots,\sigma_{2,1}^{a_{2,1}},\ldots, \sigma_{c_0,n_{c_0}}^{a_{c_0,n_{c_0}}},\sigma_{c_0,n_{c_0}-1}^{a_{c_0,n_{c_0}-1}},\ldots,\sigma_{c_0,1}^{a_{c_0,1}}\rangle
\end{equation}
is a nilpotent subgroup of $N'$ with $\rank(N'')=m$.
\end{corollary}

\begin{proof}
$(i)$ We prove it by induction on the step of the nilpotent group $N$. It is clear that 
the statement is true if $\Step(N)=1$. In this case, $N'$ is a finite-index free abelian subgroup in $N$  such that $\rank(N')=\rank(N)$. Assume that the statement is true for any finitely generated nilpotent group of step $\leq c-1$, we will prove that the statement also holds for $N$ with $\Step(N)=c\geq 2$. 

Let us denote
\begin{equation}
N_1'\equiv\langle\sigma_{2,n_2},\sigma_{2,n_2-1},\ldots,\sigma_{2,1},\ldots, \sigma_{c_0,n_{c_0}},\sigma_{c_0,n_{c_0}-1},\ldots,\sigma_{c_0,1}\rangle\leq N_1=[N,N].
\end{equation}
Notice that, $\Step(N_1)=c-1$ and then
\begin{equation}
\rank(N_1')=\rank(N_1)=m-n_1,\label{inductive-rank}.
\end{equation}
where $n_1=\rank(N/N_1)$.
Let us denote 
\begin{equation}
N_{1,j}'\equiv\langle\sigma_{1,j},\ldots,\sigma_{1,1},\sigma_{2,n_2},\sigma_{2,n_2-1},\ldots,\sigma_{2,1},\ldots, \sigma_{c_0,n_{c_0}},\sigma_{c_0,n_{c_0}-1},\ldots,\sigma_{c_0,1}\rangle,\ 1\leq j\leq n_1.
\end{equation}
 It suffices to show that for every $1\leq j\leq n_1-1$,
\begin{equation}
[N_{1,j+1}':N_{1,j}']=\infty\label{infinite-index-1}\end{equation}
and
\begin{equation}
[N_{1,1}':N_1']=\infty.\label{infinite-index-2}
\end{equation}
In fact, by Lemma \ref{basic-nil-lemma}, the above two equations imply that
\begin{equation}
\rank(N_{1,j+1}')\geq\rank(N_{1,j}')+1,\ 1\leq j\leq n_1-1
\end{equation}
and
\begin{equation}
\rank(N_{1,1}')\geq \rank(N_1')+1.
\end{equation}
Therefore, by inequality (\ref{inductive-rank}),
 \begin{equation}
m\geq \rank(N')=\rank(N_{1,n_1}')\geq\rank(N_1')+n_1=(m-n_1)+n_1,
 \end{equation}
and thus $\rank(N')=m$.
Now we prove equation (\ref{infinite-index-1}) and equation (\ref{infinite-index-2}). In fact, by Theorem \ref{t:refinement-of-lower-central-series}, $N_{1,j+1}/N_{1,j}=\langle \sigma_{j+1}\cdot N_{1,j}\rangle\cong\mathbb{Z}$.
Notice that, $N_{1,j}'\leq N_{1,j}$ and then for any $d_{1,j+1}\in\mathbb{Z}$, $\sigma_{1,j+1}^{d_{1,j+1}}\not\in N_{1,j}'$.
We have proved (\ref{infinite-index-1}). Similarly, (\ref{infinite-index-2}) also holds.
The proof is complete

$(ii)$ The argument is the same as that of $(i)$. So we just omit the proof.

\end{proof}

\subsection{Geometric Properties of Almost Nilpotent Groups}
\label{ss:geometric-properties-of-nilpotency}
Those algebraic tools developed in the above subsections will now be presented and used in a more geometric fashion.
The basic setting is that on a Riemannian manifold $(M^n,g,p)$, assume that  $\Gamma_{\delta}(p)\equiv\Image[\pi_1(B_{\delta}(p))\rightarrow\pi_1(B_2(p))]$ is $(w,\ell)$-nilpotent for some a priori constants $w>0$, $\ell>0$. 
Note that in the context of $\Ric_g\geq-(n-1)$, theorem \ref{Generalized-Margulis-Lemma} shows that the above statement is always true for sufficiently small $\delta(n)$.

We will focus on more delicate geometric properties of $(w,\ell)$-nilpotency. 
Specifically, the nilpotent subgroup $\mathcal{N}$ of controlled index $w$ in $\Gamma_{\delta}(p)$ can be generated by the short loops of scale roughly $\delta$. Moreover, 
the lower central series of $\mathcal{N}$ has a nice refinement such that each infinite cyclic factor is generated by short loops of scale roughly $\delta$.
This refinement of the normal series plays an important role in the proof of our $\epsilon$-regularity Theorems. First, we give a precise description of the above picture.
\begin{definition}
[$(m,\epsilon)$-displacement property]\label{def-small-displacement}Given a metric space $(X,d,p)$, a nilpotent group $N$ with $\rank(N)=m$, assume that $N\leq\Isom(X)$.
We say $N$ satisfies the $(m,\epsilon)$-\textit{displacement property at} $p$ if there exists a $c$-tuple $(n_1,n_2,\ldots, n_{c})\in\mathbb{Z}^c$ with $c=\Step(N)$ and $m=\sum\limits_{s=1}^{c}n_s$ such that for each $1\leq s\leq c$, the following properties hold:

$(i)$ the lower central series has the refinement
\begin{equation}
N_{s-1}\rhd N_{s,n_s}\rhd N_{s,n_s-1}\rhd N_{s,n_s-2}\rhd\ldots\rhd N_{s,1}\rhd N_{s,0}\rhd N_{s}\label{refined-lcs}
\end{equation}
 such that $N_{s,k}/N_{s,k-1}\cong\mathbb{Z}$ with $1\leq k\leq n_s$, where $N_s\equiv [N_{s-1},N]$ and $N_0\equiv N$. Moreover, $N_{s-1}/N_{s,n_s}$ and $N_{s,0}/N_s$ are finite groups.

$(ii)$ There exists a subset \begin{equation}\{\sigma_{s,1},\sigma_{s,2},\ldots,\sigma_{s,n_s}\}\subset N_{s-1}\label{good-cyclic-factors}\end{equation} such that for each $1\leq k\leq n_s$, 
\begin{equation}d(\sigma_{s,k}\cdot p,p)<\epsilon,\end{equation} 
and  $N_{s,k}/N_{s,k-1}=\langle\sigma_{s,k}\cdot N_{s,k-1}\rangle$.
Each $\sigma_{s,k}$ is called an {\textit{
$\epsilon$-graded generator}}.
\end{definition}

\begin{remark}
In applications, we will focus on a normal cover with the isometric action of the deck transformation group.\end{remark}

By Definition \ref{def-small-displacement}, Theorem \ref{t:refinement-of-lower-central-series}
can be written in a geometric fashion. That is,

\begin{lemma}\label{w-nilpotent-nice-factors}
Given $\delta>0$, $\ell,w<\infty$, then the constant $K=10\cdot w\cdot 2^{\ell}<\infty$ satisfies the following property. Let $(X,d,p)$ be a metric space and $\Gamma\leq\Isom(X)$ be  finitely generated 
such that:

$(i)$ $\Gamma$ is $(w,\ell)$-nilpotent,

$(ii)$ $\Gamma$ has a finite generating set \begin{equation}S_{\delta}(p)\equiv\Big\{\gamma_1,\ldots,\gamma_{d_1}\Big|d(\gamma_j\cdot p,p)<\delta,\ \gamma_j\in\Isom(X),\ 1\leq j\leq d_1<\infty\Big\},\end{equation}
then $\Gamma$ has a nilpotent subgroup $N$ with $[\Gamma:N]\leq w$, $\length(N)\leq \ell$ and $N$ satisfies the $(m,K\cdot\delta)$-displacement property at $p$ for some $m\leq \ell$.
\end{lemma}
\begin{proof}
By $(i)$, there exists $N\leq\Gamma$ such that $[\Gamma:N]\leq w$, $\rank(N)=m$ and $\Step(N)=c_0\leq \length(N)\leq \ell$. Lemma \ref{effective-generating-set} states that 
$N$ has a generating set
$B$ with $\#(B)\leq w^2\cdot d_1$ such that for each $b\in B$, 
\begin{equation}
\length_{S_{\delta}(p)}(b)\leq 2w+1,
\end{equation}
which implies that
\begin{equation}
d(b\cdot p,p)\leq (2w+1)\cdot\delta.
\end{equation}

The $(m,K\cdot\delta)$-displacement property of $N$ immediately follows from Theorem \ref{t:refinement-of-lower-central-series}. That is, by Theorem \ref{t:refinement-of-lower-central-series}, we can refine the lower central series of $N$ such that (\ref{refined-lcs}) holds. Moreover, there exists a subset
\begin{equation}
\{\sigma_{s,1},\ldots,\sigma_{s,n_s}\}\subset N_{s-1}
\end{equation}
such that 
\begin{equation}
N_{s,k}/N_{s,k-1}=\langle\sigma_{s,k}\cdot N_{s,k-1}\rangle
\end{equation}
and 
\begin{equation}
\length_B(\sigma_{s,k})\leq 3\cdot 2^{c_0}-2.
\end{equation}
Therefore, $d(\sigma_{s,k}\cdot p, p)\leq (2w+1)\cdot( 3\cdot 2^s-2)\cdot\delta< K\cdot\delta$.

\end{proof}

In the context of lower Ricci curvature, we consider a normal covering space $\widehat{M}$ with the deck transformation group $G$. The Theorem below shows the structure of $\hat{G}_{\epsilon}(p)$ 
which is generated by 
short elements in $G$.  The proof is a combination of Theorem \ref{t:KW_almost_nilpotent} and Lemma \ref{w-nilpotent-nice-factors}.

\begin{theorem}\label{t:small-displacement-property}
Let $(Z^k,z^k)$ be a pointed Ricci-limit metric space with $\dim Z^k = k$ in the sense of theorem \ref{limiting-dimension}.  Then there exists 
$\epsilon_0=\epsilon_0(n,B_1(z^k))>0$, $w_0=w_0(n,B_1(z^k))<\infty$, $K_0=K_0(n,B_1(z^k))<\infty$ such that the following holds. For every $0<\delta\leq\epsilon_0$, if $(M^n, g,p)$  is  a Riemannian manifold with
$\Ric\geq-(n-1)$ such that $B_2(p)$ has a compact closure in $B_4(p)$ and
 \begin{equation}d_{GH}(B_2(p),B_2(z^k))<\delta,
\end{equation}
and if $\pi:(\widehat{B_2(p)},\hat{p},G)\rightarrow B_2(p)$ is a normal covering with $\pi(\hat{p})=p$ and the deck transformation group $G$, then the group 
\begin{equation} \hat{G}_{\delta}(p)\equiv\langle g\in G|d(g\cdot\hat{p},\hat{p})<2\delta\rangle\end{equation} contains a nilpotent subgroup $\mathcal{N}$
such that

$(i)$
$[\hat{G}_{\delta}(p):\mathcal{N}]\leq w_0,\ \length(\mathcal{N})\leq n-k,$

$(ii)$ $\mathcal{N}$ has the $(m_0,K_0\cdot\delta)$-displacement property at $\tilde{p}$, where
$m_0\equiv\rank(\mathcal{N})\leq n-k$.\end{theorem}

\begin{proof} First, we construct a nilpotent subgroup $\mathcal{N}\leq\hat{G}_{\delta}(p)$ such that $(i)$ holds.
To this end, let
\begin{equation}{S}_{\delta}(p,G)\equiv\{\gamma\in G|d(\gamma\cdot\hat{p},\hat{p})<2\delta\}\end{equation}
and then by definition ${S}_{\delta}(p,G)$ is a generating set of $\hat{G}_{\delta}(p)$.
Let $\tilde{p}$ be a lift of $\hat{p}$ on the universal cover and then let us denote by $\tilde{\gamma}\in\pi_1(B_2(p))$ the unique lifting such that 
\begin{equation}
d(\tilde{\gamma}\cdot\tilde{p},\tilde{p})<2\delta,
\ \pr(\tilde{\gamma})=\gamma,
\end{equation}
where $\pr:\pi_1(\widetilde{M}^n,\tilde{p})\rightarrow(\widehat{M}^n,\hat{p})$
is the natural quotient homomorphism.
By the above unique lifting, we define the group
\begin{equation}
\widetilde{G}_{\delta}(p)\equiv\Big\langle\tilde{\gamma}\in\pi_1(B_2(p))\Big|\gamma\in{S}_{\delta}(p,G)\Big\rangle.
\end{equation}
Immediately, $\widetilde{G}_{\delta}(p)\leq\Gamma_{\delta}(p)\equiv\Image[\pi_1(B_{\delta}(p))\rightarrow\pi_1(B_2(p))]$. Theorem \ref{t:KW_almost_nilpotent} shows that there exists a nilpotent subgroup $\mathcal{N}_0\leq \Gamma_{\delta}(p)$ such that
\begin{equation}
[\Gamma_{\delta}(p):\mathcal{N}_0]\leq w_0(n,B_1(z^k)),\ \length(\mathcal{N}_0)\leq n-k.
\end{equation}
For $\mathcal{N}_1\equiv\mathcal{N}_0\cap\widetilde{G}_{\delta}(p)\leq \mathcal{N}_0$, it holds that,
\begin{equation}
[\widetilde{G}_{\delta}(p):\mathcal{N}_1]=[\widetilde{G}_{\delta}(p):\mathcal{N}_0\cap\widetilde{G}_{\delta}(p)]\leq [\Gamma_{\delta}(p):\mathcal{N}_0]\leq w_0(n,B_1(z^k)),
\end{equation}
and
\begin{equation} \length(\mathcal{N}_1)\leq\length(\mathcal{N}_0)\leq n-k.
\end{equation}
Let $\mathcal{N}\equiv\pr(\mathcal{N}_1)$ be the homomorphism image and immediately
\begin{equation}
\mathcal{N}\leq\hat{G}_{\delta}(p),\ \length(\mathcal{N})\leq\length(\mathcal{N}_1)\leq  n-k. 
\end{equation}
Note that
$\hat{G}_{\delta}(p)=\pr(\widetilde{G}_{\delta}(p)),
$ and then by lemma \ref{isomorphism-lemma},
\begin{equation}
[\hat{G}_{\delta}(p):\mathcal{N}]= [\pr(\widetilde{G}_{\delta}(p)):\pr(\mathcal{N}_1)]\leq [\widetilde{G}_{\delta}(p)):\mathcal{N}_1]\leq w_0(n,B_1(z^k)).
\end{equation}
We have verified $(i)$.

By lemma \ref{basic-nil-lemma} $(iii)$, $\Step(\mathcal{N})\leq\length(\mathcal{N})\leq n-k$ and then property $(ii)$ directly follows from Theorem \ref{t:refinement-of-lower-central-series}.

\end{proof}

\section{Quantitative Splitting and Non-Collapse on Covering Spaces}\label{s:quantitative-splitting}

In this section we discuss the technical heart of this paper.  The following Theorem is the most crucial technical ingredient in the proof of $\epsilon$-regularity for both lower Ricci curvature bound and bounded Ricci curvature.  Its proof is also the main goal of this section:

\begin{theorem}\label{non-collapsed-splitting-maximal-rank}
Let $(M^n,g,p)$ be a Riemannian manifold with $\Ric\geq-(n-1)$ such that $B_2(p)$ has a compact closure in $B_4(p)$.  Given a pointed Ricci-limit space $(Z^{\ell},z^{\ell})$ with $\dim Z^{\ell}=\ell$ there exists $v_0(n,B_1(z^{\ell}))>0$, and for every $\epsilon>0$ there exists $\delta_0(\epsilon,n,B_1(z^{\ell}))>0$ such that if

$(i)$ $d_{GH}(B_2(p),B_2(0^{k-\ell},z^{\ell}))<\delta_0, \ (0^{k-\ell},z^{\ell})\in\dR^{k-\ell}\times Z^{\ell}$,

$(ii)$ $\Gamma_{\delta_0}(p)\equiv\Image[\pi_1(B_{\delta_0}(p))\rightarrow\pi_1(B_2(p))]$ satisfies $\rank(\Gamma_{\delta_0}(p))=n-k$,
\\
then for each $q\in B_1(p)$ 

$(1)$ $\Vol(B_{1/2}(\tilde{q}))\geq v_0(n,B_1(z^{\ell}))>0$,

$(2)$ for some $\delta_0<r<1$,
\begin{equation}
d_{GH}(B_{r}(\tilde{q}),B_{r}(0^{n-\ell},\hat{z}))<r\epsilon,\ (0^{n-\ell},\hat{z})\in\dR^{n-\ell}\times C(\hat{Z}),
\end{equation}
where $\tilde{q}$ is a lift of $q$ on the universal cover of $B_1(q)$ and $(C(\hat{Z}),\hat{z})$ is a metric cone over some compact metric space $\hat{Z}$ with the cone tip $\hat{z}$.
\end{theorem}

The key aspect of the above result, as opposed to similar results in the literature and in this subsection, is that the space $Z$ is an arbitrary complete space.  To accomplish this we will need to use the {\textit{cone splitting principle}}, as opposed to the line splitting (see Lemma \ref{Ricci-homoline}) which is more commonly used.  The cone splitting principle in the context of lower Ricci curvature requires noncollapsing, thus we will need to first prove $(1)$ before we can prove $(2)$.

The outline of this Section is as follows.  In Section \ref{ss:nonlocal_fund_group} we prove a result which in effect tells us that the rank of the fibered fundamental group is nonlocal in nature.  Specifically, it will allow us to control the fibered fundamental group of a point based on the fibered fundamental group of nearby points.  In Section \ref{ss:symmetry-and-splitting} we will prove a version of Theorem \ref{non-collapsed-splitting-maximal-rank} when $Z$ is assumed to be a point.  On the other hand, this version will not require maximality of the rank and will hold for general covering spaces.  We will use this in Section \ref{ss:non-collapse-universal-cover}, in combination with the nonlocal results of Section \ref{ss:nonlocal_fund_group}, in order to prove $(1)$, the noncollapsing of the covering space.  Finally, in Section \ref{ss:proof_noncollapsed_splitting} we will combine this with the cone splitting principle in order to finish the proof of Theorem \ref{non-collapsed-splitting-maximal-rank}.

\subsection{Nonlocalness Properties of Fibered Fundamental Groups}\label{ss:nonlocal_fund_group}

Here we prove that if the fibered fundamental group $\Gamma_\delta(p)\equiv \Image[\pi_1(B_{\delta}(p))\rightarrow \pi_1(B_2(p))]$ has large nilpotency rank at some point $p$, then the nilpotency rank of the fibered nilpotent group centered at any other nearby point is also large. This non-localness property will be used in the proof of the quantitative splitting result on a normal cover. This is also important in order to extend the region of control of our $\epsilon$-regularity theorem.  We begin with the following:

\begin{lemma} \label{every-point-Z}
Given $n\geq 2$, $0<\epsilon<1/10$, $R\geq 1$, there exists $\Psi_1=\Psi_1(\epsilon,R,n)>0$, $N_1=N_1(\epsilon, R,n)<\infty$
such that the following properties hold: let $(M^n, g,p)$ be a Riemannian manifold with $\Ric\geq-(n-1)$ such that $B_{2R}(p)$ has a compact closure in $B_{4R}(p)$. If  $\gamma\in \Isom(M^n)$ satisfies \begin{equation}
d(\gamma\cdot p,p)<\Psi_1(\epsilon,R,n),
\end{equation}
then there is some positive integer $1\leq d\leq N_1$ such that  for all $x\in B_{R}(\hat{p})$
\begin{equation}
d(\gamma^{d}\cdot x,x)<\epsilon.\end{equation}
\end{lemma}

\begin{remark}
From the proof of this Lemma, it holds that for fixed $R\geq 1$, $\lim\limits_{\epsilon\rightarrow0}\Psi(\epsilon, R,n)=0$ and $\lim\limits_{\epsilon\rightarrow0}N_1(\epsilon, R,n)=\infty$. Moreover, when $\epsilon\rightarrow0$ and $R\rightarrow\infty$ simultaneously, $\Psi\rightarrow0$ and $N_1\rightarrow \infty$.
\end{remark}

\begin{proof} 

As the first step, we show that for each $0<\epsilon<1/10$, $R\geq 1$,  there exists $\Psi_1=\Psi_1(\epsilon,R,n)$, $M_0=M_0(\epsilon^{-1},R,n)$ such that for fixed
$q\in B_{R}(p)$, if $\gamma\in G$ satisfies
\begin{equation}
d(\gamma\cdot{p},{p})<\Psi_1(\epsilon,n,R),\label{c1-control}
\end{equation}
then for some positive integer $1\leq d({q})\leq M_0$ such that  
\begin{equation}
d(\gamma^{d({q})}\cdot{q},{q})<\epsilon.\label{1-point-small-displacement}\end{equation}

To this end, let $N$ be a positive integer such that \begin{equation}B_{\epsilon/2}(\gamma^{\alpha}\cdot{q})\subset B_{2R}({p}),\ \forall\ 1\leq \alpha\leq N.\end{equation}
We claim that there exists $\Psi_1=\Psi_1(\epsilon,R,n)$, $N_0=N_0(\epsilon^{-1},R,n)$ such that if inequality (\ref{c1-control}) holds for $\Psi_1$, then for all $N>N_0$ there are $1\leq \alpha_0< \beta_0\leq N$ with
\begin{equation}B_{\epsilon/2}(\gamma^{\alpha_0}\cdot{q})\cap B_{\epsilon/2}(\gamma^{\beta_0}\cdot{q})\neq\emptyset.\label{non-empty}
\end{equation}

Let us consider a little more general case than the claim. We will prove that there exists $N_0(\epsilon^{-1},R,n)<\infty$ such that
if $\{{x}_{\alpha}\}_{\alpha=1}^{N'}$ is a finite subset in $B_{R}({p})$ with $B_{\epsilon/2}({x}_{\alpha})\subset B_{2R}({p})$ and
\begin{equation}B_{\epsilon/2}({x}_{\alpha})\cap B_{\epsilon/2}({x}_{\beta})=\emptyset,\ \forall\ 1\leq \alpha<\beta\leq N',
\end{equation}
then
\begin{equation}
N'\leq N_0.
\end{equation}
We choose $x_0\in\{{x}_{\alpha}\}_{\alpha=1}^{N'}$ such that
\begin{equation}\Vol(B_{\epsilon/2}(x_0))=\min\limits_{1\leq\alpha\leq N'}\{\Vol(B_{\epsilon/2}({x}_{\alpha}))\}.\end{equation}
Since
$B_{\epsilon/2}({x}_0)\subset B_{2R}({p})\subset B_{10R}({x}_0)$ and the balls in $\{B_{\epsilon/2}({x}_a)\}$ are disjoint, 
 then
\begin{eqnarray}N'\leq\frac{\Vol(B_{2R}({p}))}{\Vol(B_{\epsilon/2}({x}_0))}
\leq\frac{\Vol(B_{10R}({x}_0))}{\Vol(B_{\epsilon/2}({x}_0))}\leq\frac{V_{-1}^n(10R)}{V_{-1}^n(\epsilon/2)}\equiv N_0(\epsilon^{-1},R,n),\label{def-N_0} 
\end{eqnarray} 
the last inequality is by Bishop-Gromov volume comparison theorem, 
 where $V_{-1}^n(r)$ is the volume of the ball $B_r(0^n)$ in the space form of curvature $-1$ and dimension $n$.
 
If we take ${x}_{\alpha}\equiv\gamma^{\alpha}\cdot{q}$, then the claim follows. In fact,
let $\Psi_1\equiv\frac{\epsilon}{2M_0(\epsilon^{-1},R,n)}$, where
$M_0\equiv2N_0(\epsilon^{-1},R,n)$ and $N_0$ is the constant in (\ref{def-N_0}). Clearly, by (\ref{def-N_0}), for fixed $R\geq 1$,
\begin{equation}
\lim\limits_{\epsilon\rightarrow 0}M_0(\epsilon^{-1},R,n)=\infty,\ \lim\limits_{\epsilon\rightarrow0} \Psi_1(\epsilon^{-1},R,n)=0.
\end{equation}
By triangle inequality, for each $1\leq \alpha\leq M_0$
\begin{equation}
d(\gamma^{\alpha}\cdot{q},{p})\leq d(\gamma^{\alpha}\cdot{q},\gamma^{\alpha}\cdot{p})+d(\gamma^{\alpha}\cdot{p},{p})<
R+M_0\cdot\Psi_1<2R,
\end{equation}
which implies that for each $1\leq \alpha\leq M_0$, 
$B_{\epsilon/2}(\gamma^{\alpha}\cdot{q})\subset B_{2R}({p})$.
Since $M_0\equiv2N_0>N_0$, then there existis $1\leq \alpha_0<\beta_0\leq M_0$ such that 
\begin{equation}B_{\epsilon/2}(\gamma^{\beta_0}\cdot{q})\cap B_{\epsilon/2}(\gamma^{\alpha_0}\cdot{q})\neq\emptyset,\end{equation}
That is,
\begin{equation}B_{\epsilon/2}(\gamma^{\beta_0-\alpha_0}\cdot{q})\cap B_{\epsilon/2}({q})\neq\emptyset,\ 1\leq \beta_0-\alpha_0\leq M_0.\end{equation}
We have finished the proof of the claim.
Let $d({q})\equiv\beta_0-\alpha_0$,
immediately, $d(\gamma^{d({q})}\cdot{q},{q})<\epsilon$. The proof of the first step is complete.

\vspace{0.5cm}

The next stage is to finish the proof of the statement. Fix a positive constant $\epsilon>0$.
Let $\{{y}_{\alpha}\}_{\alpha=0}^k$ be a $\frac{\epsilon}{10}$-dense subset of $B_{R}({p})$ with ${y}_0\equiv{p}$ such that $B_{R}({p})\subset\bigcup\limits_{\alpha=0}^k B_{\epsilon/10}({y}_{\alpha})$ and 
\begin{equation}
d({y}_{\alpha},{y}_{\beta})\geq \frac{\epsilon}{20},\ \forall \alpha\neq\beta.
\end{equation}
Relative volume comparison theorem implies that there exists some large integer $k_0=k_0(\epsilon^{-1},R,n)>0$ such that
\begin{equation}
k\leq k_0(\epsilon^{-1},R,n).
\end{equation}

Let $\Psi_1$ be the function in the first step and define
\begin{equation}
\Psi_2(\epsilon|n)\equiv\frac{\epsilon}{2M_0^{k_0}},\ \Psi(\epsilon|n)\equiv\frac{\Psi_1(\Psi_2,R,n)}{M_0^{k_0}}<<\Psi_2.\label{def-Psi}
\end{equation}
Let
\begin{equation}
d({p},\gamma\cdot{p})<\Psi.\label{initial-control}
\end{equation}
By the arguments in the first step,  inequality (\ref{initial-control}) implies that for each 
$0\leq \alpha\leq k$, there exists some integer $1\leq d_{\alpha}\leq M_0(\epsilon^{-1},R,n)$ with $d_0\equiv1$ and 
\begin{equation}
d(\gamma^{d_{\alpha}}\cdot {y}_{\alpha},{y}_{\alpha})<\Psi_2.\end{equation}
Define
$d\equiv\prod\limits_{\alpha=1}^{k}d_{\alpha},$ then immediately 
$d\leq M_0^k\leq M_0^{k_0},$
and let
\begin{equation}
N_1(\epsilon^{-1},R,n)\equiv M_0(\epsilon^{-1},R,n)^{k_0(\epsilon^{-1},R,n)}.\label{def-N_1}
\end{equation}
Triangle inequality implies that for each $0\leq \alpha\leq k$ and for each ${q}_{\alpha}\in B_{\epsilon/10}({y}_{\alpha})$
\begin{eqnarray}
d(\gamma^{d}\cdot{q}_{\alpha},{q}_{\alpha})&\leq & d(\gamma^{d}\cdot{q}_{\alpha},\gamma^{d}\cdot{y}_{\alpha})+
d(\gamma^{d}\cdot{y}_{\alpha},{y}_{\alpha})+d({y}_{\alpha},{q}_{\alpha})\nonumber\\
&\leq & \frac{\epsilon}{5}+M_0^{k_0}\cdot d(\gamma^{d_{\alpha}}\cdot{y}_{\alpha},{y}_{\alpha})\nonumber\\
&<&\frac{\epsilon}{5}+M_0^{k_0}\cdot\Psi_2<\epsilon.
\end{eqnarray}
Notice that $B_{R}({p})\subset \bigcup\limits_{\alpha=0}^kB_{\epsilon/10}({y}_{\alpha})$, then for each ${x}\in B_{R}({p})$, the above inequality gives that \begin{equation}d(\gamma^{d}\cdot{x},{x})<\epsilon\end{equation}
with $1\leq d\leq N_1(\epsilon^{-1}|n)$.

The constants $N_1=N_1(\epsilon^{-1},R,n)$ and $\Psi=\Psi(\epsilon,R,n)$ have been defined in (\ref{def-Psi}) and (\ref{def-N_1}), then the proof of the Lemma is complete.

\end{proof}

By applying the above Lemma, we obtain the following 
non-localness of the nilpotency rank of the fibered fundamental group.

\begin{lemma} \label{every-point-rank}
Let $(M^n,g,p)$ be a Riemannian manifold with $\Ric\geq-(n-1)$ and such that $B_2(p)$ has a compact closure in $B_4(p)$. There exists $\epsilon_1(n)>0$, $w(n)<\infty$
such that for each $\epsilon\leq\epsilon_1(n)$ there exists $\Psi_0=\Psi_0(\epsilon|n)<<\epsilon$, $N_1=N_1(\epsilon^{-1}|n)<\infty$
such that the following holds:

$(i)$ $\Gamma_{\Psi_0}(p)\equiv\Image[\pi_1(B_{\Psi_0}(p))\rightarrow\pi_1(B_2(p))]$ 
 is $(w,n)$-nilpotent.
 
$(ii)$ For each $x\in B_1(p)$, the group $\Gamma_{\epsilon}'(x)\equiv\Image[\pi_1(B_{\epsilon}(x))\rightarrow\pi_1(B_2(p))]$ is $(w,n)$-nilpotent with 
$\rank(\Gamma_{\epsilon}'(x))\geq\rank(\Gamma_{\Psi_0}(p))$. Moreover, $\Gamma_{\epsilon}(x)\equiv\Image[\pi_1(B_{\epsilon}(x))\rightarrow\pi_1(B_1(x))]$ is $(w,n)$-nilpotent with 
$\rank(\Gamma_{\epsilon}(x))\geq\rank(\Gamma_{\Psi_0}(p))$.

$(iii)$ The group
$\Gamma_{\epsilon}(B_1(\tilde{p}))\equiv\langle\{\gamma\in\pi_1(B_{2}(p))|d(\gamma\cdot\tilde{x},\tilde{x})<2\epsilon,\ \forall\tilde{x}\in B_1(\tilde{p})\}\rangle
$ is $(w,n)$-nilpotent with $\rank(\Gamma_{\delta}(B_1(\tilde{p})))\geq\rank(\Gamma_{\Psi_0}(p))$, where $\tilde{p}$ is a lift of $p$ on the universal cover of $B_2(p)$.  
\end{lemma}

\begin{remark}
Property $(i)$ is a rewording of theorem \ref{Generalized-Margulis-Lemma} due to Kapovitch-Wilking. We restate it here just for the convenience of our arguments. 
\end{remark}

\begin{proof}

Let $K\equiv10\cdot w\cdot 2^{n}$ 
be the corresponding constant in Lemma \ref{w-nilpotent-nice-factors}, and for each $\epsilon>0$ define
\begin{equation}
\Psi_0(\epsilon|n)\equiv\frac{\Psi(\epsilon|n)}{2K(n)}<<\epsilon,
\end{equation}
where  $\Psi$ is the function in Lemma \ref{every-point-Z}.
Let $N_1\equiv N_1(\epsilon^{-1}|n)$
be the corresponding constant in Lemma \ref{every-point-Z}.
We will prove that the above constants satisfy properties $(i)$, $(ii)$, $(iii)$.

$(i)$ 
Take $\epsilon_1(n)$, $w(n)$ to be the corresponding constants as in theorem \ref{Generalized-Margulis-Lemma}, then
$\Gamma_{\epsilon_1}(p)
$ contains a nilpotent group $\mathcal{N}$ such that
\begin{equation}
[\Gamma_{\epsilon_1}(p):\mathcal{N}]\leq w(n),\ \Step(\mathcal{N})\leq\length(\mathcal{N})\leq n.
\end{equation}
By definition, for each $\epsilon\leq \epsilon_1(n)$,
\begin{equation}
\Gamma_{\epsilon}(p)\leq \Gamma_{\epsilon_1}(p).
\end{equation}
Denote $\widehat{\mathcal{N}}\equiv\Gamma_{\epsilon}(p)\cap\mathcal{N}$, and then
\begin{equation}
[\Gamma_{\epsilon}(p):\widehat{\mathcal{N}}]\leq [\Gamma_{\epsilon_1}(p):\mathcal{N}]\leq w,\ \Step(\widehat{\mathcal{N}})\leq \Step(\mathcal{N})\leq n.
\end{equation}
Therefore, property $(i)$ is proved.

$(ii)$
First, we construct a nilpotent subgroup of $\Gamma_{\epsilon}'(p)$ whose rank is equal to $\rank(\Gamma_{\Psi_0}(p))$.

Notice that, $(i)$ gives that $\Gamma_{\Psi_0}(p)$
is $(w,n)$-nilpotent. Now applying Lemma \ref{w-nilpotent-nice-factors} implies that, 
$\Gamma_{\Psi_0}(p)$
contains a nilpotent subgroup $\mathcal{N}$ such that\begin{equation}[\Gamma_{\Psi_0}(p):\mathcal{N}]\leq w, \ \Step(\mathcal{N})= c_0\leq\length(\mathcal{N})\leq n,\end{equation}
and $\mathcal{N}$
satisfies $(m_0,K\cdot\Psi_0)$-displacement property with $m_0\equiv\rank(\mathcal{N})=\rank(\Gamma_{\Psi_0}(p))$. Recall Definition \ref{def-small-displacement}, we collect all those $(m_0,K\cdot\Psi_0)$-graded generators
\begin{equation}
\mathcal{S}=\{\sigma_{1,1},\ldots,\sigma_{1,n_1},\ldots,\sigma_{c_0,1},\ldots,\sigma_{c_0,n_{c_0}}\}.
\end{equation}
 Note that, for each $1\leq i\leq c_0$, $1\leq j\leq n_i$,
 \begin{equation}
 d(\sigma_{i,j}\cdot\tilde{p},\tilde{p})<K\cdot\Psi_0<\Psi(\epsilon|n),
 \end{equation}
 then by Lemma \ref{every-point-Z}, there exists 
\begin{equation}
1\leq a_{i,j}\leq N_1(\epsilon^{-1}|n)
\end{equation}
such that 
\begin{equation}
d(\sigma_{i,j}^{a_{i,j}}\cdot\tilde{x},\sigma_{i,j})<\epsilon,\ \forall\tilde{x}\in B_1(\tilde{p}),
\end{equation}
where $\tilde{p}$ is a lift of $p$ on the universal cover of $B_2(p)$.
Denote
\begin{equation}
\mathcal{S}_a\equiv\{\sigma_{1,1}^{a_{1,1}},\ldots,\sigma_{1,n_1}^{a_{1,n_1}},\ldots,\sigma_{c_0,1}^{a_{c_0,1}},\ldots,\sigma_{c_0,n_{c_0}}^{a_{c_0,n_{c_0}}}\},\ \mathcal{N}_a\equiv\langle\mathcal{S}_a\rangle,
\end{equation}
then by definition,
$\mathcal{N}_a\leq\Gamma_{\epsilon}(B_1(\tilde{p}))$.
In particular, $\mathcal{N}_a\leq\Gamma_{\epsilon}'(x)$ 
for every $x\in B_1(p)$. 
Lemma \ref{collection-of-graded-generators} shows that 
\begin{equation}
\rank(\mathcal{N}_a)=\rank(\langle\mathcal{S}\rangle)=\rank(\mathcal{N})=\rank(\Gamma_{\Psi_0}(p)).
\end{equation}

Since $\Gamma_{\epsilon}'(x)$ is $(w,n)$-nilpotent, Lemma \ref{basic-nil-lemma} implies that 
\begin{equation}
\rank(\Gamma_{\epsilon}'(x))\geq\rank(\mathcal{N}_a)=\rank(\Gamma_{\Psi_0}(p)).
\end{equation}
Note that, from the construction, it is clear that $\mathcal{N}_a\leq\Gamma_{\epsilon}(x)\equiv\Image[\pi_1(B_{\epsilon}(x))\rightarrow\pi_1(B_1(x))]$. Therefore, 
$\rank(\Gamma_{\epsilon}(x))\geq \rank(\mathcal{N}_a)=\rank(\Gamma_{\Psi_0}(p))$.
We have finished the proof of property $(ii)$.

$(iii)$ From the Claim in the proof of Theorem \ref{t:KW_almost_nilpotent}, the group 
$\Gamma_{\epsilon}(B_1(\tilde{p}))$ is $(w,n)$- nilpotent. Therefore, by $(ii)$, 
\begin{equation}
\rank(\Gamma_{\epsilon}(B_1(\tilde{p})))\geq\rank(\mathcal{N}_a)=\rank(\Gamma_{\Psi_0}(p)).
\end{equation}

\end{proof}

\subsection{Symmetry and Quantitative Splitting of Normal Covering Spaces}
\label{ss:symmetry-and-splitting}

In this subsection we prove quantitative splittings on normal covering spaces in two different contexts, which depends on the compactness of limit space of the base manifolds.

We start with a quantitative splitting by assuming the limit space of the base manifolds is Euclidean, which is the foundation of the non-collapse arguments on the universal cover in Theorem \ref{non-collapsed-splitting-maximal-rank} (see more details in Proposition \ref{non-collapsed-universal-cover}):

\begin{proposition}\label{quantitative-splitting-at-origin} Let $(M^n,g,p)$ be a Riemannian manifold with $\Ric\geq-(n-1)$  and assume that $B_2(p)$ has a compact closure in $B_4(p)$.  Let $\pi:(\widehat{B_2(p)},\hat{p})\rightarrow (B_2(p),p)$ be a normal cover with $\pi(\hat{p})=p$ and the deck transformation group $G$. Given $\epsilon>0$, there exists
$\delta(\epsilon,n)>0$ such that if

$(i)$ $d_{GH}(B_2(p),B_2(0^{k}))<\delta,\ 0^{k}\in\dR^{k},$

$(ii)$ $\hat{G}_{\delta}(p)\equiv\langle\gamma\in G|d(\gamma\cdot\hat{p},\hat{p})<2\delta\rangle$ satisfies $\rank(\hat{G}_{\delta}(p))\geq m$,
\\
 then for some $\delta<r<1$ and some integer $d(p)\geq m$, it holds that
\begin{equation}d_{GH}\Big(B_{r}(\hat{p}),B_{r}(0^{k+d(p)})\Big)<\epsilon r,\ 0^{k+d(p)}\in\dR^{k+d(p)}.\label{unit-ball-close}\end{equation}
\end{proposition}
\begin{remark}\label{remark-on-line-splitting}
Similar, though less general, statements are proved in \cite{FY}, \cite{KW}.  Besides the greater generality of the above statement, there is also a technical challenge which is new in the above statement which does not appear previously.  Namely, in \cite{KW} an additional assumption is made, which is that $\pi_1(B_2(p))$ is generated by loops of length $\leq\delta$.  This allows them to handle the compact case.  To avoid this assumption we use the algebraic structure of Theorems \ref{t:refinement-of-lower-central-series} and \ref{t:small-displacement-property} to build a geometrically compatible extension of the lower central series of the nilpotent subgroup in $\hat{G}_{\delta}(p)$.  
\end{remark}

\begin{remark} Theorem \ref{t:small-displacement-property} implies that for sufficiently small $\delta>0$, $\hat{G}_{\delta}(p)$ is $(w_0(n),n-k)$-nilpotent. In particular, $\rank(\hat{G}_{\delta}(p))$ is well-defined. Assumption $(ii)$ is just to emphasize the relation between the nilpotency  $\rank(\hat{G}_{\delta}(p))$ and the number of the Euclidean factors on the normal cover.

\end{remark}

\begin{remark}If the normal cover is chosen as the universal cover of $B_2(p)$, then the corresponding deck transformation group is exactly the fundamental group of $B_2(p)$, i.e. 
$G=\pi_1(B_2(p))$ and $\hat{G}_{\delta}(p)=\Image[\pi_1(B_{\delta}(p))\rightarrow\pi_1(B_2(p))]$. In this case, immediately we obtain the quantitative splitting on the universal cover (see Proposition \ref{non-local-splitting}).
\end{remark}

\begin{remark}
Note that we are not giving explicitly the $r>0$ for which the result holds, only that it holds for some $r$ of a definite size.
\end{remark}

Before proving the Proposition, we need to show the following lemmas.

\begin{lemma}\label{Ricci-homoline}
Let $(M_i^n,g_i,p_i)$ be a sequence of complete Riemannian manifolds with $\Ric_{g_i}\geq-(n-1)\epsilon_i^2$. Assume that 
\begin{equation}
(M_i^n,g_i,p_i)\xrightarrow{GH}(X,d,p),
\end{equation}
for some complete non-compact length space 
$(X,d,p)$. If $X$ is $C$-homogeneous for some $C<\infty$, i.e. for every $x,y\in X$, there is an isometry $f\in \Isom(X)$ with
$d_X(y,f(x))\leq C,$
then $X$ is isometric to $\dR^k\times X'$, where $k\geq 1$ and $X'$ is compact. In particular, any ray $\gamma\subset X$ can be extended to a line which is tangential to $\dR^k$. 
\end{lemma}
\begin{proof}Lemma \ref{homoline} implies that if $X$ is non-compact, then $X$ admits a line. Due to Cheeger-Colding's splitting theorem, $X$ is isometric to $\dR\times X_1$ for some complete length space $X_1$.  Since $X$ is $C$-homogeneous, we must have that $X_1$ is $C$-homogeneous.  If $X_1$ is compact then we are done, otherwise we may again find a line in $X_1$. Therefore, by Cheeger-Colding's splitting theorem again, $X\cong\dR^2\times X_2$. The above process stops after a finite number of steps such that $X\cong\dR^k\times X'$ for some compact length space $X'$. In this case we have that any ray $\gamma$ in $X$ is tangential to $\dR^k$ and thus can be extended to a line. 
\end{proof}

\begin{lemma}\label{lift}  For each 
$\epsilon>0$, $n\geq 2$, there exists $\delta=\delta(\epsilon, n)>0$ such that the following property holds: let $(M^n,g,p)$ be a Riemannian manifold with 
$\Ric_{g}\geq-(n-1)\delta^2$ such that $B_{\delta^{-1}}(p)$ has a compact closure in $B_{2\delta^{-1}}(p)$ and for some complete length space $(Z,z)$,
\begin{equation}
d_{GH}(B_{\delta^{-1}}(p),B_{\delta^{-1}}(0^k,z))<\delta,\ (0^k,z)\in\dR^k\times Z,
\end{equation}
then on any covering space $\pi: (\widehat{B_{\delta^{-1}}(p)},\hat{p})\rightarrow B_{\delta^{-1}}(p)$ with $\pi(\hat{p})=p$,
  \begin{equation}
d_{GH}(B_{\epsilon^{-1}}(\hat{p}),B_{\epsilon^{-1}}(0^k,z'))<\epsilon,\ (0^k,z')\in\dR^k\times Z',
\end{equation}
where $B_{\epsilon^{-1}}(\hat{p})$ is a ball on the covering space and $(Z',z')$ is complete length space.
\end{lemma}

\begin{proof} 
We argue by contradiction. Suppose for some $\epsilon_0>0$, there is no such $\delta>0$. That is, we have a sequence of $\delta_i\rightarrow0$, a sequence of Riemannian manifolds $(M_i^n,g_i,p_i)$ with $\Ric\geq-(n-1)\delta_i^2$ and a sequence of complete length space $(Z_i,z_i)$ such that $B_{\delta_i^{-1}}(p_i)$ has a compact closure in $B_{2\delta_i^{-1}}(p_i)$ and 
\begin{equation}
d_{GH}(B_{\delta_i^{-1}}(p_i),B_{\delta_i^{-1}}(0^k,z_i))<\delta_i,\ (0^k,z_i)\in\dR^k\times Z_i,
\end{equation}
but there is some covering space $\pi_i:(\widehat{B_{\delta_i^{-1}}(p_i)},\hat{p}_i)\rightarrow B_{\delta_i^{-1}}(p_i)$ such that for any complete length space $(Z',z')$,
\begin{equation}
d_{GH}(B_{\epsilon_0^{-1}}(\hat{p}_i),B_{\epsilon_0^{-1}}(0^k,z'))\geq \epsilon_0,\ (0^k,z')\in\dR^k\times Z'.
\end{equation}

By Gromov's pre-compactness theorem, there is some complete length space $(Z_0,z_0)$
such that for sufficiently large $i$,
\begin{equation}
d_{GH}(B_{\delta_i^{-1}}(p_i),B_{\delta_i^{-1}}(0^k,z_0))<2\delta_i,\ (0^k,z_i)\in\dR^k\times Z_0.
\end{equation}
Let  $F(\hat{p}_i)$ be the fundamental domain of 
$(\widehat{B_{\delta_i^{-1}}(p_i)},\hat{p}_i)$
 containing $\hat{p}_i$, and then denote by $U_{2R}(\hat{p}_i)$ the subset of the covering space $(\widehat{B_{\delta_i^{-1}}(p_i)},\hat{p}_i)$ which satisfies
\begin{equation}
U_{2R}(\hat{p}_i)\equiv\pi_i^{-1}(B_{2R}(p_i))\cap F(\hat{p}_i).
\end{equation}
Fix $0<R<<\epsilon_i^{-1}$, then as in \cite{ChC} we have the following harmonic maps:
\begin{equation}\label{harmonicGHA}
\Phi_i\equiv(h_i^{(1)},\ldots, h_i^{(k)}): B_{2R}(p_i)\longrightarrow B_{2R}(0^k),
\end{equation}
 where $h_i^{(1)},\ldots, h_i^{(k)}$ are harmonic functions and satisfy the following estimate,
\begin{equation}\fint_{B_{2R}(p_i)}\sum\limits_{1\leq \alpha\leq \beta\leq k}|\langle\nabla h_i^{(\alpha)},\nabla h_i^{(\beta)}\rangle-\delta_{\alpha\beta}|
+\fint_{B_{2R}(p_i)}\sum\limits_{\alpha=1}^k|\nabla^2 h_i^{(\alpha)}|^2\leq\Psi(\delta_i|n,R).\label{integralclose}
\end{equation}

Observe that the inequality (\ref{integralclose}) is identical to the one on $U_{2R}(\hat{p}_i)$, that is,
\begin{equation}\fint_{U_{2R}(\hat{p}_i)}\sum\limits_{1\leq \alpha\leq \beta\leq k}|\langle\nabla \hat{h}_i^{(\alpha)},\nabla\hat{h}_i^{(\beta)}\rangle-\delta_{\alpha\beta}|+\fint_{U_{2R}(\hat{p}_i)}\sum\limits_{\alpha=1}^k|\nabla^2\hat{h}_i^{(\alpha)}|^2\leq\Psi(\delta_i|n,R),\label{domain-estimate}
\end{equation}
where $\hat{h}_i^{(\alpha)}$ is the lifting of $h_i^{(\alpha)}$ on the covering space $\widehat{B_{2\delta_i^{-1}}(p_i)}$.

Denote \begin{equation}
f_i\equiv\sum\limits_{1\leq \alpha\leq\beta\leq k}|\langle\nabla \hat{h}_i^{(\alpha)},\nabla\hat{h}_i^{(\beta)}\rangle-\delta_{\alpha\beta}|+
\sum\limits_{\alpha=1}^k|\nabla^2\hat{h}_i^{(\alpha)}|^2,
\end{equation}
and let $\{U_{\ell}\}_{\ell=1}^N$ be a finite collection of copies of $U_{2R}(\hat{p}_i)$ which covers $B_{R}(\hat{p}_i)$ and satisfies $U_{\ell}\cap B_{R}(\hat{p}_i)\neq\emptyset$ for each $1\leq\ell\leq N$.
Then clearly
\begin{equation}\bigcup\limits_{\ell=1}^NU_{\ell}\subset B_{10R}(\hat{p}_i)\end{equation}
and inequality (\ref{domain-estimate}) holds on each $U_{\ell}$,
which implies
\begin{eqnarray}\fint_{B_{R}(\hat{p}_i)}f_i&=&\frac{1}{\Vol(B_{R}(\hat{p}_i))}\int_{B_{R}(\hat{p}_i)}f_i\nonumber\\
&\leq&\frac{C(n,R)}{\Vol(B_{10R}(\hat{p}_i))}\int_{B_{R}(\hat{p}_i)}f_i\nonumber\\
&\leq&
\frac{C(n,R)}{\sum\limits_{\ell=1}^N\Vol(U_{\ell})}\int_{B_{R}(\hat{p}_i)}f_i\nonumber\\
&\leq&\frac{C(n,R)}{\sum\limits_{\ell=1}^N\Vol(U_{\ell})}\sum\limits_{\ell=1}^N\int_{U_{\ell}}f_i\nonumber\\
&\leq &C(n,R)\frac{\Psi(\epsilon_i|n,R)\sum\limits_{\ell=1}^N\Vol(U_{\ell})}{\sum\limits_{\ell=1}^N\Vol(U_{\ell})}\nonumber\\
&=&\Psi(\delta_i|n,R).
\end{eqnarray}
The last inequality follows from (\ref{domain-estimate}).  Using the quantitative splitting theorem of \cite{ChC} the above estimates imply that there exists some complete length space 
$(Z',z')$ such that
\begin{equation}
d_{GH}(B_{R}(\hat{p}_i),B_{R}(0^k,z'))<\Psi(\delta_i|n,R), \ (0^k,z')\in\dR^k\times Z'.
\end{equation}
 Therefore, taking $R=\epsilon_0^{-1}$ and sufficiently large $i$ such that $\Psi(\delta_i|n,R)<\epsilon_0/2$, we obtain the contradiction.
\end{proof}

With the above Lemmas, we proceed to prove Proposition \ref{quantitative-splitting-at-origin}.

\begin{proof}[Proof of Proposition \ref{quantitative-splitting-at-origin}]  We argue by contradiction. Suppose for some 
$\epsilon_0>0$, no such $\delta>0$ exists. That is, we have a contradicting sequence $(M_i^n,g_i,p_i)$ with $\Ric_{g_i}\geq-(n-1)$ and $B_2(p_i)$ has a compact closure in $B_4(p_i)$, and let $\pi_i:(\widehat{B_{2}(p_i)},\hat{p}_i,G_i)\rightarrow (B_2(p_i),p_i)$ be a sequence of normal covers with the deck transformation groups $G_i$, such that for $\delta_i\rightarrow0$ and for any arbitrary sequence $\{s_i\}$ with $1>s_i>\delta_i$, we have the following properties:

$(i)$ $d_{GH}(B_2(p_i),B_2(0^{k}))<\delta_i\rightarrow0$, $0^{k}\in\dR^{k}$.

$(ii)$ for sufficiently large $i$ (depending only on $n$),  $\rank(\hat{G}_{\delta_i}(p_i))=m_0\geq m$ (passing to a subsequence if necessary), notice that by Theorem \ref{t:small-displacement-property} for sufficiently large $i$, $\hat{G}_{\delta_i}(p_i)$ is always $(w(n),n-k)$-nilpotent.  

$(iii)$ but for any integer $d\geq m$, it holds that for each $i$,
\begin{equation}d_{GH}\Big(B_{s_i}(\hat{p}_i),B_{s_i}(0^{k+d})\Big)\geq s_i\epsilon_0,\ 0^{k+d}\in\dR^{k+d},\label{contradicting-rank-splitting}
\end{equation}
where $\hat{p}_i$ is a lift of $p_i$ on the normal cover.

Let $\epsilon_i\equiv\delta_i^{1/2}$ and rescale the metric $h_i\equiv\epsilon_i^{-2}g_i$, then we have the following properties: the contradicting sequence $(M_i^n,h_i,p_i)$ satisfies $\Ric_{h_i}\geq-(n-1)\epsilon_i^2$ and $B_{2\epsilon_i^{-1}}(p_i,h_i)$
has a compact closure in $B_{4\epsilon_i^{-1}}(p_i,h_i)$ such that for sufficiently large $i$, the following holds,

$(i)'$ $d_{GH}(B_{2\epsilon_i^{-1}}(p_i,h_i),B_{2\epsilon_i^{-1}}(0^{k}))<\epsilon_i\rightarrow0$, $0^{k}\in\dR^{k}$,

$(ii)'$ $\hat{G}_{\epsilon_i}(p_i,h_i)$ is $(w(n),n-k)$-nilpotent with $\rank(\hat{G}_{\epsilon_i}(p_i,h_i))=m_0\geq m$, where 
\begin{equation}
\hat{G}_{\epsilon_i}(p_i,h_i)\equiv\Big\langle\Big\{\gamma_i\in G_i\Big|d_{h_i}(\gamma_i\cdot\hat{p}_i,\hat{p}_i)<2\epsilon_i\Big\}\Big\rangle.
\end{equation}

In the following proof, we will obtain a contradiction to $(iii)$ by applying $(i)'$ and $(ii)'$. 
We start with the first claim.

\vspace{0.5cm}

{\bf{Claim 1.}} Passing to a subsequence, we have the following commutative diagram
\begin{equation}
\xymatrix{
\Big(\widehat{B_{2\epsilon_i^{-1}}(p_i)}, G_i,\hat{p}_i\Big)\ar[rr]^{eqGH}\ar[d]_{\pi_i} &   & \Big(\dR^k\times\dR^d\times Y_0,G_{\infty},(0^{k+d},y_0)\Big)\ar [d]^{\pi_{\infty}}\\
 \Big(B_{2\epsilon_i^{-1}}(p_i)=\widehat{B_{2\epsilon_i^{-1}}(p_i)}\Big/G_i,p_i\Big)\ar[rr]^{pGH} && \Big(\dR^k,0^k\Big),
}\label{hat-covering-diagram}
\end{equation} where $d\geq 0$, $(Y_0,y_0)$ is a compact length space and the normal covering maps $\pi_i$ converge to a submetry $\pi_{\infty}$. Moreover, $G_{\infty}$ acts trivially on $\dR^k$ and thus $G_{\infty}$ acts transitively on $\dR^{d}\times Y_0$.

\vspace{0.5cm}

Lemma \ref{lift} implies that passing to a subsequence, it holds that
\begin{equation}
\Big(\widehat{B_{2\epsilon_i^{-1}}(p_i)},\hat{p}_i\Big)\xrightarrow{pGH}\Big(\dR^{k}\times Y,(0^{k},y)\Big),\ (0^k,y)\in\dR^k\times Y,
\end{equation}
where $(Y,y)$ is a complete length space. If $Y$ is compact, it is done. 
If $Y$ is non-compact, we will prove that $Y$ is isometric to $\dR^d\times Y_0$ with $d\geq 1$ and $Y_0$ compact.

As in the proof of Lemma \ref{lift}, fix $0<R<<2\epsilon_i^{-1}$ and let
\begin{equation}
\Phi_{i,R}\equiv\Big(\Phi_{i,R}^{(1)},\ldots,\Phi_{i,R}^{(k)}\Big): B_{R}(p_i)\rightarrow  B_{R}(0^k)
\end{equation}
be the $\dR^k$-splitting map such that $\triangle\Phi_{i,R}^{(\alpha)}=0$ for each $1\leq \alpha\leq k$. Let 
$\widehat{\Phi}_{i,R}$
be the lifted map of $\Phi_{i,R}$, i.e. $\widehat{\Phi}_{i,R}\equiv\Phi_{i,R}\circ\pi_i$.
Lemma \ref{lift} shows that $\widehat{\Phi}_{i,R}$ is also a $\dR^k$-splitting map.
Therefore, $(Y,y)$ is the Gromov-Hausdorff limit of the lifted level sets 
$\widehat{\Phi}_{i,R}^{-1}(\widehat{\Phi}_{i,R}(\hat{p}_i))$. Moreover, by definition, the level sets $\widehat{\Phi}_{i,R}^{-1}(\widehat{\Phi}_{i,R}(\hat{p}_i))$
keep invariant under the deck transformation group $G_i$. Then the isometry group $G_{\infty}$ can be viewed as an action on $(Y,y)$, i.e. $G_{\infty}\leq \Isom(Y)$, and thus $G_{\infty}$ acts transitively on $Y$ and it acts trivially on $\dR^k$. Since $(Y,y)$ is assumed non-compact, Lemma \ref{Ricci-homoline} implies that the complete non-compact length space $Y$ is isometric to $\dR^d\times Y_0$ with $Y_0$ compact for some $d\geq 1$. We have proved Claim 1.

\vspace{0.5cm}

The main ingredient is to prove the following:
the assumption  \begin{equation}\rank(\hat{G}_{\epsilon_i}(p_i,h_i))=m_0\geq m,\label{rank-at-least-m}\end{equation} gives the inequality \begin{equation}d\geq m_0\geq m,\label{Euclidean-factor-inequality}
\end{equation}
in diagram (\ref{hat-covering-diagram}). 
Inequality (\ref{Euclidean-factor-inequality}) will be proved through the following arguments.

To simplify the arguments, we will blow down the limit space $Y_0$ in diagram (\ref{hat-covering-diagram}). That is, there exists a slowly converging sequence $\mu_{i}\rightarrow0$ and $\mu_i>\epsilon_{i}^{1/2}$ such that under the rescaled metrics $\mathfrak{h}_i\equiv\mu_i^2h_i$ with ${\bf{B}}_i\equiv\mu_i B_{2\epsilon_i^{-1}}(p_i,h_i)$ we have the following commutative diagram:
\begin{equation}
\xymatrix{
\Big(\widehat{{\bf{B}}}_i, G_i,\hat{p}_i\Big)\ar[rr]^{eqGH}\ar[d]_{\pi_i} &   & \Big(\dR^k\times\dR^d,G_{\infty},0^{k+d}\Big)\ar [d]^{\pi_{\infty}}\\
 \Big({\bf{B}}_i,p_i\Big)\ar[rr]^{pGH} && \Big(\dR^k,0^k\Big),
}\label{hat-covering-diagram-2}
\end{equation}
where $G_{\infty}$ acts transitively on $\dR^d$ and trivially on $\dR^k$. 
In the following arguments, the inequality 
$d\geq m$ will be proved by the dimension estimate of the limiting orbit. 
To this end, we will analyze the behavior of the nilpotent subgroup of $G_i$ and its limit in $G_{\infty}$. 

In this paragraph, we sum up the properties of the nilpotent subgroup of $G_i$ which will be applied in the proof of inequality (\ref{Euclidean-factor-inequality}).
Let $\lambda_i\equiv\mu_i\epsilon_i\rightarrow0$, by the contradicting assumption $(ii)'$, under the rescaled metric $\mathfrak{h}_i$, it holds that the group
\begin{equation}
\hat{G}_{\lambda_i}(p_i,\mathfrak{h}_i)\equiv\Big\langle\Big\{g_i\in G_i\Big| d_{\mathfrak{h}_i}(g_i\cdot\hat{p}_i,\hat{p}_i)<\lambda_i\Big\}\Big\rangle
\end{equation}
is $(w(n),n)$-nilpotent. Hence, $\hat{G}_{\lambda_i}(p_i,\mathfrak{h}_i)$
contains a nilpotent subgroup $\mathcal{N}_i$
such that (passing to a subsequence if necessary), 
\begin{equation}
[\hat{G}_{\lambda_i}(p_i,\mathfrak{h}_i):\mathcal{N}_i]\leq w(n),\  \ c_0\equiv\Step(\mathcal{N}_i)\leq\length(\mathcal{N}_i)\leq n. 
\end{equation}
Lemma \ref{w-nilpotent-nice-factors} implies that there exists 
\begin{equation}K=10^n\cdot w(n)<\infty,\end{equation} such that $\mathcal{N}_i$ has the $(m_0, K\cdot\lambda_i)$-displacement property at $\hat{p}_i\in(\widehat{\bf{B}}_i,\hat{p}_i)$. That is,
for each $1\leq s\leq c_0$, let $\mathcal{N}^{(i)}_{s}\equiv[\mathcal{N}^{(i)}_{s-1},\mathcal{N}_i]$, $\mathcal{N}^{(i)}_{0}\equiv\mathcal{N}_i$ and denote \begin{equation}n_s\equiv\rank(\mathcal{N}^{(i)}_{s}/\mathcal{N}^{(i)}_{s-1}),\end{equation}
so the $(m_0, K\cdot\lambda_i)$-displacement property of $\mathcal{N}_i$ implies that for each $1\leq s\leq c_0$, there is a normal series
\begin{equation}
\mathcal{N}^{(i)}_{s-1}\rhd \mathcal{N}^{(i)}_{s,n_s}\rhd\ldots\rhd\mathcal{N}^{(i)}_{s,1}\rhd\mathcal{N}^{(i)}_{s,0}\rhd\mathcal{N}^{(i)}_{s}\label{refined-normal-series}
\end{equation}
such that  for every $1\leq \alpha\leq n_s$ the following holds:

$(a)$ $\mathcal{N}^{(i)}_{s,\alpha}/\mathcal{N}^{(i)}_{s,\alpha-1}\cong\mathbb{Z}$. $\mathcal{N}^{(i)}_{s-1}/\mathcal{N}^{(i)}_{s,n_s}$ and $\mathcal{N}^{(i)}_{s,0}/\mathcal{N}^{(i)}_{s}$ are finite groups.

$(b)$ There are $\{\sigma^{(i)}_{s,1},\ldots,\sigma^{(i)}_{s,n_s}\}\subset\mathcal{N}^{(i)}_{s}$ such that,
\begin{equation}
d(\sigma_{i,s\alpha}\cdot\hat{p}_i,\hat{p}_i)<K\cdot\lambda_i,\label{factor-with-small-displacement} 
\end{equation}
and $\mathcal{N}^{(i)}_{s,\alpha}/\mathcal{N}^{(i)}_{s,\alpha-1}=\langle\sigma^{(i)}_{s,\alpha}\cdot \mathcal{N}^{(i)}_{s,\alpha-1}\rangle\cong\mathbb{Z}$.

The above normal series actually gives a sequence of normal $\mathbb{Z}$-cover.
Denote by 
\begin{equation}\pi^{(i)}_{s,\alpha}:\Big(\widehat{\bf{B}}_i/\mathcal{N}_{i,s(\alpha-1)},p_{i,s(\alpha-1)}\Big)\longrightarrow \Big(\widehat{\bf{B}}_i/\mathcal{N}_{i,s\alpha},p_{i,s\alpha}\Big)\end{equation} 
the normal cover 
induced by the transformation group $\Lambda_{i,s\alpha}\equiv\mathcal{N}_{i,s\alpha}/\mathcal{N}_{i,s(\alpha-1)}=\langle\sigma_{i,s\alpha}\cdot\mathcal{N}_{i,s(\alpha-1)}\rangle.$
Actually, the cyclic normal covers imply the following diagram of equivariant convergence,
\begin{equation}
\xymatrix{
\Big(\widehat{\bf{B}}_i/\mathcal{N}^{(i)}_{s}, p^{(i)}_{s}\Big)\ar@{->}[rr]^{pGH}\ar@{.>}[d]_{} &  & \Big(\dR^{k}\times Z_{s},(0^{k},z_{s})\Big)\ar@{.>} [d]^{}
\\
\Big(\widehat{\bf{B}}_i/\mathcal{N}^{(i)}_{s,\alpha-1},\Lambda^{(i)}_{s,\alpha},p^{(i)}_{s,\alpha-1}\Big)\ar[rr]^{eqGH}\ar[d]_{\pi^{(i)}_{s,\alpha}} &  & \Big(\dR^{k}\times Z_{s,\alpha-1},\Lambda^{(\infty)}_{s,\alpha},(0^{k},z_{s,\alpha-1})\Big)\ar [d]^{\pi^{(\infty)}_{s,\alpha}}
\\ 
\Big(\widehat{\bf{B}}_i/\mathcal{N}^{(i)}_{s,\alpha},p^{(i)}_{s,\alpha}\Big)\ar@{.>}[d]^{}
\ar[rr]^{pGH}& & \Big(\dR^{k}\times  Z_{s,\alpha}, (0^{k},z_{s,\alpha})\Big)\ar@{.>}[d]^{}
\\
\Big(\widehat{\bf{B}}_i/\mathcal{N}^{(i)}_{s-1},p^{(i)}_{s-1}\Big)\ar[rr]^{pGH}&  & \Big(\dR^{k}\times Z_{s-1},(0^{k},z_{s-1})\Big)} \label{cyclic-cover}
\end{equation}
where $Z_{s,\alpha}=Z_{s,\alpha-1}/\Lambda^{(\infty)}_{s,\alpha}$.

\vspace{0.5cm}

{\bf{Claim 2.}} In diagram (\ref{cyclic-cover}), given $1\leq \alpha\leq n_s$, $\dim_{\mathcal{H}}(\Lambda^{(\infty)}_{s,\alpha}\cdot z_{s,\alpha-1})\geq 1$.

\vspace{0.5cm}

By theorem \ref{lie-group}, $\Isom(\dR^k\times Z_{s,\alpha-1})$ is a Lie group.
Since $\Lambda^{(\infty)}_{s,\alpha}$ is a closed subgroup of
$\Isom(\dR^k\times Z_{s,\alpha-1})$ (actually acting trivially on $\dR^k$), 
$\Lambda^{(\infty)}_{s,\alpha}$ is also a Lie group.
Denote by $\Lambda^{0}_{s,\alpha}$ the identity component of 
$\Lambda^{(\infty)}_{s,\alpha}$,
we will show that
\begin{equation}
\dim_{\mathcal{H}}(\Lambda^{0}_{s,\alpha}\cdot z_{s(\alpha-1)})\geq 1.
\end{equation}
By definition, $\Lambda^{0}_{s,\alpha}$ is connected, so is the orbit $\Lambda^{0}_{s,\alpha}\cdot z_{s,\alpha-1}$. It suffices to show that the orbit $\Lambda^{0}_{s,\alpha}\cdot z_{s,\alpha-1}$ contains at least two points. 

First, we give a geometric description of the identity component $\Lambda^{0}_{s,\alpha}$. For each
$\epsilon>0$, define
\begin{equation}
I_{s,\alpha}(\epsilon)\equiv\Big\{ g\in \Lambda^{0}_{s,\alpha}\Big|d(g\cdot y_{s,\alpha-1},y_{s,\alpha-1})<\epsilon,\ \forall y_{s,\alpha-1}\in\overline{B_{\epsilon^{-1}}(z_{s,\alpha-1})}\Big\}.
\end{equation}
then applying Lemma \ref{generating-set-of-identity-component},  there exists $\bar{\epsilon}(Z_{s,\alpha-1})>0$ such that $\langle I_{s,\alpha}(\bar{\epsilon})\rangle=\Lambda^{0}_{s,\alpha}$.
Inequality (\ref{factor-with-small-displacement}) implies that, if $i$ is sufficiently large (depending only on $n$ and $\bar{\epsilon}$), it holds that (where we have identified $\sigma^{(i)}_{s,\alpha}$ with its natural projection on $\mathcal{N}^{(i)}_{s,\alpha}/\mathcal{N}^{(i)}_{s,\alpha-1}$),
\begin{equation}d(\sigma^{(i)}_{s,\alpha}\cdot p^{(i)}_{s,\alpha-1}, p^{(i)}_{s,\alpha-1})<K\cdot\lambda_i<\mu(\bar{\epsilon},n),
 \end{equation}
with $\mu(\bar{\epsilon},n)\equiv\frac{1}{10}\Psi\Big(\frac{\bar{\epsilon}}{2}\Big|n\Big)$, where $\Psi$ is the function in Lemma \ref{every-point-Z}.  Define
\begin{equation}
L_i(\mu/2)\equiv\sup\Big\{\ell\in\mathbb{N}_+\Big| d\Big((\sigma^{(i)}_{s,\alpha})^{\ell}\cdot p^{(i)}_{s,\alpha-1},p^{(i)}_{s,\alpha-1}\Big)\leq\mu/2\Big\},
\end{equation}
and notice that $\langle\sigma^{(i)}_{s,\alpha}\rangle\cong\mathbb{Z}$ acts isometrically and properly on $\widehat{\bf{B}}_i/\mathcal{N}^{(i)}_{s,\alpha-1}$, then
\begin{equation}
L_i(\mu/2)<\infty.
\end{equation}
Consequently, for each $\ell\geq L_i+1$,
\begin{equation}
d((\sigma^{(i)}_{s,\alpha})^{\ell}\cdot p^{(i)}_{s,\alpha-1},p^{(i)}_{s,\alpha-1})>\mu/2.
\end{equation}
Since the generator $\sigma^{(i)}_{s,\alpha}$ satisfies $d(\sigma^{(i)}_{s,\alpha}\cdot  p^{(i)}_{s,\alpha-1}, p^{(i)}_{s,\alpha-1})<\mu$,
then by triangle inequality,
\begin{eqnarray}
\frac{\mu}{2}&<&d((\sigma^{(i)}_{s,\alpha})^{L_i+1}\cdot p^{(i)}_{s,\alpha-1},p^{(i)}_{s,\alpha-1})\nonumber\\
&\leq& d((\sigma^{(i)}_{s,\alpha})^{L_i+1}\cdot p^{(i)}_{s,\alpha-1},(\sigma^{(i)}_{s,\alpha})^{L_i}\cdot p^{(i)}_{s,\alpha-1})+d((\sigma^{(i)}_{s,\alpha})^{L_i}\cdot p^{(i)}_{s,\alpha-1}, p^{(i)}_{s,\alpha-1})\nonumber\\
&=& d(\sigma^{(i)}_{s,\alpha}\cdot p_{i,s(\alpha-1)}, p_{i,s(\alpha-1)})+d((\sigma^{(i)}_{s,\alpha})^{L_i}\cdot p_{i,s(\alpha-1)}, p_{i,s\alpha-1}) \nonumber\\ &<& \mu+\frac{\mu}{2}<\Psi(\bar{\epsilon}|n).\end{eqnarray}
On the other hand, Lemma \ref{every-point-Z} implies that there exists $1\leq a_i\leq N_1(\bar{\epsilon},n)$ such 
that
\begin{equation}
d((\sigma^{(i)}_{s,\alpha})^{ a_i\cdot(L_i+1)}\cdot y^{(i)}_{s,\alpha-1},y^{(i)}_{s,\alpha-1})<\frac{\bar{\epsilon}}{2}, \ \forall y^{(i)}_{s,\alpha-1}\in B_{\bar{\epsilon}^{-1}}(p_{i,s(\alpha-1)}).
\end{equation}
Moreover, due to $a_i\cdot(L_i+1)>L_i$,
\begin{equation}
d((\sigma^{(i)}_{s,\alpha})^{ a_i\cdot(L_i+1)}\cdot p^{(i)}_{s,\alpha-1},p^{(i)}_{s,\alpha-1})>\frac{\mu}{2}.
\end{equation}
Denote $g^{(i)}_{s,\alpha}\equiv(\sigma^{(i)}_{s,\alpha})^{ a_i\cdot(L_i+1)}$, then the equivariant convergence 
\begin{equation}
\Big(\widehat{\bf{B}}_i/\mathcal{N}^{(i)}_{s,\alpha-1},\Lambda^{(i)}_{s,\alpha},p^{(i)}_{s,\alpha-1}\Big)\xrightarrow{eqGH}\Big(\dR^{k}\times Z_{s,\alpha-1},\Lambda^{(\infty)}_{s,\alpha},(0^{k},z_{s,\alpha-1})\Big)
\end{equation}
gives that there exists $g_{s,\alpha}\in\Lambda^{(\infty)}_{s,\alpha}$ such that
\begin{equation}
d(g_{s,\alpha}\cdot y_{s,\alpha-1},y_{s,\alpha-1})<\bar{\epsilon}, \forall y_{s,\alpha-1}\in B_1(z_{s,\alpha-1})
\end{equation}
and 
\begin{equation}
d(g_{s,\alpha}\cdot z_{s,\alpha-1},z_{s,\alpha-1})>\frac{\mu(\bar{\epsilon},n)}{4}.\label{definitely-move}
\end{equation}
The above two inequalities imply that $g_{s,\alpha}$ is a non-trivial element in $\langle I_{s,\alpha}(\bar{\epsilon})\rangle=\Lambda_{s,\alpha}^0$ and inequality (\ref{definitely-move}) shows that the orbit $\Lambda_{s,\alpha}^0\cdot z_{z,\alpha-1}$ contains at least two points. By definition, the identity component $\Lambda_{s,\alpha}^0$ is connected,
then the connected orbit $\Lambda_{s,\alpha}^0\cdot z_{s,\alpha-1}$ contains a non-trivial continuous path connecting $z_{s,\alpha-1}$ and $g_{s,\alpha}\cdot z_{s,\alpha-1}$. Therefore, for each $1\leq \alpha\leq n_s$,
 \begin{equation}\dim_{\mathcal{H}}(\Lambda^{(\infty)}_{s,\alpha}\cdot z_{s,\alpha-1})\geq\dim_{\mathcal{H}}(\Lambda_{s,\alpha}^0\cdot z_{s,\alpha-1})\geq 1.\end{equation}
So we have proved Claim 2.

\vspace{0.5cm}

With Claim 2, take $s=c_0$ in diagram (\ref{cyclic-cover}), then $Z_{c_0}=\dR^d$. We proceed to prove the dimension estimate 
\begin{equation}
\dim_{\mathcal{H}}(Z_{c_0})\geq m.
\end{equation}
In fact, in the context of quotients of Euclidean space, it is standard that in diagram (\ref{cyclic-cover}),
\begin{equation}\dim_{\mathcal{H}}(Z_{s,\alpha-1})\geq\dim_{\mathcal{H}}(\Lambda^{(\infty)}_{s,\alpha}\cdot z_{s,\alpha-1})+\dim_{\mathcal{H}}(Z_{s,\alpha}),\end{equation}
which implies that
\begin{eqnarray}
\dim_{\mathcal{H}}(Z_{s})\geq\dim_{\mathcal{H}}(Z_{s,0})&\geq& \sum\limits_{\alpha=1}^{n_s}\dim_{\mathcal{H}}(\Lambda^{(\infty)}_{s,\alpha}\cdot z_{s,\alpha-1})+\dim_{\mathcal{H}}(Z_{s,n_s})\nonumber\\
&\geq& n_s+\dim_{\mathcal{H}}(Z_{s,0})\nonumber\\
&\geq& n_s+\dim_{\mathcal{H}}(Z_s).
\end{eqnarray}
Therefore,
\begin{equation}
\dim_{\mathcal{H}}(Z_{c_0})\geq\sum\limits_{s=1}^{c_0}n_s+\dim_{\mathcal{H}}(Z_{0})= m_0+\dim_{\mathcal{H}}(Z_{0})\geq m,
\end{equation}
where $Z_0\equiv\dR^d/\mathcal{N}_{\infty}$ and by assumption $m_0\geq m$.

Combining Claim 1 and Claim 2, we have obtained that

\begin{equation}
\Big(\widehat{B_{2\epsilon_i^{-1}}(p_i)}, G_i,\hat{p}_i\Big)\xrightarrow{eqGH} \Big(\dR^k\times\dR^d\times Y_0,G_{\infty},(0^k,y_0)\Big), \ d\geq m,
\end{equation}
where $(Y_0,y_0)$ is a compact length space.
It gives that for each $10<R<<2\epsilon_i^{-1}$ and for some compact length space $Y_0$,
\begin{equation}
d_{GH}(B_R(\hat{p}_i),B_R(0^{k+d},y_0))<\Psi(\epsilon_i|n,R)\rightarrow0,\ d\geq m.
\end{equation}
We have proved inequality (\ref{Euclidean-factor-inequality}).

\vspace{0.5cm}

We will finish the proof of the Proposition.
Rescaling to the original metrics $g_i$, it holds that\begin{equation}
d_{GH}(B_{R\epsilon_i}(\hat{p}_i),B_{R\epsilon_i}(0^{k+d},y_i))<\epsilon_i\cdot\Psi(\epsilon_i|n,R),\ (0^{k+d},y_i)\in\dR^{k+d}\times Y_i,
\end{equation}
where $(Y_i,y_i)\equiv(\epsilon_i Y_0,y_0)$.
Therefore, 
\begin{eqnarray}
d_{GH}\Big(B_{R\epsilon_i}(\hat{p}_i),B_{R\epsilon_i}(0^{k+d})\Big)&\leq & d_{GH}\Big(B_{R\epsilon_i}(\hat{p}_i),B_{R\epsilon_i}(0^{k+d},y_i)\Big)\nonumber\\ &+& d_{GH}\Big(B_{R\epsilon_i}(0^{k+d},y_i),B_{R\epsilon_i}(0^{k+d})\Big)\nonumber\\
&<& R\epsilon_i\Big(\Psi(\epsilon_i|n,R)+\frac{\diam(Y_0)}{R}\Big).
\end{eqnarray}
Take some large $R>10^3$ such that $\frac{\diam(Y_0)}{R} <\frac{\epsilon_0}{4}$. By definition $\epsilon_i\equiv\delta_i^{1/2}>\delta_i$, then for sufficiently large $i$ such that $\Psi(\epsilon_i|n,R)<\frac{\epsilon_0}{4}$,
\begin{equation}
d_{GH}\Big(B_{R\delta_i^{1/2}}(\hat{p}_i),B_{R\delta_i^{1/2}}(0^{k+d})\Big)<(R\delta_i^{1/2})\cdot\frac{\epsilon_0}{2},\ 0^{k+d}\in\dR^{k+d}, \ d\geq m,
\end{equation}
which gives a contradiction to (\ref{contradicting-rank-splitting}).

\end{proof}

Now we extend the above splitting to a more general context. 
That is,  the limit space of the base manifolds is assumed of the form $\dR^k\times Z$, where $Z$ is any arbitrary metric space. In particular, the metric space $Z$ is not necessarily compact. In this general case, Fukaya-Yamaguchi's type line splitting strategy fails because their method essentially depends on the compactness of $Z$ (see Lemma \ref{homoline} for the reason). 
In exchange, we assume that the normal cover $(\widehat{B_2(p)},\hat{p})$ is non-collapsed. The non-collapsing condition enables us to argue in a different manner. That is, we will apply the cone splitting principle to produce $\dR$-factors on the normal covers. Precisely, we have the following Proposition:

\begin{proposition}\label{non-collapsed-splitting}
Given $\epsilon>0$, $n\geq 2$, $v>0$, there exists $\delta_0(\epsilon,n,v)>0$ such that the following holds. If $(M^n,g,p)$ is a Riemannian manifold which satisfies $\Ric\geq-(n-1)$ and $B_2(p)$ has a compact closure in $B_4(p)$, and if

$(i)$ $d_{GH}(B_2(p),B_2(0^{k},z))<\delta_0, \ (0^{k},z)\in\dR^{k}\times Z$,
where $(Z,z)$ is a complete length space,

$(ii)$ $\pi:(\widehat{B_2(p)},\hat{p},G)\rightarrow(B_2(p),p)$
is the normal cover of $B_2(p)$ with $\pi(\hat{p})=p$ and with the deck transformation group $G$ such that the group
$\hat{G}_{\delta_0}(p)\equiv\langle\{g\in G|d(g\cdot\hat{p},\hat{p})<2\delta\}\rangle
$ satisfies $\rank(\hat{G}_{\delta_0}(p))\geq m$,

$(iii)$ $\Vol(B_1(\hat{p}))\geq v>0$,
\\
then for some $\delta_0<r<1$, $d(p)\geq m$,
\begin{equation}
d_{GH}(B_{r}(\hat{p}),B_{r}(0^{k+d(p)},\hat{z}))<r\epsilon,\ (0^{k+d(p)},\hat{z})\in\dR^{k+d(p)}\times C(\hat{Z}),
\end{equation}
where $(C(\hat{Z}),\hat{z})$ is a metric cone over some compact metric space $\hat{Z}$ with cone tip $\hat{z}$.
\end{proposition}

\begin{proof}
We will prove this Proposition by contradiction. Suppose the statement fails. That is, for some $\epsilon_0>0$, $v_0>0$, there is a sequence of Riemannian manifolds $(M_i^n,g_i,p_i)$
with $\Ric_{g_i}\geq-(n-1)$,  
such that $B_{2}(p_i)$ has a compact closure in $B_{4}(p_i)$, and there is a sequence of complete length spaces  $(Z_i,z_i)$, and for any sequence $s_i>\delta_i$ the following properties hold (passing to a subsequence if necessary).

$(i)$ $d_{GH}(B_2(p_i),B_2(0^{k},z_i))<\delta_i, \ (0^{k},z_i)\in\dR^{k}\times Z_i$,
for some complete length space $(Z_i,z_i)$,

$(ii)$ $\pi_i:(\widehat{B_2(p_i)},\hat{p}_i,G_i)\rightarrow(B_2(p_i),p_i)$
is the normal cover of $B_2(p_i)$ with $\pi(\hat{p}_i)=p_i$ and with the deck transformation group $G$ such that the group
$\hat{G}_{\delta_i}(p_i)\equiv\langle\{g_i\in G_i|d(g_i\cdot\hat{p}_i,\hat{p}_i)<2\delta_i\}\rangle
$ satisfies $\rank(\hat{G}_{\delta_i}(p_i)) =m_0\geq m$,

$(iii)$ $\Vol(B_1(\hat{p}_i))\geq v_0>0$,
\\
but

$(iv)$ for any $m'\geq m$ and for any metric cone $(C(\hat{Z}),\hat{z})$,
\begin{equation}
d_{GH}(B_{s_i}(\hat{p}_i),B_{s_i}(0^{k+m'},\hat{z}))>s_i\cdot\epsilon_0,\ (0^{k+m'},\hat{z})\in\dR^{k+m'}\times C(\hat{Z}).
\end{equation}

Let $\epsilon_i\equiv\delta_i^{1/2}$ and rescale the metric $h_i\equiv\epsilon_i^{-2}g_i$, then we have the following properties: the contradicting sequence $(M_i^n,h_i,p_i)$ satisfies $\Ric_{h_i}\geq-(n-1)\epsilon_i^2$ and $B_{2\epsilon_i^{-1}}(p_i,h_i)$
has a compact closure in $B_{4\epsilon_i^{-1}}(p_i,h_i)$ such that the following holds,

$(i)'$ by Gromov's precompactness theorem, there is some complete length space $(Z_0',z_0')$ such that \begin{equation}d_{GH}(B_{2\epsilon_i^{-1}}({p}_i,h_i),B_{2\epsilon_i^{-1}}(0^{k},z_0'))<\epsilon_i\rightarrow0, (0^{k},z_0')\in\dR^{k}\times Z_0',\end{equation}

$(ii)'$ $\hat{G}_{\epsilon_i}(p_i,h_i)\equiv\langle \{g_i\in G_i|d_{h_i}(g_i\cdot\hat{p}_i,\hat{p}_i)<2\epsilon_i\}\rangle$ has rank $m_0\geq m$,

$(iii)'$
$\Vol_{h_i}(B_1(p_i,h_i))\geq C(n)\cdot v_0>0,
$

$(iv)'$  for any $m'\geq m$, any sequence $s_i>\delta_i$ and for any metric cone $(C(\hat{Z}),\hat{z})$,
\begin{equation}
d_{GH}(B_{s_i\cdot\epsilon_i^{-1}}(\hat{p}_i,h_i),B_{s_i\cdot\epsilon_i^{-1}}(0^{k+m'},\hat{z}))>s_i\cdot\epsilon_i^{-1}\cdot\epsilon_0,\ (0^{k+m'},\hat{z})\in\dR^{k+m'}\times C(\hat{Z}).
\end{equation}

In the following proof, we will obtain a contradiction to $(iv')$ by applying $(i)'$, $(ii)'$ and $(iii)'$.
By $(iii)'$, $\widehat{B_{2\epsilon_i^{-1}}(p_i)}$ is non-collapsed and thus every tangent cone on the limit space is a metric cone.
So
we can choose a slowly converging  
sequence $\mu_i\rightarrow0$ with $\mu_i>\epsilon_i^{1/2}$ such that under the rescaled metrics $\mathfrak{h}_i\equiv\mu_i^{-2}h_i$, the following diagram holds,

\begin{equation}
\xymatrix{
\Big(\widehat{{\bf{B}}}_i, G_i,\hat{p}_i\Big)\ar[rr]^{eqGH}\ar[d]_{\pi_i} &   & \Big(\dR^k\times C(Y),G_{\infty},(0^k,y^*)\Big)\ar [d]^{\pi_{\infty}}\\
 \Big({\bf{B}}_i,p_i\Big)\ar[rr]^{pGH} && \Big(\dR^k\times Z_0'',(0^k,z_0'')\Big),
}\label{tangent-cone-diagram}
\end{equation}
where ${\bf{B}}_i\equiv\mu_i^{-1}B_{2\epsilon_i^{-1}}(p_i)$, $(C(Y),y^*)$ is a metric cone over some compact length space $Y$ with the cone vertex $y^*$, and $G_{\infty}\in\Isom(C(Y))$ is a Lie group. By $(ii)'$, then
\begin{equation}
\hat{G}_{\lambda_i}(\hat{p}_i,\mathfrak{h}_i)\equiv\langle\{g_i\in G_i|d_{\mathfrak{h}_i}(g_i\cdot\hat{p}_i,\hat{p}_i)<2\lambda_i\}\rangle
\end{equation}
satisfies 
 $\rank(\hat{G}_{\lambda_i}(p_i,\mathfrak{h}_i))=m_0\geq m$ with $\lambda_i\equiv\mu_i^{-1}\cdot\epsilon_i<\epsilon_i^{1/2}\rightarrow0$.

Since the Lie group $G_{\infty}$ isometrically acts on $C(Y)$, then each point $y\in G_{\infty}\cdot y^*$ is a cone tip. If the orbit $G_{\infty}\cdot y^*$ is not a point, Lemma \ref{cone-splitting-principle} implies that for some $s\geq 1$,
\begin{equation}
C(Y)\cong\dR^d\times C(\hat{Z}),\ \Gamma_{\infty}\cdot y^*\subset\dR^d,\end{equation}
where the metric cone $C(\hat{Z})$ does not admit any line. In fact, we will prove that 
\begin{equation}d\geq m,\label{euclidean-factor-esimtate}\end{equation}
by applying the assumption  
$\rank(\hat{G}_{\lambda_i}(p_i,\mathfrak{h}_i))=m_0\geq m$. The proof of inequality (\ref{euclidean-factor-esimtate}) is similar to the proof of inequality (\ref{Euclidean-factor-inequality}) in Proposition \ref{quantitative-splitting-at-origin}.
Inequality (\ref{euclidean-factor-esimtate}) easily follows from the following Claim:

\vspace{0.5cm}

{\bf{Claim.}} If $\rank(G_{\lambda_i}(p_i,\mathfrak{h}_i))=m$, in digram (\ref{tangent-cone-diagram}), the inequality $\dim_{\mathcal{H}}(G_{\infty}\cdot y^*)\geq m$
holds on the limit cone $(C(Y),y^*)$.

\vspace{0.5cm}

The proof of this Claim
is very similar to Claim 2 in the proof of Proposition \ref{quantitative-splitting-at-origin}. 
We have proved in Theorem \ref{t:small-displacement-property} that for sufficiently large $i$, $\hat{G}_{\lambda_i}(p_i,\mathfrak{h}_i)$ contains a nilpotent subgroup $\mathcal{N}_{i}$ such that (passing to a subsequence if necessary),
\begin{equation}
\rank(\mathcal{N}_i)=m_0\leq n,\ \ c_0\equiv\Step(\mathcal{N}_i)\leq \length(\mathcal{N}_i)\leq n\ ,\ [\hat{G}_{\lambda_i}(p_i,\mathfrak{h}_i):\mathcal{N}_i]\leq w(n),
\end{equation}
and $\mathcal{N}_i$ satisfies the 
$(m_0,K(n)\cdot\lambda_i)$-displacement property for some constant $K=K(n)>0$. Let $\mathcal{N}_{\infty}$ be the equivariant limit of $\mathcal{N}_i$.
Now it suffices to prove that \begin{equation}\dim_{\mathcal{H}}(\mathcal{N}_{\infty}\cdot y^*)\geq m.\label{cone-orbit-dimension-estimate}\end{equation} 
To this end, consider the following diagrams:
\begin{equation}
\xymatrix{
\Big(\widehat{\bf{B}}_i,G_i,\tilde{p}_{i}\Big)\ar[rr]^{eqGH}\ar[d]_{\pr^{(i)}_{s,\alpha}} &  & \Big(\dR^{k}\times\dR^{d}\times C(\hat{Z}),G_{\infty},(0^{k},0^{d},\hat{z})\Big)\ar [d]^{\pr^{(\infty)}_{s,\alpha}}\\
\Big(\widehat{\bf{B}}_i/\mathcal{N}^{(i)}_{s,\alpha},\Lambda^{(i)}_{s,\alpha+1},p^{(i)}_{s,\alpha}\Big)\ar[rr]^{eqGH}\ar[d]_{\varphi^{(i)}_{s,\alpha+1} }&  & \Big(\dR^{k}\times Z_{s,\alpha},\Lambda^{(\infty)}_{s,\alpha+1},(0^{k},z_{s,\alpha})\Big)\ar [d]^{\varphi^{(\infty)}_{s,\alpha+1}}\\
\Big(\widehat{\bf{B}}_i/\mathcal{N}^{(i)}_{s,\alpha+1},{p}^{(i)}_{s,\alpha+1}\Big)\ar[rr]^{pGH}& & \Big(\dR^{k}\times Z_{s,\alpha+1}, (0^{k},z_{s,\alpha+1})\Big),
}\label{cone-splitting-diagram}\end{equation}
where $Z_{s,\alpha}\equiv C(Y)/\mathcal{N}^{(\infty)}_{s,\alpha}$, $\mathcal{N}^{(\infty)}_{s,\alpha}=\lim\limits_{i\rightarrow}\mathcal{N}^{(i)}_{s,\alpha}$, $\Lambda^{(i)}_{s,\alpha}\equiv\mathcal{N}^{(\infty)}_{s,\alpha+1}/\mathcal{N}^{(\infty)}_{s,\alpha}$ and
$\Lambda^{(\infty)}_{s,\alpha}\equiv\lim\limits_{i\rightarrow\infty}\Lambda^{(i)}_{s,\alpha}$. 
Note that the above diagram is commutative. In particular, 
\begin{equation}\pr^{(i)}_{s,\alpha+1}=\varphi^{(i)}_{s,\alpha+1}\circ\pr^{(i)}_{s,\alpha},\ \pr^{(\infty)}_{s,\alpha+1}=\varphi^{(\infty)}_{s,\alpha+1}\circ\pr^{(\infty)}_{s,\alpha}.\label{composed-projections}\end{equation}

In the following arguments, we concentrate on the slice $\dR^d\times\{\hat{z}\}$. First, we show that
the limit Lie group $G_{\infty}$ preserves the slice $\dR^d\times\{\hat{z}\}$ in diagram (\ref{cone-splitting-diagram}), where $\hat{z}$ is the cone vertex of $C(\hat{Z})$. That is, 
$G_{\infty}\leq\Isom(\dR^d\times C(\hat{Z}))$ induces an isometric Lie group action on $\dR^d\times\{\hat{z}\}$. In fact, if not, suppose there exists $(v_0,\hat{z})\in \dR^d\times\{\hat{z}\}$, $g_{\infty}\in G_{\infty}$ such that $g_{\infty}((v_0,\hat{z}))\not\in\dR^d\times\{\hat{z}\}$. 
Recall that the metric cone $C(\hat{Z})$ is chosen such that it does not contain any line.  
Moreover, $g_{\infty}(\dR^d\times C(\hat{Z}))$ is a metric cone with the cone tip $g_{\infty}((v_0,\hat{z}))$ because $g_{\infty}\in G_{\infty}$ is an isometry and 
$\dR^d\times C(\hat{Z})$ is a metric cone with vertex $(v_0,\hat{z})$. By Lemma \ref{cone-splitting-principle}, there exists $d'\geq 1$, a metric cone $C(W)$ such that $C(\hat{Z})$ is isometric to $\dR^{d'}\times C(W)$. This contradicts to the choice of the integer $d$ and $C(\hat{Z})$. 
 
 We denote 
$U_{s,\alpha}\equiv\pr_{s,\alpha}^{(\infty)}(\dR^d\times\{\hat{z}\}).
$ The proof of the dimension estimate (\ref{cone-orbit-dimension-estimate}) follows from the following inequality, for each $1\leq s\leq c_0$ and $0\leq \alpha\leq n_s-1$,
\begin{equation}
\dim_{\mathcal{H}}(U_{s,\alpha})\geq \dim_{\mathcal{H}}(U_{s,\alpha+1})+1.\label{inductive-dimension-inequality}
\end{equation}
If (\ref{inductive-dimension-inequality}) holds, then it is straightforward that
\begin{equation}
\dim_{\mathcal{H}}(U_{s,0})\geq \dim_{\mathcal{H}}(U_{s,n_s})+n_s.
\end{equation}
Moreover, the normal series gives that 
\begin{equation}
\dim_{\mathcal{H}}(U_{s,n_s})\geq \dim_{\mathcal{H}}(U_{s-1,0}).
\end{equation}
Therefore, \begin{eqnarray}
\dim_{\mathcal{H}}(\dR^d)\geq 
\dim_{\mathcal{H}}(U_{c_0,0})&\geq & \Big(\dim_{\mathcal{H}}(U_{c_0,0})-\dim_{\mathcal{H}}(U_{1,n_1})\Big)\nonumber\\
&=&\sum\limits_{s=1}^{n_{c_0}}\Big((\dim_{\mathcal{H}}(U_{s,0})- \dim_{\mathcal{H}}(U_{s,n_s}))+(\dim_{\mathcal{H}}(U_{s,n_s})-\dim_{\mathcal{H}}(U_{s-1,0}))\Big)\nonumber\\
&\geq& \sum\limits_{s=1}^{n_{c_0}}(\dim_{\mathcal{H}}(U_{s,0})- \dim_{\mathcal{H}}(U_{s,n_s}))\nonumber\\
&\geq &\sum\limits_{s=1}^{c_0}n_s=m_0\geq m. \end{eqnarray}
Then we are done. 
Now we focus on the proof of inequality (\ref{inductive-dimension-inequality}). 
Notice that, equation (\ref{composed-projections}) gives
\begin{equation}
U_{s,\alpha+1}=\pr^{(\infty)}_{s,\alpha+1}(\dR^d\times\{z\})=\varphi^{(\infty)}_{s,\alpha+1}\circ\pr^{(\infty)}_{s,\alpha}(\dR^d\times\{z\})=\varphi^{(\infty)}_{s,\alpha+1}(U_{s,\alpha}).
\end{equation}
Notice that $\Lambda_{s,\alpha}^{(\infty)}$ gives an isometry Lie group action on $U_{s,\alpha}$, then it is standard that
\begin{equation}
\dim_{\mathcal{H}}(U_{s,\alpha})\geq \dim_{\mathcal{H}}(U_{s,\alpha+1})+\dim_{\mathcal{H}}(\Lambda^{(\infty)}_{s,\alpha+1}\cdot z_{s,\alpha}).
\end{equation}
By the same arguments of Claim 2 in the proof of Proposition \ref{quantitative-splitting-at-origin}, it holds that
\begin{equation}
\dim_{\mathcal{H}}(\Lambda^{(\infty)}_{s,\alpha+1}\cdot z_{s,\alpha})\geq 1.\label{1-dim-cyclic-orbit}\end{equation}
Specifically, inequality (\ref{1-dim-cyclic-orbit}) follows from the facts that the identity component of $\Lambda^{(\infty)}_{s,\alpha+1}$ is non-trivial and the corresponding orbit contains at least two points. 
Therefore, we have prove inequality (\ref{inductive-dimension-inequality}), and the proof of the Claim is complete.

\vspace{0.5cm}

The proof of this Proposition is almost finished. In fact, for sufficiently large $i$, 
\begin{equation}
d_{GH}(B_1(\hat{p}_i,\mathfrak{h}_i), B_1(0^{k+s},\hat{z}))<\epsilon_0,\ \hat{z}\in(C(\hat{Z}),\hat{z}),
\end{equation}
and by rescaling back to the metrics $h_i=\mu_i^2\mathfrak{h}_i$,
\begin{equation}
d_{GH}(B_{\mu_i}(\hat{p}_i), B_{\mu_i}(0^{k+s},\hat{z}))<\epsilon_0\cdot\mu_i,\ \hat{z}\in(C(\hat{Z}),\hat{z}),
\end{equation}
where $(C(\hat{Z}),\hat{z})$ is a metric cone over some compact metric space $\hat{Z}$ with cone tip $\hat{z}$.
But the above inequality contradicts $(iv)'$ because $\mu_i>\epsilon_i^{1/2}$.

\end{proof}

\subsection{Quantitative Splitting and Non-Collapse of Universal Covers}
\label{ss:non-collapse-universal-cover}

In this subsection, we prove the non-collapsing of the universal cover in Theorem  \ref{non-collapsed-splitting-maximal-rank} by using those technical tools developed in the above subsections.

First, we prove the non-localness of the quantitative splitting on the universal cover. That is, if the quantitative splitting in Proposition \ref{quantitative-splitting-at-origin} holds at one point, the splitting in fact holds at all nearby points. The non-localness of the quantitative splitting property follows from the non-localness of the nilpotency rank of the fibered fundamental group (See Lemma \ref{every-point-rank}).
This non-local quantitative splitting result will be applied in the proof of the non-collapse of the universal cover.

\begin{proposition}\label{non-local-splitting}
Let $(M^n,g,p)$ be a Riemannian manifold with $\Ric \geq -(n-1)$ and such that $B_2(p)$ has a compact closure in $B_4(p)$.  For each $\epsilon>0$, there exists
$\delta_0=\delta_0(n,\epsilon)>0$, such that if

$(i)$ $d_{GH}(B_2(p),B_2(0^{k}))<\delta_0,\ 0^{k}\in\dR^k$,

$(ii)$ $\Gamma_{\delta_0}(p)\equiv\Image[\pi_1(B_{\delta_0}(p))\rightarrow\pi_1(B_2(p))]$ satisfies 
$\rank(\Gamma_{\delta_0}(p))\geq m$, 
\\
then there is some $\delta_0<r<1$ such that for each $x\in B_1(p)$ and for some integer (depending on $x$) $d(x)\geq m$, it holds that 
\begin{equation}
d_{GH}\Big(B_{r}(\tilde{x}),B_{r}(0^{k+d(x)})\Big)<r\epsilon,\ 0^{k+d(x)}\in\dR^{k+d(x)},\end{equation}
where $\tilde{x}$ is a lift of $x$ on the universal cover of $B_1(x)$.\end{proposition}

\begin{proof}

Fix $0<{\delta}<\epsilon_1(n)$ which is to be determined later, where $\epsilon_1(n)$ is the same constant in Lemma \ref{every-point-rank}. Then we define the positive constant $0<\delta_0<<1$ by letting \begin{equation}\delta_0\equiv\Psi_0(\delta|n)<<\delta<\epsilon_1(n)\end{equation} where $\Psi_0$ is the function in Lemma \ref{every-point-rank}. Property $(ii)$ of Lemma \ref{every-point-rank} shows that for every $x\in B_1(p)$, $\Gamma_{\delta}(x)$ is 
$(w(n),n)$-nilpotent with $\rank(\Gamma_{\delta}(x))\geq m$.
Moreover, if the Gromov-Hausdorff control $(i)$ holds for $\delta_0>0$, then for every $x\in B_{1}(p)$,
\begin{equation}d_{GH}(B_1(x),B_1(0^{k}))<\delta_0<<\delta,\ 0^k\in\dR^k.\label{delta-close}\end{equation}
Therefore, $B_1(x)$ satisfies the assumptions on the Gromov-Hausdorff control and the nipotency rank in Proposition \ref{quantitative-splitting-at-origin}.
 Hence by choosing the above $\delta=\delta(\epsilon,n)>0$ as the corresponding constant in Proposition \ref{quantitative-splitting-at-origin},
and let assumptions $(i)$ and $(ii)$ hold for ${\delta}_0\equiv\Psi_0(\delta|n)>0$,
then for some $1>r\geq \delta(n,\epsilon)>\delta_0(n,\epsilon)>0$,
\begin{equation}d_{GH}(B_{r}(\tilde{x}),B_{r}(0^{k+d(x)}))<r\epsilon, \ 0^{k+d(x)}\in\dR^{k+d(x)},\ d(x)\geq m,
\end{equation}
where $\tilde{x}$ is a lift of $x$ on the universal cover of $B_1(x)$.
\end{proof}

Recall that Theorem \ref{t:KW_almost_nilpotent} states that the nilpotency rank of the fibered fundamental group is bounded by the collapsed dimension. We will prove that if the nilpotency rank attains the maximum, i.e. is equal to the collapsed dimension, then the universal cover is non-collapsed. Precisely, we have the following: 

\begin{proposition}
\label{non-collapsed-universal-cover} Let $(M^n,g,p)$ be a Riemannian manifold with $\Ric\geq-(n-1)$ and such that $B_2(p)$ has a compact closure in $B_4(p)$.  Given a Ricci-limit space $(Z^{k},z^{k})$ with $\dim Z^{k}=k$ in the sense of theorem \ref{limiting-dimension}, there exists $\delta_0=\delta_0(n,B_1(z^{k}))>0$, $v_0=v_0(n,B_1(z^{k}))>0$ such that if

$(i)$ $d_{GH}(B_2(p),B_2(z^{k}))<\delta_0,\ z^k\in Z^k$,

$(ii)$ $\Gamma_{\delta_0}(p)\equiv\Image[\pi_1(B_{\delta_0}(p))\rightarrow\pi_1(B_2(p))]$ satisfies $\rank(\Gamma_{\delta_0}(p))=n-k$,
\\
then for any $x\in B_1(p)$ ,\begin{equation}\Vol(B_{1/2}(\tilde{x}))\geq v_0(n,B_1(z^{k}))>0,
 \end{equation}
 where $\tilde{x}$ is a lift of $x$ on the universal cover of $B_1(x)$.

\end{proposition}

\begin{remark}
This non-collapse result works in a very general setting, that is, we only assume uniform lower Ricci curvature bound and maximal nilpotency rank. In particular, the limit space in assumption $(i)$
is for any $k$-dimensional Ricci-limit space in the sense of theorem \ref{limiting-dimension}. 
\end{remark}

\begin{proof}

It suffices to prove the statement at $p$. The general case follows from the non-localness of the nilpotency rank of the fibered fundamental group.
Let assumptions $(i)$, $(ii)$  hold for some $\delta_0>0$ which will be determined later.  Due to theorem \ref{limiting-dimension}, the regular set $\mathcal{R}_k(Z^k)$ is of full measure (the limiting renormalized measure). 
Given any $\delta>0$, we can define the positive constant 
\begin{equation}
s=s(\delta,B_1(z^{k}))\equiv\sup\Big\{r>0\Big|\exists\ y\in B_{1/10}(z^k)\ \text{such that}\ d_{GH}(B_r(y),B_r(0^k))<\frac{1}{2}r\delta\Big\},
\end{equation}
then $s(\delta,B_1(z^k))>0$. Consequently, there exists some $z_0\in B_{1/5}(z^k)$ such that
\begin{equation}
d_{GH}(B_s(z_0),B_s(0^k))\leq\frac{1}{2}s\delta.
\end{equation}
Since 
$d_{GH}(B_2(p),B_2(z^k))<\delta_0<<1$, we can pick some $q\in B_{1/4}(p)$ such that
\begin{equation}
d_{GH}(B_s(q),B_s(z_0))<2\delta_0.
\end{equation}
Let  $\delta_0>0$ satisfy 
\begin{equation}\delta_0<\frac{1}{10}\Psi_0(s\delta|n)<<s\delta<s\epsilon_1,\end{equation}
where $\epsilon_0$, $\Psi_0$ are in Lemma \ref{every-point-rank}.
It follows that,\begin{equation}
d_{GH}(B_s(q),B_s(0^k))<2\delta_0+\frac{1}{2}s\delta<s\delta.\label{se-GH}
\end{equation}

Let $\pi:(\widetilde{B_2(p)},\tilde{p})\rightarrow B_2(p)$ be the universal cover, then 
the pre-image $(\pi^{-1}(B_{s}(q)),\tilde{q})$ on $\widetilde{B_2(p)}$
is a normal cover of $B_s(q)$. Note that the deck transformation group of this normal cover is 
$G\equiv\pi_1(B_2(p))$.
Let 
\begin{equation}
\hat{G}_{s\delta}(q)\equiv\Big\langle\Big\{\gamma\in G\Big|d(\gamma\cdot\tilde{q},\tilde{q})<s\delta\Big\}\Big\rangle,
\end{equation}
and then
\begin{equation}
\hat{G}_{s\delta}(q)\leq 
\Image[\pi_1(B_{s\delta}(q))\rightarrow\pi_1(B_2(p))].
\end{equation}
Theorem \ref{Generalized-Margulis-Lemma} shows that if $\delta$ is sufficiently small, $\hat{G}_{s\delta}(q)$ is $(w(n),n)$-nilpotent.
Moreover,  
$\rank(\Gamma_{\delta_0}(p))=n-k$ and $\delta_0<\Psi_0(s\delta|n)$, 
then Lemma \ref{every-point-rank} implies that  
$\rank(\hat{G}_{s\delta}(q))\geq n-k$, 
and thus $\rank(\hat{G}_{s\delta}(q))= n-k$.
 Now fix $\epsilon>0$ and combined with (\ref{se-GH}), if  the above $\delta>0$ is chosen as the one in Proposition \ref{quantitative-splitting-at-origin}, then for some $\delta(n,\epsilon)<r<1$ we have that, 
\begin{equation}
d_{GH}(B_{sr}(\tilde{q}), B_{sr}(0^{k+d(q)}))<sr\epsilon,\ 0^{k+d(q)}\in\dR^{k+d(q)},\ d(q)\geq n-k.\label{R^n-convergence}\end{equation}
Immediately, the inequality $k+d(q)\geq k+(n-k)=n$ and the curvature condition $\Ric\geq-(n-1)$ implies that $k+d(q)=n$, and thus 
\begin{equation}
d_{GH}(B_{sr}(\tilde{q}), B_{sr}(0^{n}))<sr\epsilon,\ 0^n\in\dR^n.
\end{equation}
Applying volume convergence theorem, we have that, if (\ref{R^n-convergence}) holds for sufficiently small $\epsilon_2(n)>0$,\begin{equation}
\Vol(B_{sr}(\tilde{q}))\geq\frac{1}{2}\Vol(B_{sr}(0^n)).
\end{equation}
Therefore,
\begin{equation}
\Vol(B_{1/2}(\tilde{p}))\geq \Vol(B_{sr}(\tilde{q}))\geq\frac{1}{2}\Vol(B_{sr}(0^n))\geq v_0(n,B_1(z^k))>0.
\end{equation}
\end{proof}

\subsection{Proof of Theorem \ref{non-collapsed-splitting-maximal-rank}}\label{ss:proof_noncollapsed_splitting}

Proposition \ref{non-collapsed-universal-cover} of the last subsection gives the non-collapse of universal covers by assuming the maximality of the nilpotency rank of $\Gamma_{\delta}(p)$.  Hence this non-collapse result enables us to apply
Proposition \ref{non-collapsed-splitting} to obtain a strong quantitative splitting result in which the limiting base space is arbitrary (compared with Proposition \ref{quantitative-splitting-at-origin}).  That is, we can now prove Theorem \ref{non-collapsed-splitting-maximal-rank}:

\begin{proof}[Proof of Theorem \ref{non-collapsed-splitting-maximal-rank}]

First, by Proposition \ref{non-collapsed-universal-cover}, there exists $v_0(n,B_1(z^{\ell}))>0$, $\delta(n,B_1(z^{\ell}))>0$ such that for any $q\in B_1(p)$ if $(i)$, $(ii)$ holds for the above $\delta>0$, then 
\begin{equation}
\Vol(B_{1/2}(\tilde{q}))\geq v_0(n,B_1(z^{\ell}))>0,
\end{equation}
where $\tilde{q}$ is a lift of $q$ on the universal cover of $B_1(q)$.
Therefore, by Proposition \ref{non-collapsed-splitting}, for every $\epsilon>0$, there exists $\delta'(\epsilon, n, B_1(z^{\ell}))>0$ such that if $(i)$, $(ii)$ holds for $\delta_0\equiv\min\{\delta,\delta'\}$, then for some $\delta_0<r<1$,
\begin{equation}
d_{GH}(B_r(\tilde{q}),B_r(0^{n-\ell},\hat{z}))<r\cdot\epsilon,\ (0^{n-\ell},\hat{z})\in\dR^{n-\ell}\times C(\hat{Z}),
\end{equation}
where $C(\hat{Z})$
is a metric cone over some compact space $\hat{Z}$
with  cone tip $\hat{z}$.

\end{proof}

\section{The $\epsilon$-Regularity Theorems for Lower and Bounded Ricci Curvature}\label{s:eps_reg_lower_and_bounded_ricci}

In this section we prove the $\epsilon$-regularity theorems for collapsed manifolds in the context of Ricci curvature bounded from below and bounded Ricci curvature respectively.  In essence, we will see that under the conditions of Theorem \ref{t:eps_reg_collapsed} with only a lower Ricci curvature bound, then the {\textit{ weak conjugate radius}} will be bounded uniformly from below.  From Subsection \ref{ss:harmonic-radius} to Subsection \ref{ss:proof-of-main-theorem}, we will improve upon these results in the context of bounded Ricci curvature in order to prove the main Theorem in this paper.

The $\epsilon$-regularity in the context of lower Ricci curvature will hold when our ball is close to Euclidean space or a half-space.  On the other hand, by using Theorem \ref{t:eps_reg_noncollapsed}, we will see that the $\epsilon$-regularity in the bounded Ricci context only requires that $M^n$ is close to a space of the form $\dR^{k-\ell}\times Z^{\ell}$ with $\dim(Z^{\ell})= \ell\leq 3$.  This is a sharp assumption (see Example \ref{ss:example3}).  Let us point out here that we do not assume that $Z^{\ell}$ is a cone space here, since this is an unnatural assumption in the collapsed context.  See the comments after Theorem \ref{t:eps_reg_collapsed}.

\subsection{The Statement and the Proof of the $\epsilon$-Regularity for Lower Ricci Curvature}
\label{ss:eps-reg-lower-Ricci}

In this subsection we will give a precise statement of the $\epsilon$-regularity theorem for collapsed manifolds with lower Ricci curvature bound. 
Let us begin with a useful notion of regularity.

\begin{definition}\label{def-weak-conj}
Given a Riemannian manifold $(M^n,g)$ and $\epsilon>0$ we define for each $x\in M^n$ the {\textit{weak conjugate radius (or $\epsilon$-conjugate radius)}} 
$\Conj_\epsilon(x)$ to be the supremum of $r>0$ such that for each $0<s\leq r$ there exists a mapping $\phi_s:B_s(0^n)\to B_s(\tilde x)$, where 
\begin{enumerate}
\item $B_s(0^n)\subseteq \dR^n$ and $\tilde x\in \widetilde{B_{2r}(x)}$ is a lift of $x$ to the universal cover of $B_{2r}(x)$,
\item $\phi_s(0^n)=x$.
\item $\phi_s$ is a homeomorphism onto its image.
\item $\phi_s$ is a $\epsilon s$-Gromov-Hausdorff map, i.e. $d_{GH}(B_s(\tilde{x}),B_s(0^n))<s\epsilon$.
\end{enumerate}
\end{definition}

\begin{remark}
If we assume $\Ric\geq-(n-1)$, then (3) automatically follows from (4) due to the uniform Reifenberg property (see \cite{ChC1}). In fact, $\phi_s$ is a $C^{\alpha}$-homeomorphism.
\end{remark}

Notice that the size of the weak conjugate radius contains both geometric and topological consequences.
We prove the following useful lemma on the estimate 
on the weak conjugate radius, which will be used in the proof of the $\epsilon$-regularity theorem for lower Ricci curvature bound. 

\begin{lemma}\label{weak-conj-lemma} Assume $\Ric\geq-(n-1)$, then for each $\epsilon>0$, there exists $r_0=r_0(\epsilon,n)>0$  and $\delta=\delta(\epsilon,n)>0$ such that if
\begin{equation}
d_{GH}(B_{r_0}(\tilde{x}),B_{r_0}(0^n))<r_0\delta,\label{delta-radius}
\end{equation}
then it holds that \begin{equation}
\Conj_{\epsilon}(x)\geq\frac{r_0}{2}>0.
\end{equation}

\end{lemma}

\begin{proof}It suffices to show that given $\epsilon>0$ there exists $\delta(n,\epsilon)>0$ and $r_0(n,\epsilon)>0$ such that (\ref{delta-radius}) implies that  for all $0<r\leq r_0$, 
\begin{equation}
d_{GH}(B_r(\tilde{x}),B_r(0^n))<r\epsilon.\label{epsilon-radius}
\end{equation}
In fact, if (\ref{epsilon-radius}) holds, by the uniform Reifenberg property, it holds that $B_{r_0/2}(0^n)$ is $C^{\alpha}$-homeomorphic to its image in $B_{r_0/2}(\tilde{x})$. Therefore, $\Conj_{\epsilon}(x)\geq r_0/2$.

We proceed to show the inequality (\ref{epsilon-radius}). Given $\epsilon>0$ assume that (\ref{delta-radius}) holds for some constants $\delta>0$, $r_0>0$  which depend only on $n$, $\epsilon$ and will be determined later.
By rescaling, 
\begin{equation}
d_{GH}(B_1(\tilde{x},r_0^{-2}g),B_1(0^n))<\delta.
\end{equation}
By volume convergence theorem and relative volume comparison,  for all $0<s\leq 1$,
\begin{equation}
\frac{\Vol(B_s(\tilde{x},r_0^{-2}g))}{\Vol(B_s(0^n))}\geq\frac{\Vol(B_1(\tilde{x},r_0^{-2}g))}{\Vol(B_1(0^n))}\geq 1-\Psi_1(\delta|n)\rightarrow1.\label{almost-maximal-volume}
\end{equation}
With respect to the rescaled metric $\Ric_{s^{-2}r_0^{-2}g}\geq-(n-1)s^2r_0^2$, then the almost maximal volume property (\ref{almost-maximal-volume}) implies the following almost rigidity due to \cite{ChC},
\begin{equation}
d_{GH}(B_{sr_0}(\tilde{x}_0),B_{sr_0}(0^n))<sr_0\Psi_2(\delta,r_0|n).
\end{equation}
For the fixed $\epsilon>0$, we choose sufficiently small $\delta>0$ and $r_0>0$ such that $\Psi_2(\delta,r_0|n)<\epsilon$ and the inequality (\ref{epsilon-radius}) follows. Then we finish the proof.\end{proof}

 Now let us state the main theorem of this subsection, which gives the $\epsilon$-regularity result in the context of uniform lower Ricci curvature.

\begin{theorem}\label{t:eps_reg_lower_ricci} 
Let $(M^n,g)$ be a Riemannian manifold with $\Ric \geq -(n-1)$ and such that $B_2(p)$ has a compact closure in $B_4(p)$.  There exists $w_0(n)<\infty$, and for each $\epsilon>0$ there exists $\delta=\delta(n,\epsilon)>0$, $c_0=c_0(n,\epsilon)>0$ such that if
\begin{equation}
d_{GH}(B_2(p),B_2(0^{k}))<\delta,
\end{equation}
where $B_2(0^k)\subset\dR^{k}$ or $\dR_+^k\equiv\dR^{k-1}\times\dR_+^1$, then the group $\Gamma_{\delta}(p)\equiv \Image[\pi_1(B_{\delta}(p))\rightarrow\pi_1(B_2(p))]$ is $(w_0,n-k)$-nilpotent.  In particular, $\rank(\Gamma_{\delta}(p))\leq n-k$, and if equality holds, then for every $q\in B_1(p)$ we have that
\begin{equation}
\Conj_\epsilon(q)\geq c_0(n,\epsilon)>0.
\end{equation}
\end{theorem}
\begin{remark}
We see from Theorem \ref{t:KW_almost_nilpotent} that the fibered fundamental group $\Gamma_{\delta}(p)$ in fact has {\it nilpotency length} bounded by $n-k$ with uniformly bounded index.  Our main interest then is to study the case of maximal rank.
\end{remark}

\begin{remark}
In this $\epsilon$-regularity Theorem, we assume that the limit space splits off at least $\dR^{k-1}$. This GH-control 
on the limit space is sharp because we can take a convex surface ($K\geq 0$) converging to a flat cone. 
\end{remark}

\begin{remark}
The conclusion of this $\epsilon$-regularity gives uniform control the weak conjugate radius and this is optimal in the context of lower Ricci. In fact, Example \ref{ss:small-conj} shows that in general it is impossible to expect any control on the conjugate radius even in the non-collapsed setting with smooth limit space.
\end{remark}

The proof of Theorem \ref{t:eps_reg_lower_ricci} quickly follows from Theorem \ref{non-collapsed-splitting-maximal-rank}.

\begin{proof}[Proof of Theorem \ref{t:eps_reg_lower_ricci}] First, Theorem \ref{t:KW_almost_nilpotent} gives the rank bound and the index bound $w_0(n)$. So the remaining is to prove the weak conjugate radius estimate. Observe that,
it suffices to show that for each $\epsilon>0$, there exists $\delta_0(\epsilon,n)>0$, $c_0(\epsilon,n)>0$
such that 
 if 
 \begin{equation}
 d_{GH}(B_2(p),B_2(0^k))<\delta_0,\label{lower-ricci-GH-control}
 \end{equation}
 where $B_2(0^k)\subset\dR^k$ or $\dR_+^k$, and if
 \begin{equation}
 \rank(\Gamma_{\delta_0}(p))=n-k,\label{lower-ricci-maximal-rank}
 \end{equation}
then for each $q\in B_1(p)$,
\begin{equation}
d_{GH}(B_{c_0}(\tilde{q}),B_{c_0}(0^n))<c_0 \epsilon,\ 0^n\in\dR^n,\label{weak-conj}
\end{equation}
where $\tilde{q}$ is the lifting of $q$ on the universal cover of $B_1(q)$. In fact, by Lemma \ref{weak-conj-lemma} and inequality (\ref{weak-conj}) we would then have that 

\begin{equation}\Conj_{\epsilon}(q)\geq c_0/2.\end{equation}
 
Now let us prove inequality (\ref{weak-conj}). Applying Theorem \ref{non-collapsed-splitting-maximal-rank}, there exists $v_0(n)>0$, and for each $\bar{\epsilon}>0$ there exists $\delta_0(n,\bar{\epsilon})>0$ such that if  (\ref{lower-ricci-GH-control}) and  (\ref{lower-ricci-maximal-rank}) hold for $\delta_0>0$, then $\Vol(B_{1/2}(\tilde{q}))>v_0>0$ and  for some $0<\delta_0<r<1$,
\begin{equation}
d_{GH}(B_{r}(\tilde{q}),B_{r}(0^{n-1},0^*))<r\bar{\epsilon},\label{(n-1)-dim-splitting}\end{equation}
where $0^*\in\dR^1$ or $0^*$ is the cone tip of $\dR_+^1$. By the stratification result on the singular set of \cite{ChC1}, we have that for every $\epsilon>0$, there exists $\bar{\epsilon}(n,v_0)>0$ such that if inequality (\ref{(n-1)-dim-splitting}) holds for $\bar{\epsilon}>0$, then
\begin{equation}
d_{GH}(B_{r}(\tilde{q}),B_{r}(0^{n}))<r\epsilon,\ 0^n\in\dR^n,\end{equation}
and thus the proof is complete.
\end{proof}

\subsection{Harmonic Radius and Curvature Estimates}

\label{ss:harmonic-radius}

In this subsection, we discuss the main part of Theorem \ref{t:eps_reg_collapsed}. 
The essential conclusion in the $\epsilon$-regularity result for bounded Ricci curvature is that conjugate radius is uniformly controlled from below if the assumptions of Theorem \ref{t:eps_reg_collapsed} hold.
More directly we will show that the $C^1$-harmonic radius is uniformly bounded from below on the universal
cover. Furthermore, if the manifold is Einstein or just has uniform bound on $|\nabla\Ric|$, the standard elliptic regularity theory and the harmonic radius estimate in fact gives the uniform curvature estimates.

\begin{proposition}\label{harmonic-radius}
Let $(M^n,g,p)$ be a Riemannian manifold with bounded Ricci curvature $|\Ric|\leq n-1$ and such that $B_2(p)$ has a compact closure in $B_4(p)$.  For each Ricci-limit space $(Z^{\ell},z^{\ell})$ with $\dim Z^{\ell}=\ell\leq 3$, there exists $\delta_0=\delta_0(n,B_1(z^{\ell}))>0$ and $h_0=h_0(n,B_1(z^{\ell}))>0$ such that if

$(i)$ $d_{GH}(B_2(p),B_2(0^{k-\ell},z^{\ell}))<\delta_0,\ (0^{k-\ell},z^{\ell})\in\dR^{k-\ell}\times Z^{\ell}$,

$(ii)$ $\Gamma_{\delta_0}(p)\equiv\Image[\pi_1(B_{\delta_0}(p))\rightarrow\pi_1(B_2(p))]$ satisfies $\rank(\Gamma_{\delta_0}(p))=n-k$,
\\
then for each $q\in B_1(p)$

\begin{equation}
r_h(\tilde{q})\geq h_0>0,
\end{equation}
where $\tilde{q}$ is a lift of $q$ on the universal cover of $B_1(q)$.
\end{proposition}

\begin{proof}

We assume that conditions $(i)$, $(ii)$ hold for $\delta_0(n,B_1(z^\ell))>0$, which will be determined momentarily. By Theorem \ref{non-collapsed-splitting-maximal-rank}, there is some $v_0(n,B_1(z^\ell))>0$ and for each 
$\epsilon>0$, there exists some positive constant $\delta(\epsilon,n,B_1(z^\ell))>0$ such that if $(i)$, $(ii)$ holds for $\delta>0$, then \begin{equation}\Vol(B_{1/2}(\tilde{q}))\geq v_0>0\label{v_0-non-collapse}\end{equation}
and for some $\delta<r<1$, 
\begin{equation}
d_{GH}(B_{r}(\tilde{q}),B_{r}(0^{n-\ell},\hat{z}))<r\epsilon,\ (0^{n-\ell},\hat{z})\in\dR^{n-\ell}\times \hat{Z},\ \ell\leq 3,\label{(n-3)-splitting-h}
\end{equation}
where $C(\hat{Z})$ is a metric cone over some compact 
metric space $\hat{Z}$ with cone tip $\hat{z}$. Therefore, by Theorem \ref{t:eps_reg_noncollapsed}, there exists $\epsilon_0(n,v_0)>0$, $r_h(n,v_0)>0$ such that if (\ref{v_0-non-collapse}) holds for $v_0(n,B_1(z^{\ell}))>0$ and equation (\ref{(n-3)-splitting-h}) holds for $\epsilon_0(n,v_0)>0$, then
\begin{equation}
r_h(\tilde{q})\geq r\cdot r_h(n,v_0)>\delta\cdot r_h(n,v_0).
\end{equation}
Therefore, for the above $\epsilon_0(n,v_0)>0$, correspondingly, we choose
\begin{equation}
\delta_0=\delta(\epsilon_0,n,B_1(z^\ell))>0, \ h_0=\delta_0\cdot r_h(n,v_0)>0
\end{equation}
which depend only on $n$ and the geometry of $B_1(z^\ell)$.
The proof is complete.

\end{proof}

Now we prove the curvature bound in Theorem \ref{t:eps_reg_collapsed}. Actually, we can prove it in a little more general setting, that is, we will replace the Einstein condition with the uniform bound on the covariant derivative of Ricci tensor. The arguments are rather classical once we have the harmonic radius bound in Proposition \ref{harmonic-radius}.

\begin{proposition}\label{curvature-bound} Given $n\geq 2$, $\Lambda<\infty$ and a Ricci-limit space $(Z^{\ell},z^\ell)$ with $\dim(Z^{\ell})=\ell\leq 3$, there are 
positive constants $\delta_0=\delta_0(n,B_1(z^{\ell}))>0$ and $C=C(n,\Lambda,B_1(z^{\ell}))<\infty$ such that if the constant $\delta_0>0$
satisfies assumptions $(i)$, $(ii)$ in Proposition \ref{harmonic-radius} and in addition assume  
$\sup\limits_{B_2(p)}|\nabla\Ric|\leq\Lambda$, then \begin{equation}\sup_{B_1(p)}|\Rm|\leq C(n,\Lambda,B_1(z^{\ell})).\end{equation} Additionally, if $(M^n,g,p)$ is Einstein $\Ric\equiv\lambda g$ with $|\lambda|\leq n-1$, then for each $j\in \mathbb{N}$, there exists $0<C_j(n,B_1(z^\ell))<\infty$ such that
\begin{equation}\sup\limits_{B_1(p)}|\nabla^j\Rm|\leq C_j(n, B_1(z^\ell)).\end{equation}
\end{proposition}

\begin{proof} The proof is the standard Schauder estimate and we only prove $(i)$.
It  suffices to show that for each $q\in B_1(p)$, $|\Rm|(\tilde{q})\leq C(n,\Lambda,B_1(z^{\ell}))$, where $\tilde{q}$
is a lift of $q$ on the universal cover of $B_1(q)$. 
By Proposition \ref{harmonic-radius}, 
$r_h(\tilde{q})\geq h_0(n,Z^{\ell})>0.
$ Hence, in $B_{h_0/2}(\tilde{q})$, we can express Ricci tensor in terms of the harmonic coordinates $\{x_1,\ldots, x_n\}$, 
\begin{equation}
\R_{ij}=g^{kl}\frac{\partial^2 g_{ij}}{\partial x^k\partial x^l}+Q\Big(\frac{\partial g_{rs}}{\partial x_m}\Big),
\end{equation}
where $\R_{ij}$ are components of Ricci tensor and $Q$ is a quadratic form in terms of the first derivative of $g_{ij}$.
Note that $g_{ij}$ has uniform $C^{1,\alpha}$-norm in $B_{h_0/2}(\tilde{q})$.
Since $|\nabla \Ric|\leq\Lambda$, by the classical Schauder estimate, it follows that \begin{equation}h_0^{2}|\nabla^2g_{ij}|_{C^{0}(B_{h_0/4}(\tilde{q}))}+h_0^{2+\alpha}|\nabla^2g_{ij}|_{C^{\alpha}(B_{h_0/4}(\tilde{q}))}\leq C(n,\Lambda).\end{equation}
The above estimate actually gives the curvature bound,\begin{equation}\sup\limits_{B_{h_0/4}(\tilde{q})} |\Rm|\leq C(n,\Lambda)h_0^{-2}.\end{equation} Descending to the base space, we obtain the desired curvature bound.
\end{proof}

\subsection{Conjugate Radius Bound Implies Maximal Nilpotency Rank}
\label{ss:proof-maximal-rank}

The converse direction of Theorem \ref{t:eps_reg_collapsed} is an immediate consequence of the following 
fibration construction that we will discuss. Actually, a similar version of Proposition \ref{fiber-bundle} was stated in \cite{Dai-Wei-Ye} without proof. Although this result is known for experts, we will give a proof in Appendix \ref{s:proof-of-fiber-bundle} for completeness sake. 
Our proof of this fiber bundle result essentially relies on the pointwise $C^1$ and $C^2$ estimates of harmonic functions under a lower conjugate radius bound, which gives an explicit and intrinsic construction of the fiber bundle map (compared with Fukaya and Yamaguchi's embedding method via distance function, for instance see \cite{yamaguchi}). In fact, this analytical method can be applied in a broad setting.
Finally, we will apply the smoothing techniques in  \cite{Dai-Wei-Ye} to topologically identify the fiber.

\begin{proposition} \label{fiber-bundle}
Let $(M^n, g, p)$ be a Riemannian manifold with $|\Ric|\leq n-1$ and such that $B_3(p)$ has a compact closure in $B_6(p)$.
Given $c_0>0$, there exists $\delta=\delta(n,c_0)>0$ such that if

$(i)$ $d_{GH}(B_3(p), B_3(0^k))<\delta,\ 0^k\in\dR^k$, $k\leq n$,

$(ii)$ for each $q\in B_3(p)$, we have the conjugate radius lower bound $\Conj(q)\geq c_0>0$,
\\
then there is a smooth map $\Phi:B_{5/2}(p)\rightarrow B_{5/2}(0^k)$ such that 
$\Phi^{-1}(B_{9/4}(0^k))$ is diffeomorphic to $B_{9/4}(0^k)\times N^{n-k}$, where $N^{n-k}$ is an
infra-nilmanifold of dimension $n-k$. 
\end{proposition}

We will prove the above in Appendix \ref{s:proof-of-fiber-bundle}.  The converse direction of Theorem \ref{t:eps_reg_collapsed} now follows from the Proposition below.

\begin{proposition}\label{converse-rank-maximum} Let $(M^n, g, p)$ be a Riemannian manifold with $|\Ric|\leq n-1$ such that $B_3(p)$ has a compact closure in $B_6(p)$.
There exists $w(n)>0$ and given $c_0>0$  there exists $\delta(n,c_0)>0$ such that if

$(i)$ $d_{GH}(B_3(p), B_3(0^k))<\delta,\ 0^k\in\dR^k,\ k\leq n$,

$(ii)$ for each $q\in B_3(p)$, we have the conjugate radius lower bound $\Conj(q)\geq c_0>0$,
\\
then 
\begin{equation}\Gamma_{\delta}(p)\equiv\Image[\pi_1(B_{\delta}(p))\rightarrow \pi_1(B_2(p))]\end{equation}
contains a nilpotent subgroup of rank $n-k$ of index at most $w(n)<\infty$.

\end{proposition}

\begin{proof}
If all of the assumptions hold, by Proposition \ref{fiber-bundle}, there exists a smooth map \begin{equation}\Phi:B_{5/2}(p)\rightarrow B_{5/2}(0^k),\ 0^k\subset\dR^k,\end{equation} such that the pre-image 
$\Phi^{-1}(B_{9/4}(0^k))$ is homeomorphic to the trivial bundle $B_{9/4}(0^k)\times F^{n-k}$, where $F^{n-k}$ is an infranilmanifold of dimension $n-k$.  Hence there exists some $w(n)<\infty$ and a $w(n)$-finite normal cover of $F^{n-k}$, denoted by $\hat{F}^{n-k}$, such that $\hat{F}^{n-k}$ is a compact nilmanifold. Then $\hat{F}^{n-k}$ is diffeomorphic to $N^{n-k}/\mathcal{N}$ for some simply-connected nilpotent Lie group $N^{n-k}$ (and thus diffeomorphic to $\dR^{n-k}$) and
 for some co-compact lattice $\mathcal{N}<N^{n-k}$. It is a standard fact that $\rank(\mathcal{N})=\dim(N^{n-k})=n-k$ (see \cite{rag} for instance).
Since $\hat{F}^{n-k}$ is a bounded normal cover of $F^{n-k}$ by $w(n)$, it follows that $[\pi_1(F^{n-k}):\mathcal{N}]\leq w(n)<\infty$ and thus $\rank(\pi_1(F^{n-k}))=\rank(\mathcal{N})=n-k$.

Now we  claim that the   group \begin{equation}\Gamma_{\delta}(p)\equiv\Image[\pi_1(B_{\delta}(p))\rightarrow\pi_1(B_2(p))]\end{equation} is isomorphic to 
$\pi_1(F^{n-k})$. Take a fiber $F^{n-k}\subset B_{\delta}(p)$ at $p$. Notice that, each element in $\pi_1(F^{n-k})$ and $\Gamma_{\delta}(p)$  can be generated by a closed geodesic of length at most $2\delta$. In the following arguments, we give the injective homomorphisms
\begin{equation}
\varphi:\pi_1(F^{n-k})\rightarrow\Gamma_{\delta}(p),\ \psi:\Gamma_{\delta}(p)\rightarrow \pi_1(F^{n-k})
\end{equation}
 such that \begin{equation}\psi\circ\varphi=\id_{\pi_1(F^{n-k})}, \ \varphi\circ\psi=\id_{\Gamma_{\delta}(p)},\label{composition-identity}\end{equation} which particularly implies that $\pi_1(F_1)\cong\Gamma_{\delta}(p)$.

Observe that, by Proposition \ref{fiber-bundle}, there exists an open set such that  $B_2(p)\subset U$ 
 and  $U$ is diffeomorphic to $D^k\times F^{n-k}$. Denote by $\varphi$ the homomorphism induced by $F^{n-k}\hookrightarrow U$, and then $U\cong D^k\times F^{n-k}$ implies that $\pi_1(F^{n-k})\cong\varphi(\pi_1(F^{n-k}))\cong\pi_1(U)$. Since every generator $\gamma\in\pi_1(F^{n-k})$ is chosen such that $\ell(\gamma)< 2\delta$, $\varphi(\gamma)$ canonically determines an element in $\Gamma_{\delta}(p)$. That is, $\varphi$ gives a homomorphism $\pi_1(F^{n-k})\rightarrow \Gamma_{\delta}(p)$.
 The next is to check the injectivity of 
$\varphi$. In fact, for a trivial loop $\gamma\in\Gamma_{\delta}(p)$, $\gamma$ is trivial in $\pi_1(U)$, and thus $\gamma$ is trivial in $\pi_1(F^{n-k})$. 
Hence $\varphi$ is an injective homomorphism from $\pi_1(F)$ to $\Gamma_{\delta}(p)$. 

On the other hand, 
let $\psi$ be the  natural composed homomorphism $\Gamma_{\delta}(p)\rightarrow\pi_1 (U)\xrightarrow{\cong}\pi_1(F^k)$. For any trivial loop $\gamma\in\pi_1(F^{n-k})$, then $\gamma$ is trivial in $\pi_1(U)$ and thus trivial in $\Gamma_{\delta}(p)$. Therefore, $\psi$ is injective.
Moreover, it is straightforward to prove equation (\ref{composition-identity}). We have proved the claim.

Therefore,  
$\Gamma_{\delta}(p)$ contains a nilpotent subgroup $\mathcal{N}$ of rank $n-k$ and of index at most $w(n)<\infty$.

\end{proof}

\subsection{The Proof of the Main $\epsilon$-Regularity Theorem}

\label{ss:proof-of-main-theorem}

We complete this section by proving Theorem \ref{t:eps_reg_collapsed} and also giving an example to show the optimality of the converse direction in Theorem \ref{t:eps_reg_collapsed}.

\begin{proof}[Proof of Theorem \ref{t:eps_reg_collapsed}]

The proof of the main theorem is now just a combination of Theorem \ref{t:KW_almost_nilpotent}, Proposition  \ref{harmonic-radius}, and Proposition \ref{converse-rank-maximum}.  Indeed, let us assume that for some $\delta>0$ to be determined momentarily we have
\begin{align}
d_{GH}\Big(B_2(p), B_2((0^{k-\ell},z^\ell))\Big)<\delta,\label{GH-control-(k-l)}
\end{align}
where $(0^{k-\ell},z^\ell)\in \dR^{k-\ell}\times Z^\ell$ with $\dim(Z^{\ell})=\ell\leq 3$.  First,
by Theorem \ref{t:KW_almost_nilpotent}, it holds that $\rank(\Gamma_{\delta}(p))\leq n-k$ for sufficiently small $0<\delta\leq \delta_0(n,B_1(z^{\ell}))$. Furthermore, the assumption $\rank(\Gamma_{\delta}(p))=n-k$ and the Gromov-Hausdorff control (\ref{GH-control-(k-l)})
 enable us to apply Proposition \ref{harmonic-radius}. That is, we may choose $\delta(n,B_1(z^{\ell}))$) such that if the above assumptions hold, then we have on the universal cover the estimate for every $\tilde{q}\in B_1(\tilde{p})$,
\begin{align}
r_h(\tilde{q})\geq h_0(n,B_1(z^{\ell})).
\end{align}
It is a standard point, indeed just work in coordinates, to check that this implies the desired conjugate radius lower bound.  Additionally, if the manifold is assumed to be Einstein, then curvature bound follows from  Proposition \ref{curvature-bound}.  The converse conclusion is contained in Proposition \ref{converse-rank-maximum}.  Thus we have finished the proof.

\end{proof}

In the rest of this subsection, we give a simple counterexample to show that the assumption in the converse direction of Theorem \ref{t:eps_reg_collapsed} is sharp. More precisely,
the converse direction of Theorem \ref{t:eps_reg_collapsed} is not true in general if the limit space is not euclidean or not a smooth manifold. That is, without the Gromov-Hausdorff control (\ref{smooth-limit}), bounded curvature does not imply that the nilpotency rank of
$\Gamma_{\delta}(p)\equiv\Image[\pi_1(B_{\delta}(p))\rightarrow\pi_1(B_2(p))]$ attains the maximum.

\begin{example}\label{half-space-counter-example}Let $M^3=\dR^2\times S^1$. We define an isometric action of $\mathbb{Z}_k=\{e^{{\sqrt{-1}}\frac{2j\pi}{k}}|0\leq j\leq k-1\}$
as follows: for any $(re^{\sqrt{-1}\theta},e^{\sqrt{-1}\varphi})\in M^3$,
 the rotation $\gamma_{j}=e^{\sqrt{-1}\frac{2j\pi}{k}}$ is given by
\begin{equation}\gamma_{\ell}\cdot\Big(re^{\sqrt{-1}\theta},e^{\sqrt{-1}\varphi}\Big)
\equiv\Big(re^{\sqrt{-1}(\theta+\frac{2j\pi}{k})},e^{\sqrt{-1}(\varphi+\frac{2j\pi}{k})}\Big).\end{equation}
Under the above assumption, $\mathbb{Z}_k$ properly discontinuously and freely acts on $M^3$,
then the quotient space $N_k^3\equiv M^3/\mathbb{Z}_k$ is a flat manifold with the quotient metric.
Letting $k\rightarrow\infty$, then the limit space is a half line, i.e.
\begin{equation}(N_k^3,g_k,p_k)\xrightarrow{GH}(\dR^1_+,g_0,0),
\end{equation}
and the convergence keeps sectional curvature bounded. Notice that $M^3$ is a normal $\mathbb{Z}_k$-cover 
of $N_k^3$, i.e. $\pi_1(N_k^3)/\pi_1(M^3)\cong\mathbb{Z}_k$ and $\pi_1(M^3)\cong\mathbb{Z}$. On the other 
hand, the inclusion homomorphism
$i_*:\pi_1(B_1(p_k))\rightarrow\pi_1(N_k^3)$ is surjective.
Therefore, \begin{equation}\rank\Big(\Image[\pi_1(B_1(p_k))\rightarrow\pi_1(N_k^3)]\Big)=\rank(\pi_1(N_k^3))=1.\end{equation}


\end{example}

\appendix

\section{Generalized Margulis Lemma for Collapsed Manifolds}

\label{s:proof-of-margulis}

In this section, we give a proof of Theorem \ref{t:KW_almost_nilpotent}, 
which extends theorem \ref{Generalized-Margulis-Lemma} in 
Kapovitch-Wilking's joint paper \cite{KW}. Although they did not explicitly state it in this way, in fact their essential techniques lead to the refinement stated in Theorem \ref{t:KW_almost_nilpotent}.
Here we will show that the nilpotency length of the fibered fundamental group is controlled from above by the the collapsed dimension. The point here is that although in general the limit space is singular, we have the concept of limiting dimension in the sense of theorem \ref{limiting-dimension}.  We recall the statement of Theorem \ref{t:KW_almost_nilpotent}.

\vspace{0.5cm}

{\textit{ Let $(Z^k,z^k)$ be a pointed Ricci-limit metric space with $\dim Z^k = k$ in the sense of theorem \ref{limiting-dimension}.  Then there exists 
$\epsilon_0=\epsilon_0(n,B_1(z^k))>0$, $w_0=w_0(n,B_1(z^k))<\infty$ such that if a Riemannian manifold $(M^n, g,p)$ with 
$\Ric\geq-(n-1)$ satisfies that $B_2(p)$ has a compact closure in $B_4(p)$ and
 \begin{equation}d_{GH}(B_2(p),B_2(z^k))<\epsilon_0,
\end{equation}
 then the fibered fundamental group
$\Gamma_{\epsilon_0}(p)\equiv \Image[\pi_1(B_{\epsilon_0}(p))\rightarrow\pi_1(B_2(p))]\label{image-group}$ is $(w_0,n-k)$-nilpotent.
}}

\vspace{0.5cm}

\begin{proof}[Proof of Theorem \ref{t:KW_almost_nilpotent}]  Let us first outline the proof. The first step is to construct a nilpotent subgroup in $\Gamma_{\epsilon}(p)\cap\Gamma_{\delta}(B_1(\tilde{p}))$ for some positive constants $\epsilon>0$, $\delta>0$ of definite amount, where 
\begin{equation}
\Gamma_{\delta}(B_{1}(\tilde{p}))
\equiv\Big\langle\Big\{\gamma\in\pi_1(B_2(p))\Big|d(\gamma\cdot\tilde{q},\tilde{q})<\delta,\ \forall\tilde{q}\in B_{1}(\tilde{p})\Big\}\Big\rangle,
\end{equation}
and $\tilde{p}$ is a lift of $p$ on the universal cover of $B_2(p)$. In the second step, we will prove a claim which shows that 
$\Gamma_{\delta}(B_1(\tilde{p}))$ has a nilpotent subgroup of length $\leq n-k$ and of controlled index.  Finally, we will see that $\Gamma_{\epsilon}(p)\cap\Gamma_{\delta}(B_1(\tilde{p}))$ has controlled index in $\Gamma_{\epsilon}(p)$, from which we will have finished the proof.

As the first stage, we construct a nilpotent subgroup $\mathcal{N}_a\leq \Gamma_{\epsilon}(p)\cap\Gamma_{\delta}(B_1(\tilde{p}))$ such that $[\Gamma_{\epsilon}(p):\mathcal{N}_a]<C(n)$.
By theorem 2.5 in \cite{KW}, there exists $D_1(n)<\infty$ such that for any given $\epsilon>0$, there is some $d_1\leq D_1(n)$ such that $\Gamma_{\epsilon}(p)\equiv\Image[\pi_1(B_{\epsilon}(p))\rightarrow\pi_1(B_2(p))]$ has a generating set $\{\gamma_1,\ldots,\gamma_{d_1}\}$ with $d(\gamma_j\cdot\tilde{p},\tilde{p})<2\epsilon$ for every $1\leq j\leq d_1$. The positive constant $\epsilon>0$ will be determined momentarily.
On other hand, by theorem \ref{Generalized-Margulis-Lemma},  there exists $\epsilon_1(n)>0$ such that for every $0<\epsilon\leq\epsilon_1(n)$ there is a
nilpotent subgroup $\mathcal{N}$ in $\Gamma_{\epsilon}(p)$
such that
\begin{equation}
[\Gamma_{\epsilon}(p):\mathcal{N}]\leq C_1(n),\ \Step(\mathcal{N})=c\leq\length(\mathcal{N})\leq n.
\end{equation}
Lemma \ref{effective-generating-set} implies that $\mathcal{N}$ has a generating set
\begin{equation}
B=\langle g_1^{(0)},\ldots, g_{d_2}^{(0)}\rangle, \ d_2\leq C_1^2(n)\cdot D_1(n),
\end{equation}
and $d(g_j^{(0)}\cdot\tilde{p},\tilde{p})<2(2C_1+1)\cdot\epsilon$. Now we are in a position to explicitly write the nilpotent subgroup $\mathcal{N}_a\leq\mathcal{N}$. To this end, denote 
\begin{equation}
\bar{B}\equiv\bigcup\limits_{k=1}^{c}\mathcal{C}_k(B),
\end{equation}
then straightforward calculations show that
\begin{equation}
\#(\bar{B})\leq 10\cdot (d_2)^{c}\leq 10\cdot( C_1^2(n)\cdot D_1(n))^n\equiv C_3(n).
\end{equation}
Let \begin{equation}\bar{B}\equiv\{\sigma_1,\ldots,\sigma_{d_3}\},\ d_3\leq C_3(n),\end{equation}
and
\begin{equation}
d(\sigma_j\cdot\tilde{p},\tilde{p})\leq C_4(n)\cdot\epsilon,
\end{equation}
where $C_4\equiv 2\cdot(2C_1+1)\cdot(3\cdot 2^{n}-2)$.
We choose \begin{equation}\epsilon<\frac{\Psi_0(2\delta|n)}{4C_4}\label{nilpotent-group-transit}\end{equation} for some 
$\delta(n,B_1(z^k))>0$ determined later, where $\Psi_0$ is the function in Lemma \ref{every-point-Z}. By Lemma \ref{every-point-Z}, there exists a sequence of integers $1\leq a_j\leq N_1(n)$ such that if we choose 
\begin{equation}
\mathcal{N}_a\equiv\langle\sigma_1^{a_1},\ldots,\sigma_{d_3}^{a_{d_3}}\label{def-N_a}\rangle,
\end{equation}
then for each $1\leq j\leq d_3$ and for all $\tilde{x}\in B_1(\tilde{p})$,
\begin{equation}
d(\sigma_j^{a_j}\cdot\tilde{x},\tilde{x})<2\delta,\label{N_a-small-displacement}
\end{equation}
where $\tilde{p}$ is a lift of $p$ on the universal cover of $B_2(p)$.
In particular, inequality (\ref{N_a-small-displacement}) gives that
\begin{equation}
\mathcal{N}_a\leq\Gamma_{\delta}(B_1(\tilde{p})).
\end{equation}
By the construction (\ref{def-N_a}) and Lemma \ref{power-controlled-index},
\begin{equation}
[\mathcal{N}:\mathcal{N}_a]\leq N_1(n)^{C_3(n)}\equiv C_5(n),\ [\Gamma_{\epsilon}(p):\mathcal{N}_a]\leq C_1(n)\cdot C_5(n).\label{N_a-controlled-index}
\end{equation}

The next is to prove the following Claim.

\vspace{0.5cm}

{\bf{Claim.}} 
There exists $\delta=\delta(n,B_1(z^k))>0$, $C_0=C_0(n,B_1(z^k))<\infty$ such that if
\begin{equation}
d_{GH}(B_2(p),B_2(z^k))<\delta, \ z^k\in Z^k,\label{Margulis-GH-control}
\end{equation}
then  
\begin{equation}
\Gamma_{\delta}(B_{1}(\tilde{p}))
\equiv\Big\langle\Big\{\gamma\in\pi_1(B_2(p))\Big|d(\gamma\cdot\tilde{q},\tilde{q})<\delta,\ \forall\tilde{q}\in B_{1}(\tilde{p})\Big\}\Big\rangle
\end{equation}
contains a nilpotent subgroup $\widehat{\mathcal{N}}$ of length $\leq n-k$ and of index $\leq C_0(n,B_1(z^k))$.
In particular, we have the nilpotency rank bound 
$\rank(\Gamma_{\delta}(B_1(\tilde{p})))\leq n-k$.

\vspace{0.5cm}

The proof is an application of the induction theorem of \cite{KW}, we will include the details for completeness sake.  We argue by contradiction and suppose there is a pointed Ricci-limit space $(Z^k,z^k)$
with $\dim Z^k=k$ such that no such $\delta(n,B_1(z^k))>0$ and $C_0(n,B_1(z^k))<\infty$ exist. That is,
there is a sequence $\delta_i\rightarrow0$ and a  sequence of Riemannian manifolds $(M_i^n,g_i,p_i)$
with $\Ric_{g_i}\geq-(n-1)$ which satisfies $B_2(p_i)$ has a compact closure in $B_4(p_i)$ and
\begin{equation}
d_{GH}(B_{2}(p_i), B_{2}(z^k))<\delta_i.
\end{equation}
But each nilpotent subgroup in $\Gamma_{\delta_i}(B_{1}(\tilde{p}_i))$ of length $\leq n-k$ has index $\geq 2^i$, where $\tilde{p}_i$ is a lift on the universal cover of $B_2(p_i)$. 
We will produce a contradiction by considering the intermediate covering space of $B_2(p_i)$,
\begin{equation}\pi_i:(N_i,\hat{p}_i)\equiv\Big(\widetilde{B_{2}(p_i)}\Big/\Gamma_{\delta_i}(B_{1}(\tilde{p}_i)),\hat{p}_i\Big)\rightarrow \Big(B_2(p_i),p_i\Big).\end{equation}
Notice that $\pi_1(N_i,\hat{p}_i)=\Gamma_{\delta_i}(B_{1}(\tilde{p}_i)).$
Passing to a subsequence,  it holds that,  for some complete length space $(\hat{Z},\hat{z})$,\begin{equation}
d_{GH}\Big(B_{\epsilon_i^{-1}}(\hat{p}_i),B_{\epsilon_i^{-1}}(\hat{z})\Big)<\epsilon_i\rightarrow0, \ \hat{z}_{\infty}\in\widehat{Z},
\end{equation}
where $\lim\limits_{\delta_i\rightarrow0}\epsilon_i=0$. Moreover, the covering maps $\pi_i:(N_i,\hat{p}_i)\rightarrow(B_2(p_i),p_i)$ converges to a submetry $\pi_{\infty}:(\hat{Z},\hat{z})\rightarrow (Z^k,z^k)$.
Take a sequence $\hat{x}_i\in B_{1/4}(\hat{p}_i)$ with $\hat{x}_i\rightarrow\hat{x}_{\infty}$ such that $\hat{x}_{\infty}\in B_{1/4}(\hat{z})$ is a regular point in $\widehat{Z}$ in the sense of theorem \ref{limiting-dimension}.
 Rescale the metrics at $\hat{x}_i$ by letting $h_i=\lambda_i^2g_i$ with $\lambda_i\rightarrow\infty$ and $\lambda_i\leq
\min\{\epsilon_i^{-1/2}, \delta_i^{-1/2}\}$,
then 
\begin{equation}
(\lambda_iN_i,\hat{x}_i)\xrightarrow{pGH}(\dR^{k_1},\hat{x}_{\infty}),\ k_1\equiv\dim \widehat{Z},\end{equation}
and $\pi_1(\lambda_i N_i,\hat{x}_i)\cong\Gamma_{\delta}(B_{1}(\tilde{p}_i))$ is generated by loops of length $\leq\delta_i^{1/2}<1$. Observe that, 
\begin{equation}k_1\equiv\dim\widehat{Z}\geq \dim Z^k=k.\end{equation}
In fact, 
take any regular point $z_{\infty}\in B_1(z^k)\subset Z^k$ with $z_i\in B_1(p_i)$ converging to $z_{\infty}$. By definition, the tangent cone at $z_{\infty}$ is $\dR^k$ provided $\dim Z=k$. For any fixed $\bar{\epsilon}>0$, there exists $r=r(\bar{\epsilon},B_1(z^k))>0$ such that for sufficiently large $i$, we have 
an $\dR^k$-splitting harmonic map $\Phi_i:B_r(z_i)\rightarrow B_r(0^k)$
which is also an $r\bar{\epsilon}$-GH approximation between $B_r(p_i)$ and $B_r(0^k)$.
The standard rescaling arguments combined with the average estimate in Lemma \ref{lift} show that the
lifted harmonic map $\widehat{\Phi}_i$ with $\widehat{\Phi}_i\equiv\Phi_i\circ\pi_i$ is also an $\dR^k$-splitting map. Therefore, we can pass the quantitative $\dR^k$-splitting property to the limit space $(\widehat{Z},\hat{z})$, which gives $\dim\widehat{Z}\geq k$.

Therefore, by the induction theorem (theorem 6.1) in \cite{KW}, there is a positive constant $C_0(n)<\infty$ such that for sufficiently large $i$ (depending on the dimension $n$ and the geometry of $B_1(z^k)\subset Z^k$), $\Gamma_{\delta_i}(B_{1}(\tilde{p}_i))$ 
has a nilpotent subgroup of length $\leq n-k$
and of controlled index (independent of $i$).  This is a contradiction, and we have proved the claim.

\vspace{0.5cm}

Now if we choose 
$\delta(n,B_1(z^k))>0$ as the one in the above Claim, $\epsilon_0>0$ be a constant satisfying (\ref{nilpotent-group-transit}), and
take $\mathcal{N}_0\equiv\mathcal{N}_a\cap\widehat{\mathcal{N}}$, then 
$\length(\mathcal{N}_0)\leq \length(\mathcal{\widehat{N}})\leq n-k$ and 
\begin{equation}
[\mathcal{N}_a:\mathcal{N}_0]=[\mathcal{N}_a:\mathcal{N}_a\cap\widehat{\mathcal{N}}]\leq [\Gamma_{\delta}(B_1(\tilde{p})):\mathcal{\widehat{N}}]\leq C_0(n, B_1(z^k)).
\end{equation}
Therefore, by inequality (\ref{N_a-controlled-index}),
\begin{eqnarray}
[\Gamma_{\epsilon}(p):\mathcal{N}_0]&\leq& [\Gamma_{\epsilon}(p):\mathcal{N}_a]\cdot [\mathcal{N}_a:\mathcal{N}_0]\nonumber\\
&\leq& C_1(n)\cdot C_5(n)\cdot C_0(n,B_1(z^k))\equiv w_0(n, B_1(z^k)).
\end{eqnarray}
The proof of the Theorem is complete.
\end{proof}

\section{Harmonic Fiber Bundle Map for Bounded Ricci Curvature}

\label{s:proof-of-fiber-bundle}

This section is devoted to give the proof of Proposition  \ref{fiber-bundle}. For the convenience, let us restate the Proposition here.

\vspace{0.5cm}

{\textit{Let $(M^n, g, p)$ be a Riemannian manifold with $|\Ric|\leq n-1$ and $B_3(p)$ has a compact closure in $B_6(p)$.
Given $c_0>0$, there exists $\delta=\delta(n,c_0)>0$ such that if}

\textit{$(i)$ $d_{GH}(B_3(p), B_3(0^k))<\delta,\ 0^k\in\dR^k$, $k\leq n$,}

\textit{$(ii)$ for each $q\in B_3(p)$, we have the conjugate radius lower bound $\Conj(q)\geq c_0>0$,}
\\
\textit{
then there is a smooth map $\Phi:B_{5/2}(p)\rightarrow B_{5/2}(0^k)$ such that 
$\Phi^{-1}(B_{9/4}(0^k))$ is diffeomorphic to $B_{9/4}(0^k)\times N^{n-k}$, where $N^{n-k}$ is an
infra-nilmanifold of dimension $n-k$. 
}}
\vspace{0.5cm}

\begin{proof}[Proof of Proposition \ref{fiber-bundle}]

The proof consists of the following steps.

{\bf{Step 1.}} There exists $0<\delta_1(n,c_0)<1$ such that if $(i)$ holds for $0<\delta\leq\delta_1<1$, then for each $q\in B_{5/2}(p)$, there is a local fiber bundle, which is a $\Psi(\delta|n)$-Gromov-Hausdorff map.

\vspace{0.5cm}

Assume that \begin{equation}d_{GH}(B_3(p), B_3(0^k))<\delta,\ 0^k\in\dR^k,\ k\leq n,
\end{equation}
for some $\delta>0$ which will be determined later, and let 
\begin{equation}
f:B_3(p)\longrightarrow B_3(0^k),\ f(p)=0^k\in\dR^k
\end{equation}
be a $\frac{\delta}{3}$-GH approximation.
Immediately, for each $q\in B_{5/2}(p)$, \begin{equation}
d_{GH}(B_{1/6}(q),B_{1/6}(f(q)))<\delta.
\end{equation}
Let $h\equiv \delta^{-1}g$, then for the rescaled metric, it holds that $\Ric\geq-(n-1)\delta$ and
\begin{equation}
d_{GH}\Big(B_{\frac{1}{6}\cdot\delta^{-1/2}}(q,h),B_{\frac{1}{6}\cdot\delta^{-1/2}}(f(q))\Big)<\delta^{1/2}.
\end{equation}
The above inequalities implies that there exists a smooth map 
\begin{equation}
\Phi_{q,h}=(\Phi_{q,h}^{(1)},\ldots, \Phi_{q,h}^{(k)}):B_1(q,h)\rightarrow B_1(f(q)),\label{base-harmonic-map}
\end{equation}
where for each $1\leq \mu\leq k$, $\triangle\Phi_{q,h}^{(\mu)}=0$, and
they satisfy the following estimate (under the rescaled metric)
\begin{equation} \fint_{B_{1}(q,h)}\sum\limits_{\mu,\nu=1}^k|\langle \nabla \Phi_{q,h}^{(\mu)}, \nabla \Phi_{q,h}^{(\nu)}\rangle_h-\delta_{\mu\nu}|d\Vol_h+
\fint_{B_1(q,h)}\sum\limits_{\mu=1}^k|\nabla^2 \Phi_{q,h}^{(\mu)}|_h^2d\Vol_h<\Psi(\delta|n).\label{average-C1-C2}
\end{equation}
Moreover, $\Phi_{q,h}$ is a $\Psi(\delta|n)$-Gromov-Hausdorff map.
Note that  for each $q$, the conjugate radius is uniformly bounded from below, that is, 
\begin{equation}\Conj_h(q)\geq c_0\cdot\delta^{-1}>10^3,
\end{equation}
where we just choose \begin{equation}\delta<\frac{c_0}{10^3}.\label{delta-conjugate}\end{equation}
 Therefore, the exponential map $\exp_q: (B_{1}(0^n),\hat{h})\rightarrow (B_{1}(q),h)$ is non-degenerate with $\hat{h}\equiv\exp_q^*(h)$, and thus we have the injectivity radius estimate at $0^n\in B_1(0^n)$, i.e.
\begin{equation}
\InjRad_h(0^n)\geq 1>0,
\end{equation}
then the $C^1$-harmonic radius bound at $0^n$ has definite lower bound provided  $|\Ric_{\hat{h}}|\leq (n-1)\delta$,
\begin{equation}
r_h(0^n)\geq s_0(n)>0,\label{s0-harmonic-radius-bound}
\end{equation}
and by Schauder estimate the $C^{1,\alpha}$-norm of the rescaled metric $\hat{h}$ is controlled by
\begin{equation} |\nabla(\hat{h}_{ij}-\delta_{ij})|_{C^{0}(B_{2s_0/3}(0^n))}+
[\nabla(\hat{h}_{ij}-\delta_{ij})]_{C^{\alpha}(B_{2s_0/3}(0^n))}<C(n).
\end{equation}
By our construction, the pull-back function $\widehat{\Phi}_{q,h}^{(\mu)}\equiv\Phi\circ\exp_q$ is also harmonic for each $1\leq\mu\leq k$, that is, 
\begin{equation}
\triangle_{\hat{h}} \widehat{\Phi}_{q,h}^{(\mu)}(\hat{x})\equiv 0,\  \forall \hat{x}\in B_{2s_0/3}(0^n,\hat{h}),\ \forall \mu=1,\ldots, k.
\end{equation}
Since the coefficients of the above equation have uniform $C^{\alpha}$-norm, the Schauder interior estimate gives the $C^{2,\alpha}$-bound,
\begin{equation}
|\nabla^2\widehat{\Phi}_{q,h}^{(\mu)}|_{C^{0}(B_{s_0/2}(0^n))}+[\nabla^2\widehat{\Phi}_{q,h}^{(\mu)}]_{C^{\mu}(B_{s_0/2}(0^n))}<C(n),\ \mu=1,\ldots, k.
\end{equation}
Descending to the base space, it holds that
\begin{equation}
|\nabla^2\Phi_{q,h}^{(\mu)}|_{C^{0}(B_{s_0/2}(q))}+[\nabla^2\Phi_{q,h}^{(\mu)}]_{C^{\alpha}(B_{s_0/2}(q))}<C(n),\ \mu=1,\ldots, k.\label{pointwise-hess}
\end{equation}

On the other hand, relative volume comparison theorem and the average estimate (\ref{average-C1-C2}) implies the following estimate on $B_{s_0/2}(q)$, 
\begin{equation}
\fint_{B_{s_0/2}(q)}\sum\limits_{\mu,\nu=1}^k|\langle \nabla \Phi_{q,h}^{(\mu)}, \nabla\Phi_{q,h}^{(\nu)}\rangle-\delta_{\mu\nu}|d\Vol_h<\Psi(\delta|n),\label{average-C1}
\end{equation}
Observe that, given any $\bar{\epsilon}>0$, there exists \begin{equation}
\delta=\delta(\bar{\epsilon},n)>0\label{determine-delta-2}
\end{equation} such that if the integral $C^1$-estimate (\ref{average-C1}) holds for the above $\delta>0$, then by the uniform $C^{2,\alpha}$-estimate (\ref{pointwise-hess}) and simple compactness arguments, the following pointwise $C^1$-estimate holds,
\begin{equation}
\sum\limits_{\mu,\nu=1}^k|\langle \nabla \Phi_{q,h}^{(\mu)}, \nabla\Phi_{q,h}^{(\nu)}\rangle-\delta_{\mu\nu}|(x)<\bar{\epsilon}, \ x\in B_{s_0/2}(q).\label{pointwise-C1}
\end{equation}
So we can choose $\bar{\epsilon}=\bar{\epsilon}(n)>0$ such that the harmonic map $\Phi_{q,h}$ is non-degenerate.
Therefore, if $(i)$ holds for the constant $\delta_1=\delta_1(n,c_0,\bar{\epsilon}(n))>0$ which is determined by inequality (\ref{delta-conjugate}) and (\ref{determine-delta-2}), the harmonic map $\Phi_{q,h}$ is non-degenerate and it gives a fiber bundle with fiber of dimension-$(n-k)$.

\vspace{0.5cm}

{\bf{Step 2.}} For the original Riemannian metric $g$, there exists $0<\delta_2(n,c_0)<1$ such that if $(i)$ and $(ii)$ hold for $0<\delta\leq \delta_2<1$, then there exists a smooth fiber bundle map $\Phi: B_{5/2}(p)\rightarrow B_{5/2}(f(p))\subset\dR^k$ such that $\Phi^{-1}(B_{9/4}(0^k))$ is diffeomorphic to $N^{n-k}\times B_{9/4}(0^k)$ and $\Phi$ is a $\Psi(\delta|n)$-Gromov-Hausdorff map.

\vspace{0.5cm}

Let 
\begin{equation}
d_{GH}(B_3(p),B_3(0^k))<\delta, \ 0^k\in\dR^k
\end{equation}
where  $\delta>0$ is determined later. Define  
\begin{equation}r(n,\delta)\equiv\frac{s_0(n)\delta^{1/2}}{10}<<1,\end{equation}
where $s_0(n)$ is the constant in inequality (\ref{s0-harmonic-radius-bound}) then for each $q\in B_{5/2}(p)$, we have a local fiber bundle $\Phi_q: B_{r}(q)\rightarrow  B_{r}(f(q))$.

 Choose a finite $\frac{r}{2}$-dense subset $\{q_{\ell}\}_{\ell=1}^N\subset B_1(p)$ such that 
\begin{equation}
d(q_{\alpha},q_{\beta})>r/4,\ \forall \alpha\neq\beta.
\end{equation}
 By relative volume comparison, $N\leq N_0(\delta^{-1}|
n)$ and the number of the overlaps of the cover $\{B_{r}(q_{\ell})\}_{\ell=1}^N$ are uniformly bounded by some constant $Q_0(n)$.

Since we always assume $\Ric_g\geq-(n-1)$, there is a good cut-off function due to Cheeger-Colding. We briefly recall the construction here. For each $q_{\ell}$, $1\leq\ell\leq N$. First, we define the following functions,
\begin{equation}
\begin{cases}
\triangle f_{\ell}(x)=1, \ x\in B_{r}(q_{\ell})\setminus B_{r/2}(q_\ell) \\
f_{\ell} \ \text{is constant on the boundary and bounded by} \ C(r,n),
\end{cases}
\end{equation}
and the details of the above construction can be found in \cite{ChC}.
Take $\Psi:\dR^1\rightarrow [0,1]$ as a smooth cut-off function such that the composed function
$\varphi_{\ell}\equiv\Psi(f_{\ell})$ can be extended to a global smooth function on $M^n$ with \begin{equation}
\varphi_{\ell}=\begin{cases}
1, \ x\in B_{r/2}(q_{\ell})\\
0, \ x\in M^n\setminus B_{r}(q_{\ell}).
\end{cases}
\end{equation}
It is standard to construct the partition of unity yielding to the cover $\{B_{r}(q_{\ell})\}_{\ell=1}^N$:
\begin{equation}
1\equiv\sum\limits_{\ell=1}^N\Psi_{\ell}(x),\ x\in B_1(p),\  B_{r/2}(q_{\ell})\subset\Supp(\Psi_{\ell})\subset B_{r}(q_{\ell}),
\end{equation}
where 
\begin{equation}
\Psi_{\ell}\equiv\frac{\varphi_{\ell}}{\sum\limits_{\ell=1}^N\varphi_{\ell}}.
\end{equation}
Now we define
\begin{equation}
\Phi=\Big(\Phi^{(1)},\ldots,\Phi^{(k)}\Big)\end{equation} by
\begin{equation}
\Phi^{(\mu)}\equiv\sum\limits_{\ell=1}^N \Psi_{\ell}\cdot\Phi_{\ell}^{(\mu)}, \ 1\leq \mu\leq k,
\end{equation}
where $\Phi_{\ell}^{(\mu)}$ is the coordinate function of the local fiber bundle map in terms of the original metric $g$.
The remaining is to prove that $\Phi$ is a non-degenerate for each
$q\in B_{5/2}(p)$, and thus $\Phi$
is a fiber bundle map. Moreover, we will show that the  level set  is an infra-nilmanifold of dimension $n-k$.

For any $q\in B_{5/2}(p)$, there are $N'\leq Q_0(n)$ balls in $\{B_{r}(q_{\ell})\}_{\ell=1}^{N}$ which intersects with $B_{r/3}(q)$. In fact, 
\begin{equation}
\Phi^{(\mu)}(x)=\sum\limits_{\alpha=\ell_1}^{\ell_{N'}}\Psi_{\alpha}(x)\cdot\Phi_{\alpha}^{(\mu)}(x),\ 1\leq \mu\leq k,\ \forall x\in B_{r/3}(q). 
\end{equation}
It suffices to argue in terms of the rescaled metric 
$h\equiv{\delta^{-1}}g$. Denote $r_0(n)\equiv s_0(n)/10$.
Notice that, by the same argument as in Step 1, the pointwise $C^{2,\alpha}$-estimate in terms of the rescaled metric $h$ holds for the cut-off function
$\varphi_{\ell}$, 
\begin{equation}
|\varphi_{\ell}|_{C^{2,\alpha}(B_{r_0}(q_{\ell},h))}\leq C_1(n).
\end{equation}
Since the number of overlaps is bounded by $ Q_0(n)$,
\begin{equation}
|\Psi_{\ell,h}|_{C^{2,\alpha}(B_{r_0}(q_{\ell},h))}\leq C_2(Q_0(n),n).
\end{equation}
Quick computation shows that 
\begin{equation}
|\nabla^2\Phi_h^{{(\mu)}}|_{C^0(B_{r_0/3}(q,h))}\leq 
\sum\limits_{\alpha=\ell_1}^{\ell_{N'}}\Big|\nabla^2(\Psi_{\alpha,h}\cdot\Phi_{\alpha,h}^{(\mu)})\Big|_{C^0(B_{r_0}(q_{\alpha},h))}\leq C(n).\label{global-function-pointwise-C^2}
\end{equation}

The next we will prove that
\begin{equation}
|\langle\nabla\Phi_h^{{(\mu)}},\nabla\Phi_h^{{(\nu)}}\rangle-\delta_{\mu\nu}|_{C^0(B_{r_0/6}(q,h))}<\Psi(\delta|n),\ 1\leq \mu,\nu\leq k\label{global-function-pointwise-C^1}
\end{equation}
which implies that
$\Phi_h$ is non-degenerate for sufficiently small $\delta$. Since $\Phi_{\ell,h}(q_{\ell})$ gives a $\Psi(\delta|n)$-Gromov-Hausdorff map for $1\leq\ell\leq N$,
then
for any $x\in B_{r_0/3}(q)\subset B_{r_0}(q_{\alpha})\cap B_{r_0}(q_{\beta})$,
\begin{equation}
|\Phi_{\alpha,h}^{(\mu)}(x)-\Phi_{\beta,h}^{(\mu)}(x)|<\Psi(\delta|n),\ 1\leq \mu\leq k, \ \ell_1\leq \alpha,\beta\leq\ell_{N'}.
\end{equation}
It is clear that
\begin{equation}
\triangle_{h}(\Phi_{\alpha,h}^{(\mu)}-\Phi_{\beta,h}^{(\mu)})(x)=0, \ x\in B_{r_0/3}(q).
\end{equation}
By Cheng-Yau's gradient estimate for harmonic functions in the context of uniform lower Ricci curvature, it holds that
\begin{equation}
|\nabla\Phi_{\alpha,h}^{(\mu)}-\nabla\Phi_{\beta,h}^{(\mu)}|(x)\leq C(n)|\Phi_{\alpha,h}^{(\mu)}-\Phi_{\beta,h}^{(\mu)}|(x), \forall x\in B_{r_0/6}(q).
\end{equation}
The pointwise $C^1$-estimate (\ref{global-function-pointwise-C^1}) immediately follows. In fact, fix any $\ell_1\leq \alpha_0\leq\ell_{N'}$ and $x\in B_{r_0/6}(q)$, it holds that
\begin{align}
&|\langle\nabla\Phi_h^{{((\mu))}},\nabla\Phi_h^{{(\nu)}}\rangle
-\langle\nabla\Phi_{\alpha_0,h}^{{(\mu)}},\nabla\Phi_{\alpha_0,h}^{{(\nu)}}\rangle|(x)\nonumber\\
=&\Big|\sum\limits_{\alpha,\beta=\ell_1}^{\ell_{N'}}\Big\langle\nabla\Big(\Psi_{\alpha,h}\cdot(\Phi_{\alpha,h}^{{(\mu)}}
-\Phi_{\alpha_0,h}^{{(\mu)}})\Big),\nabla\Big(\Psi_{\beta,h}\cdot(\Phi_{\beta,h}^{{(\nu)}}-\Phi_{\alpha_0,h}^{{(\nu)}})\Big)\Big\rangle\Big|(x)\nonumber\\
\leq &\sum\limits_{\alpha,\beta=\ell_1}^{\ell_{N'}}\Big(|\nabla\Psi_{\alpha,h}|\cdot|\Phi_{\alpha,h}^{{(\mu)}}
-\Phi_{\alpha_0,h}^{{(\mu)}}|+|\Psi_{\alpha,h}|\cdot|\nabla(\Phi_{\alpha,h}^{{(\mu)}}
-\Phi_{\alpha_0,h}^{{(\mu)}})|\Big)(x)\cdot\nonumber\\
&\Big(|\nabla\Psi_{\beta,h}|\cdot|\Phi_{\beta,h}^{{(\nu)}}
-\Phi_{\alpha_0,h}^{{(\nu)}}|+|\Psi_{\beta,h}^{(\nu)}|\cdot|\nabla(\Phi_{\beta,h}^{{(\nu)}}
-\Phi_{\alpha_0,h}^{{(\nu)}})\Big|\Big)(x)\nonumber\\
\leq &\Psi(\delta|n). 
\end{align}
By triangle inequality, for every $x\in B_{r_0/6}(q)$,
\begin{eqnarray}
|\langle\nabla\Phi_h^{{(\mu)}},\nabla\Phi_h^{{(\nu)}}\rangle-\delta_{ij}|&\leq& |\langle\nabla\Phi_h^{{(\mu)}},\nabla\Phi_h^{{(\nu)}}\rangle
-\langle\nabla\Phi_{\alpha_0,h}^{{(\mu)}},\nabla\Phi_{\alpha_0,h}^{{(\nu)}}\rangle|+|\langle\nabla\Phi_{\alpha_0,h}^{{(\mu)}},\nabla\Phi_{\alpha_0,h}^{{(\nu)}}\rangle-\delta_{ij}|\nonumber\\
&\leq &\Psi(\delta|n).
\end{eqnarray}

Finally, notice that $B_{3}(0^k)$
is contractible and thus 
the bundle map $\Phi$ is in fact trivial.

\vspace{0.5cm}

{\bf{Step 3.}} There exists $0<\delta_3(n,c_0)<1$ such that if assumptions $(i)$, $(ii)$ hold for $0<\delta\leq\delta_3<1$, then the fiber $N^{n-k}\cong\Phi^{-1}(0^k)$ in Step 2 is diffeomorphic to an infra-nilmanifold of dimension $n-k$. 
\vspace{0.5cm} 

It suffices to argue under the rescaled metric, that is, fix some $0<\delta\leq\delta_2<1$ and let $h\equiv\delta^{-1}g$. Recall that, under the rescaled metric, the arguments in Step 2 show that the harmonic fiber bundle map $\Phi_h$ satisfies the pointwise estiamtes (\ref{global-function-pointwise-C^2}), (\ref{global-function-pointwise-C^1}). Since $|\Ric_h|\leq (n-1)\delta$ and 
$\Conj_h(q)\geq c_0\cdot\delta^{-1}>2\cdot 10^3$ in our context,  
then by the smoothing theorem due to Dai-Wei-Ye (\cite{Dai-Wei-Ye}), for each fixed small $0<\epsilon<10^{-2}$, there exists some metric with bounded curvature which satisfies,
\begin{equation}
e^{-\epsilon}h\leq h_{\epsilon}\leq e^{\epsilon}h,\ |\nabla^k h_{ij}|_{B_1(q, h_{\epsilon})}<C_k(\epsilon,n), \ \forall k\in\mathbb{N}_+
\end{equation}
and $\Conj_{h_{\epsilon}}(q)\geq 10^3$. So applying the same scheme as in Step 2 (by the exponential map), there exists $s_0(n)/2<s_1(n,\epsilon)\leq s_0(n)$ such that under the smoothing metric $h_{\epsilon}$, it holds that
\begin{equation}
|\nabla^2\Phi_{h}^{(\mu)}|_{C^0(B_{s_1}(q,h_{\epsilon})}+[\nabla^2\Phi_{h}^{(\mu)}]_{C^{\alpha}(B_{s_1}(q,h_{\epsilon}))}<C(n,\epsilon),
\end{equation}
and 
\begin{equation}
|\langle \nabla\Phi_{h}^{(\mu)}, \nabla\Phi_{h}^{(\nu)}\rangle_{g_{\epsilon}}-\delta_{ab}|
<\Psi(\delta|n,\epsilon)<<1,\ \forall x\in B_{s_1}(q,h_{\epsilon}).
\end{equation}
then $\Phi_{h}$ is non-degenerate with respect to the smoothing metric 
$h_{\epsilon}$. Moreover, 
we have the uniform bound on second fundamental form of the fiber at $q$,
\begin{equation}
|\II_{h_{\epsilon}}|\leq C(n)\frac{|\nabla^2\Phi_{h}^{(\mu)}|_{h_{\epsilon}}}{|\nabla\Phi_{h}^{(\mu)}|_{h_{\epsilon}}}\leq C(\epsilon,n).
\end{equation}
 Therefore, by Gauss-Codazzi equation, the sectional curvatures of the fibers are uniformly bounded by some constant $C(\epsilon,n)$. Since $\Phi_{h}$ a $\Psi(\delta|n)$-Gromov-Hausdorff map with respect to the metric $h$, the diameter (under $h_{\epsilon}$) of the fiber $E_q\equiv\Phi^{-1}_{h}(\Phi_{h}(q))$ satisfies
 \begin{equation}\diam_{h_{\epsilon}}(E_q)\leq e^{\epsilon}\cdot\Psi(\delta|n).\end{equation}
From now on, fix $\epsilon=1/10^6$, Gromov's theorem on the almost flat manifolds (combined with a refinement by Ruh, see \cite{Gromov} and \cite{ruh}) gives that there exists $\epsilon_{Gr}(n)>0$ such that if
\begin{equation}
\diam_{h_{\epsilon}}^2(E_q)\leq \epsilon_{Gr}(n),\label{almost-flat}
\end{equation}
then the fiber $E_q$ is diffeomorphic to an infra-nilmanifold of dimension $n-k$. 
In fact, 
 \begin{equation}
\diam_{h_{\epsilon}}^2(E_q)\cdot |\sec_{h_{\epsilon}}(E_q)|\leq e^{2\epsilon}\cdot\Psi^2(\delta|n)\cdot C(n,\epsilon),\end{equation}
 so
inequality (\ref{almost-flat}) immediately holds if we choose $\delta>0$
sufficiently small (depending on $n$ and $c_0$).
The proof is complete.

\end{proof}

\end{document}